\documentclass{article}
\usepackage{amsthm}
\usepackage{amsfonts, amssymb, amsmath, amscd, latexsym}
\usepackage[latin1]{inputenc}
\usepackage[dvipsnames]{xcolor}
\usepackage{verbatim,cite}
\usepackage[all]{xy}
\usepackage{graphicx,epstopdf}
\def\span{\hbox{\rm span}}
\def\dis{\displaystyle}
\def\al{\alpha}
\def\be{\beta}
\def\ga{\gamma}

\def\de{\delta}
\def\ep{\varepsilon}
\def\la{\lambda}

\def\th{\theta}
\def\si{\sigma}

\def\sinh{{\rm sinh}}
\def\cosh{{\rm cosh}}

\def\Om{\Omega}

\def\I{\mathbb I}
\def\R{\mathbb R}

\def\I{\mathbb I}

\def\({\left(}
\def\){\right)}
\def\<{\langle}
\def\>{\rangle}

\DeclareMathOperator{\rank}{rank}

\newtheorem{theorem}{Theorem}[section]
\newtheorem{lemma}[theorem]{Lemma}

\newtheorem{prop}[theorem]{Proposition}
\newtheorem{definition}[theorem]{Definition}
\theoremstyle{remark}
\newtheorem{remark}[theorem]{Remark}
\numberwithin{equation}{section}
\begin{document}
\title{Heteroclinic Solutions in Singularly Perturbed Discontinuous Differential Equations}
\date{\today}
\author{Flaviano Battelli	\\ Department of Industrial Engineering and Mathematics \\  Marche Polytecnic University    \\  Ancona - Italy
\and Michal Fe\v ckan\thanks{Partially supported by the National Natural Science Foundation of China (12161015), the Slovak Research and Development Agency under the contract No. APVV-18-0308 and by the Slovak Grant Agency VEGA No. 1/0084/23 and No. 2/0127/20} \\  Department of Mathematical Analysis and Numerical Analysis   \\  Comenius University in Bratislava   \\ Mlynsk\' a dolina, 842 48 Bratislava -- Slovakia     \\  and Institute of Mathematics, Slovak Academy of Sciences  \\   \v{S}tef\'anikova 49, 814 73 Bratislava -- Slovakia \and JinRong Wang \\ Department of Mathematics \\ Guizhou University, Guiyang  \\ Guizhou 550025 -- China}
\maketitle
\begin{abstract}
We derive Melnikov type conditions for the persistence of heteroclinic solutions in perturbed slowly varying discontinuous differential equations  extending to these equations similar results for continuous differential equations.
\end{abstract}

\section{Introduction}
In recent years there has been a great deal in the study of singularly perturbed equations such as
\begin{equation}\label{cont-sin}\begin{array}{l}
	\dot x = f(x,y)	\\
	\dot y = \ep g(x,y).
\end{array}\end{equation}
where $f(x,y)$ and $g(x,y)$ are sufficiently smooth functions. To the best of our knowledge this study started with remarkable papers by Tichonov \cite{T}, Vasil'eva et al. \cite{VB,VBK}. All these results concern the behaviour of the solutions of \eqref{cont-sin} on a finite time scale. Later, Hoppensteadt \cite{Hop1,Hop2} extended the result on an infinite time scale. Then Fenichel \cite{Fen} developed his geometric singularly perturbed theory that has been widely used by other authors as, for example, Szmolyan \cite{Sz}. These last results essentially prove that the $x$ component of the solution is asymptotic to invariant manifolds, sometimes called {\em centre manifolds}, described by equations like $x=X_\pm(y)$ (the case $X_+(y)=X_-(y)$ is allowed). In \cite{Kov} a combination of the Melnikov method and geometric singular perturbation theory is presented. This approach has been then extended in \cite{Ha} to investigate a mechanism of chaos near resonances both in the dissipative and the Hamiltonian context of differential equations.

The basic starting point for this kind of results is that, for all $y$ the frozen equation
\begin{equation}\label{frz}
	\dot x = f(x,y)
\end{equation}
has hyperbolic fixed points, say $u_\pm(y)$, that are bounded functions of $y$ together with their derivatives, and solutions $u_\pm(t,y)$, defined for $t\ge 0$ and $t\le 0$, resp., such that
\begin{equation}\label{asynt}
	\lim_{t\to\infty}u_+(t,y)-u_+(y) = 0, \; \lim_{t\to -\infty}u_-(t,y)-u_-(y) = 0
\end{equation}
uniformly with respect to $y$ and there exists $y_0$ such that $u_-(0,y_0) = u_+(0,y_0)$. Hence, for $y=y_0$ the frozen equation \eqref{frz} has a heteroclinic orbit connecting $u_-(y_0)$ to $u_+(y_0)$. Note that in \cite{B,BL,P2} it is also assumed that $u_\pm(t,y)$ are bounded together with their derivatives. However, in Lemma \ref{Onu+}, we prove that this condition is automatically satisfied when $u_\pm(y)$ are bounded together with their derivatives and \eqref{asynt} holds. The proof of this result is rather technical and may be skipped at a first reading of this paper.

Other relevant results concerning \eqref{cont-sin} appeared in \cite{B,BL,BP,KL,P2}. In \cite{KL}, the second order equation with slowly varying coefficients
\begin{equation}\label{KL}
	\ddot x + x(x-a(\ep t+\theta))(1-x) = 0, \quad \ep>0
\end{equation}
is studied, where $a(y)$ is a $C^1$, $1$-periodic function, such that  $0<a(y)<1$. When $\ep=0$ and $a(y)=\frac{1}{2}$, \eqref{KL} has the heteroclinic solution
\[
	x_{het}(t) = \frac{1}{1+e^{-\frac{t}{\sqrt{2}}}}.
\]
However this heteroclinic solution is broken when $a(y)\ne \frac{1}{2}$. In \cite{KL} it has been proved that if $a(y)$ has a transversal intersection with $a=\frac{1}{2}$ at $y=y_0$, i.e. $2a(y_0)=1$ and $a'(y_0)\ne 0$, then equation \eqref{KL} has a heteroclinic solution. Writing
\[
	x_1=x, \quad x_2=\dot x,	\quad y=\ep t+\al
\]
equation \eqref{KL} reads:
\begin{equation}\label{KLsys}\begin{array}{l}
	\dot x = f(x,y)		\\
	\dot y =\ep	
\end{array}\end{equation}
where
\[
	f(x,y) = \begin{pmatrix}
		x_2 \\  x_1(x_1-1)(x_1-a(y))
	\end{pmatrix}
\]
and $u(t,y_0)=(x_{het}(t),\dot x_{het}(t)$. This solution is a part of a family of solutions $u_\pm(t,y)$ such that
\[
	\lim_{t\to-\infty}u_-(t,y) = (0,0) \quad \lim_{t\to\infty}u_+(t,y) = (1,0)
\]
and it can be proved that the transversality condition on $a(y)$ corresponds to the fact that the real-valued function
\[
	\psi^*[u_-(0,y) - u_+(0,y)]
\]
where $\psi\in\R^n$ is a suitable vector, has a simple zero at $y=y_0$.

In Theorem \ref{main}, we extend these results, as far as the asymptotic behaviour of the $x$-component of the solution is concerned, to discontinuous differential equations that is a differential equation which is described in different ways according to the region the solutions belong to. The first difference with the continuous case is that the way how $u_\pm(t,y)$ passes from one domain into another is important. In this paper we assume that $u_\pm(t,y)$ intersects transversally the boundary of the domain when they pass from the domain into another (see assumption $A_1)$ in Section 2).

Problems like the above where $u(t,\al,y)$ is, instead, a family of periodic solutions depending on some further parameters $\al\in\R$ and $y\in\R^m$, have been studied in \cite{WH,BF2,BF3}. Note that, in \cite{WH}, $f(x,y)$ is replaced by $f_0(x,y)+\ep f_1(x,y,t,\ep)$ and it is assumed that system
\begin{equation}\label{unpert}
	\dot x = f_0(x,y)
\end{equation}
has a one-parameter family of periodic solutions $q(t-\theta,y,\alpha)$ with period $T(y,\alpha)$ being $C^r$ in $(y,\alpha)$ and it is because of the
$t$-dependence of the perturbed equation that the extra variable $\theta$ has been introduced. Then a vector valued function $M^{m/n}(y,\al,\theta)$ is constructed, that they called \emph{subharmonic Melnikov function}, which is a measure of the difference between the starting value and the value of the solution at the time $\frac{m}{n}T$ in a direction transverse to the unperturbed vector field at the starting point and proved that periodic solutions of the perturbed vector field arise near the simple zeros of $M^{m/n}(y,\al,\theta)$. In \cite{BF3} this result concerning the existence of periodic solutions has been extended to discontinuous systems of differential equations as the one we consider here.

Now, let us be more precise and give the definition of the discontinuous differential equation we are studying in this paper.

Let $h(x,y):\R^n\times\R^m\to \R$ be a $C^r$ function, $r\ge 2$, with bounded derivatives, $f_\ell(x,y):\R^n\times\R^m\to \R$,
$\ell = -M,\ldots,N$, be $C^r$-functions, bounded on $\R^n\times\R^m$ together with their derivatives, and $c_{-M}<\ldots<c_0\ldots<c_N$
be real numbers. By discontinuous differential equation we mean an equation like
\begin{equation}\label{pert}
    \begin{array}{l}
        \dot x = f(x,y)	\\
        \dot y = \ep g(x,y,\ep)
    \end{array}
\end{equation}
where $x\in\R^n$, $y\in\R^m$, $\ep\in\R$, $\ep>0$ and
\begin{equation}\label{eff}
	f(x,y) := \left \{\begin{array}{ll}
		f_{-M}(x,y) & \hbox{if $h(x,y)<c_{-M}$}		\\
		f_\ell(x,y) & \hbox{if $c_{\ell-1}<h(x,y)<c_\ell$}		\\
			&	\ell =-M+1,\ldots,N		%			\\
%		f_N(x,y) & \hbox{if $h(x,y)>c_N$}
\end{array}\right .
\end{equation}
It is assumed that the equation
\begin{equation}\label{unpert0}
	\dot x = f_N(x,y) \quad \hbox{[resp. $\dot x = f_{-M}(x,y)$]}
\end{equation}
has a hyperbolic fixed point $x=w_+(y)$ [resp. $x=w_-(y)$] together with continuous, piecewise $C^2$ solutions, $u_+(t,y)$, $t\ge 0$,
$u_-(t,y)$, $t\le 0$, such that
\[
	\lim_{t\to\pm\infty}|u_\pm(t,y) - w_\pm(y)| \to 0
\]
uniformly with respect to $y$, and
\begin{equation}\label{trsv}
	h(u_\pm(t_\ell,y),y)=c_\ell \Rightarrow \inf_{y\in\R^m}|h_x(u_\pm(t_\ell,y),y) \dot u_\pm(t_\ell^\pm ,y)|>0.
\end{equation}
for any $-M\le \ell\le N$, $\ell\ne0$.

Here we used the shorthand $u_\pm(t,y) - w_\pm(y)$ for $u_+(t,y) - w_+(y)$, when $t\ge 0$, and $u_-(t,y) - w_-(y)$, when $t\le 0$. We will use such a shorthand throughout the whole paper.

%In the following we will refer to this last property saying
Note that \eqref{trsv} implies that $u_\pm(t,y)$ intersects transversally the set
\[
	{\cal S}_\ell(y) =  \{ x\in\R^n \mid h(x,y)=c_\ell\},
\]
at $u_\pm(t_\ell,y)$. Then we look for solutions $(x(t,\ep),y(t,\ep))$ of \eqref{pert} such that
\begin{equation}\label{close}
	\sup_{t\in\R_\pm}|x(t,\ep) - u_\pm(t,y(t,\ep))| \ll 1.
\end{equation}
Note that, from \eqref{trsv} it follows that, for any $y\in\R^m$, the set ${\cal S}_\ell(y)$, $\ell=-M,\ldots,N$, $\ell\ne 0$, is an hypersurface in a ball around $u_\pm(t_\ell,y)$ whose radius does not depend on $y$. We emphasize that \eqref{trsv} is all that we need on ${\cal S}_\ell(y)$ for our analysis.
\medskip

To prove the existence of solutions of \eqref{pert} satisfying \eqref{close} we use Lyapunov-Schimtd method, together with a combination of singularly perturbed analysis and a technique for discontinuous dynamical systems, to construct a {\em bifurcation function} whose zeros are associated to solutions of the perturbed equation whose $x$-component satisfies \eqref{close}. This is the content of Theorem \ref{main} where we prove that if a certain generic condition is satisfied then there is a manifold of such solutions.

According to $A_3)$ (see Section 2) $x = w_\pm(y)$ are normally hyperbolic invariant manifolds for the unperturbed system $\dot x = f(x,y)$, $\dot y = 0$. From \cite{BF-1,Sa} these manifolds perturb to invariant manifolds $x=\tilde w_\pm(y,\ep)$ such that
\[
	\sup_{y\in\R^m} |\tilde w_\pm(y,\ep) - w_\pm(y)| \to 0, \hbox{as $\ep\to 0$.}
\]
From Remark \ref{remrem} it follows that $x(t,\ep)$ is asymptotic to the manifold $\tilde w_+(y,\ep)$ as $t\to\infty$ and to $\tilde w_-(y,\ep)$ as $t\to-\infty$. Hence, $(x(t,\ep),y(t,\ep))$ behaves as a heteroclinic solution of equation \eqref{pert} connecting the invariant manifold
$\tilde w_-(y,\ep)$ to $\tilde w_+(y,\ep)$.

The problem studied in this paper has been motivated by \cite{WH2}, where existence and bifurcation theorems are derived for homoclinic orbits in three-dimensional flows that are perturbations of families of planar Hamiltonian system. In this paper we study the problem of persistence of bounded solutions in the discontinuous case \eqref{pert} where $x\in\R^n$, $y\in\R^m$, all functions considered are sufficiently smooth and $\ep\in\R$ is a small parameter, assuming the existence of such an orbit in the unperturbed equation \eqref{unpert0}.

Then in Section 7 we apply Theorem \ref{main} to extend \cite[Theorem 1]{KL} to the discontinuous equation \eqref{pert}. Following \cite{B} we derive a bifurcation function characterizing the persistence of homoclinic solutions of \eqref{pert} from a generic homoclinic solution of the unperturbed system \eqref{unpert}. It can be easily checked that the results of this paper easily extend if we replace $f(x,y)$ with $f(x,y,\ep)=f_0(x,y)+\ep f_1(x,y,\ep)$ (with the same  $h(x,y)$). In this case the unperturbed system will be
\[
\dot x = f_0(x,y), \; \dot y = 0.
\]

The next step is the study of a degenerate case where, for any $y\in\R^m$ the unperturbed discontinuous equation \eqref{unpert} has a piecewise $C^1$ solution $u(t,y)$ heteroclinic to the hyperbolic fixed points $x=w_\pm(y)$.  We plan to perform this study in a forthcoming paper as it is necessary to go into a deeper analysis of the bifurcation function.

\medskip

We now briefly sketch the content of this paper. In Section 2 we provide basic assumptions and define the piecewise smooth heteroclinic solution of the unperturbed system. In Section 3 we recall the definition of exponential dichotomy and extend this notion to discontinuous, piecewise linear, systems with a jump at some points. We also extend to these systems some results concerning existence of bounded solutions on either $t\ge 0$ and $t\le 0$. In our opinion, these results are theirselves interesting as they give the form of the projection of the dichotomy of a linear discontinuous system. Although this is an important point in the proof of Theorem \ref{main}, we think that the proofs in this section can be skipped at a  first reading.

Next, in Section 4, we construct families of bounded solutions and describe them in terms of some parameters. These solutions are continuous and piecewise smooth and give the bounded solutions we look for, when they assume the same value at $t=0$. Then, after having defined  the variational equation in Section 5, in Section 6 we study the \emph{joining condition at $t=0$} which is the bifurcation condition and give a Melnikov-type condition assuring that the bifurcation equation has a manifold of solutions. Finally, in Section 7 we first state some general facts concerning two-dimensional discontinuous equations  depending on a slowly varying parameter and give an example of application of the main result of this paper. In Section 8 we give a hint on a possible extension of the example.
\medskip

In the whole paper we will use the following notation. Given a vector $v$ or a matrix $A$ with $v^T$, (resp. $A^T$) we denote the transpose of
$v$ (resp. $A$).

\section{Notation and basic assumptions}
Let $\Om\subset\R^n$ be a bounded domain,
\[
c_{-M}<\ldots<c_0<\ldots<c_{N}
\]
be real numbers and $h:\Om\times\R^m\to\R$ be a $C^r$-functions, $r\ge 2$, with bounded derivatives. For $\ell=-M,\ldots,N$, we set
\[
	\Om_\ell =\{(x,y)\in\Om\times\R^m \mid c_{i-1}\le h(x,y)<c_i\},
\]
where we set for simplicity, $c_{-M-1}=-\infty$ and let $f_\ell:\Om\times\R^m\to\R^n$ be $C^r$-function, bounded together with their derivatives in $\Om\times\R^m$. We are looking for solutions of equation
\begin{equation}\label{disc}
	\dot x = f_\ell(x,y)	, \quad (x,y)\in \Om_\ell
\end{equation}
which are contained in a compact subset $H$ of $\Om$. Hence it is not restrictive to assume that $\Om=\R^n$. So, from now on, we suppose $\Om=\R^n$.
\medskip

First we give the definition of solutions of equation \eqref{disc} we are considering in this paper
\begin{definition} A continuous, piecewise smooth function $u(t,y)$ is a solution of equation \eqref{disc} on $t\ge 0$ intersecting transversally the sets ${\cal S}_i(y) = \{x\in\Om\mid h(x,y)=c_i\}$, $i=0,\ldots,N-1$, if there exist $\eta>0$ and $0<t_1(y)<\ldots <t_{N}(y)$ such that the following conditions hold for $1\le i\le N-1$ (note that we set $t_0(y)=0$)
\begin{itemize}
	\item[$a_1)$] $\dot u(t,y) = f_{i-1}(u(t,y),y)$ for $t_{i-1}(y)<t<t_i(y)$ and $\dot u(t,y) = f_N(u(t,y),y)$ for $t>t_N(y)$;
	\item[$a_2)$] $h(u(t_i(y),y),y) = c_{i-1}, \quad \hbox{and}\quad h_x(u(t_i(y),y),y) \dot u(t_i(y)^\pm,y) >2\eta$;
	\item[$a_3)$] $c_{i-1}<h(u(t,y),y)<c_i$, for $t_i(y)<t<t_{i+1}(y)$ and $h(u(t,y),y)>c_{N-1}$, for $t>t_N(y)$.
\end{itemize}
Similarly, a continuous, piecewise smooth function $u(t,y)$ is a solution of equation \eqref{disc} on $t\le 0$ intersecting transversally the sets
${\cal S}_j(y) = \{x\in\Om\mid h(x,y)=c_j\}$, $j=-1,\ldots,-M$, if there exist $\eta>0$ and $t_{-M}(y)<\ldots <t_{-1}(y)<0$ such that the following conditions hold for any $-M+1\le j\le -1$:
\begin{itemize}
	\item[$a'_1)$] $\dot u(t,y) = f_{j+1}(u(t,y),y)$ for $t_j(y)<t<t_{j+1}(y)$ and $\dot u(t,y) = f_{-M}(u(t,y),y)$ for $t<t_{-M}(y)$;
	\item[$a'_2)$] $h(u(t_j(y),y),y) = c_j, \quad \hbox{and}\quad h_x(u(t_j(y),y),y) \dot u(t_j(y)^\pm,y)>2\eta$;
	\item[$a'_3)$] $c_j<h(u(t,y),y)<c_{j+1}$, for $t_j(y)<t<t_{j+1}(y)$ and $h(u(t,y),y)<c_{-M}$, for $t<t_{-M}(y)$.
\end{itemize}
\end{definition}
In this paper we assume that continuous, piecewise smooth solution $u_+(t,y)$ and $u_-(t,y)$ of equation \eqref{disc} exist, for $t\ge 0$, resp. $t\le 0$, such that the following conditions also hold.
\begin{itemize}
\item[$A_1)$] $w_0^\pm(y) := u_\pm(0,y)$ and their derivatives are bounded functions and belong to an open and bounded subset $B\subset\R^n$ such that $\overline{B}\times\R^m \subset \Om_0$.
\item[$A_2)$] There exist smooth and bounded functions $w_\pm(y)$ and $\mu_0>0$, such that
\[\begin{array}{l}
	f_N(w_+(y),y)=f_{-M}(w_-(y),y)=0	\\
	h(w_+(y),y)-c_{N-1}>\mu_0, \quad h(w_-(y),y)-c_{-M}<-\mu_0
\end{array}\]
for any $y\in\R^m$ and
\[
	\lim_{t\to\pm\infty} u_\pm(t,y) - w_\pm(y) = 0
\]
uniformly with respect to $y\in\R^m$.
\item[$A_3)$]  For any $y\in\R^m$, $f_{N,x}(w_+(y),y)$ and $f_{-M,x}(w_-(y),y)$ have $k$ eigenvalues with negative real parts and $n-k$ eigenvalues with positive real parts, counted with multiplicities and there exists $\de_0>0$ such that all these eigenvalues satisfy
\[
|Re\lambda(y)|>\delta_0.
\]
\item[$A_4)$] There exists $y_0\in\R^m$ such that $w_0^-(y_0) = w_0^+(y_0)=x_0$.
\end{itemize}
\begin{figure}\centering
	\includegraphics[scale = 0.55]{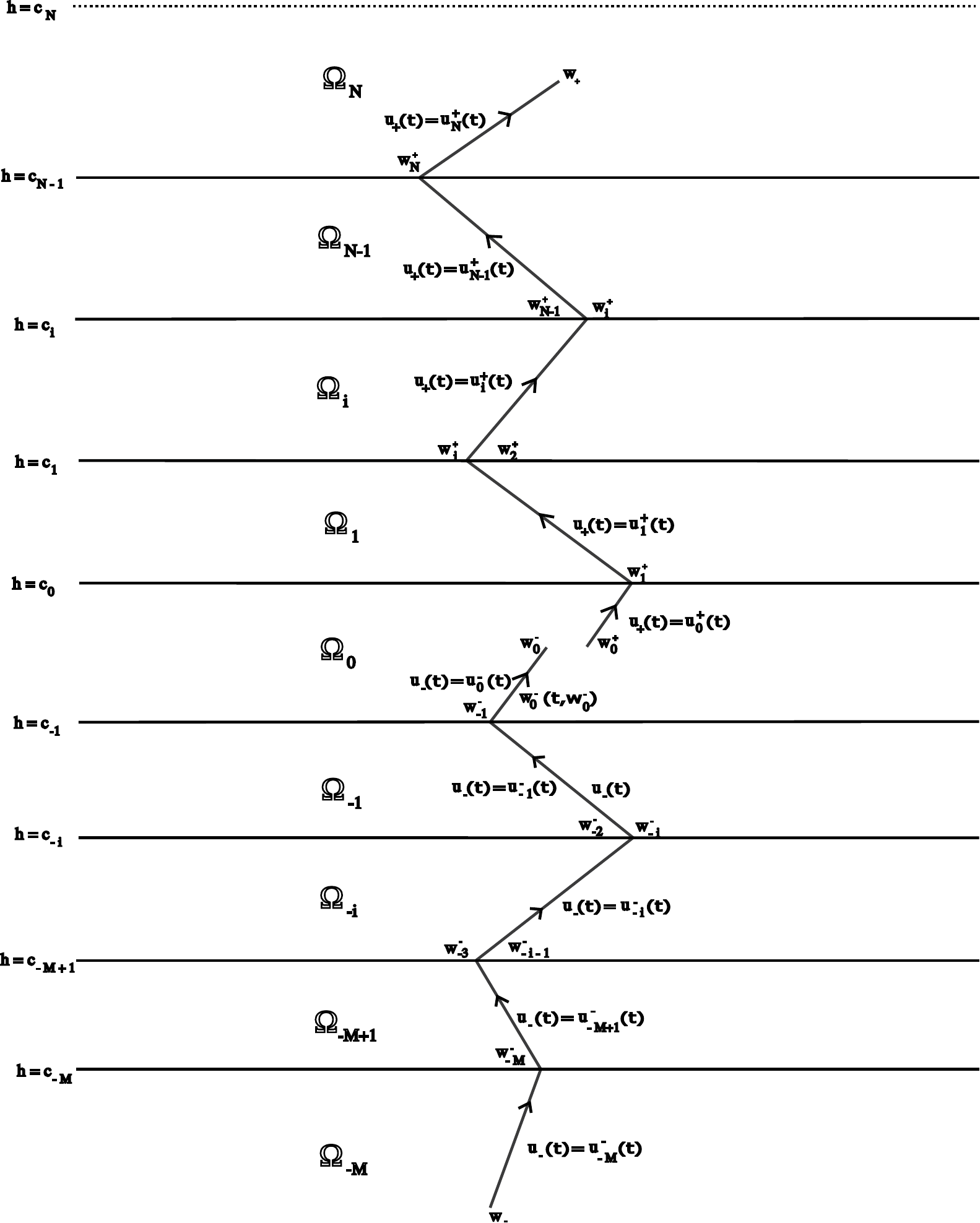}
	\caption{\small The piecewise continuous bounded solution of \eqref{frz}. For simplicity we write $w^\pm_j$ instead of
	$w^\pm_j(y)$. The solutions $u_\pm(t,y)$ go from $\Om_{-M}$ to $\Om_N$ as $t$ goes from $-\infty$ to $\infty$ and may have a
	discontinuity at $t=0$. In this figure we actually draw the  intersection of $\Om_i$ with $y=$constant.\label{fig1}}
\end{figure}
\begin{remark}\label{moregenNew}
From $A_2)-A_4)$ it follows that $u(t,y_0)$ is a continuous, piecewice $C^1$ solution of $\dot x = f(x,y_0)$ such that
\[
\lim_{t\to\pm\infty} u(t,y_0) = w_\pm(y_0).
\]
Moreover,  from $a_2)-a_3)$ it follows that, for $i=1,\ldots,N$:
\[\begin{array}{l}
	\frac{\partial}{\partial t}h(u_+(t,y),y)_{|t=t_i(y)^\pm}\ge 0
\end{array}\]
that is
\[\begin{array}{l}
	h_x(u_+(t_i(y),y) \dot u_+(t_i(y)^\pm,y) \ge 0,
\end{array}\]
Similarly we see that
\[\begin{array}{l}
	h_x(u_-(t_i(y),y) \dot u_-(t_i(y)^\pm,y) \ge 0,
\end{array}\]
So $a_2)-a'_2)$ are a kind of transversality assumption on the solutions $u_\pm(t,y)$.

ii) All results in this paper can be generalised to the case where the solutions exit transversally $\Om_i$ and enter into either $\Om_{i+1}$ or
$\Om_{i-1}$ transversally in the sense that \eqref{trsv} holds at the intersection points of the solution with the boundary of $\Om_i$. More precisely suppose, to fix ideas, $t\ge 0$. Then we assume the following (see Fig.\,\ref{fig2}). There exists $(i_0=0,i_1\ldots,i_M)$ such that given $i_h$ then $i_{h+1}$ is either $i_h-1$ or $i_h+1$ and for $t_h(y)<t<t_{h+1}(y)$ we have
\[
	c_{i_h-1}<h(u_+(t,y),y)<c_{i_h}, \quad \hbox{for any $h=0,\ldots,N-1$}.
\]
Moreover
\[\begin{array}{l}
	|h_x(u_+(t_i(y),y),y) \dot u_+(t_i(y)^\pm,y)| > 2\eta
\end{array}\]
for any $h=1,\ldots,N$.

A similar generalization can be made for $t\le 0$ and all other assumptions are changed accordingly.

Another possible generalization is the homoclinic case. This is the case, for example, where $N=M$ and assumptions $a'_1)-a'_3)$ are changed to the following where we write $i\in \{1,\ldots,N\}$ instead of $-j$ (see Fig.\,\ref{fig3}):
\begin{itemize}
	\item[$a'_1)$] $\dot u(t,y) = f_{i-1}(u(t,y),y)$ for $t_{-i}(y)<t<t_{-(i-1)}(y)$ and $\dot u(t,y) = f_N(u(t,y),y)$ for $t<t_{-N}(y)$;
	\item[$a'_2)$]  $h(u(t_{-i}(y),y),y) = c_{i-1}$ and $h_x(u(t_{-i}(y),y),y) \dot u(t_{-i}(y)^\pm,y)<-2\eta$;
	\item[$a'_3)$] $c_{i-2}<h(u(t,y),y)<c_{i-1}$, $t_{-i}(y)<t<t_{-(i-1)}(y)$ and $h(u(t,y),y)>c_{N-1}$, for $t<t_{-N}(y)$.
\end{itemize}
where, we set for simplicity, $c_{-1}=-\infty$.
\end{remark}
\begin{figure}\centering
	\includegraphics[scale = 0.55]{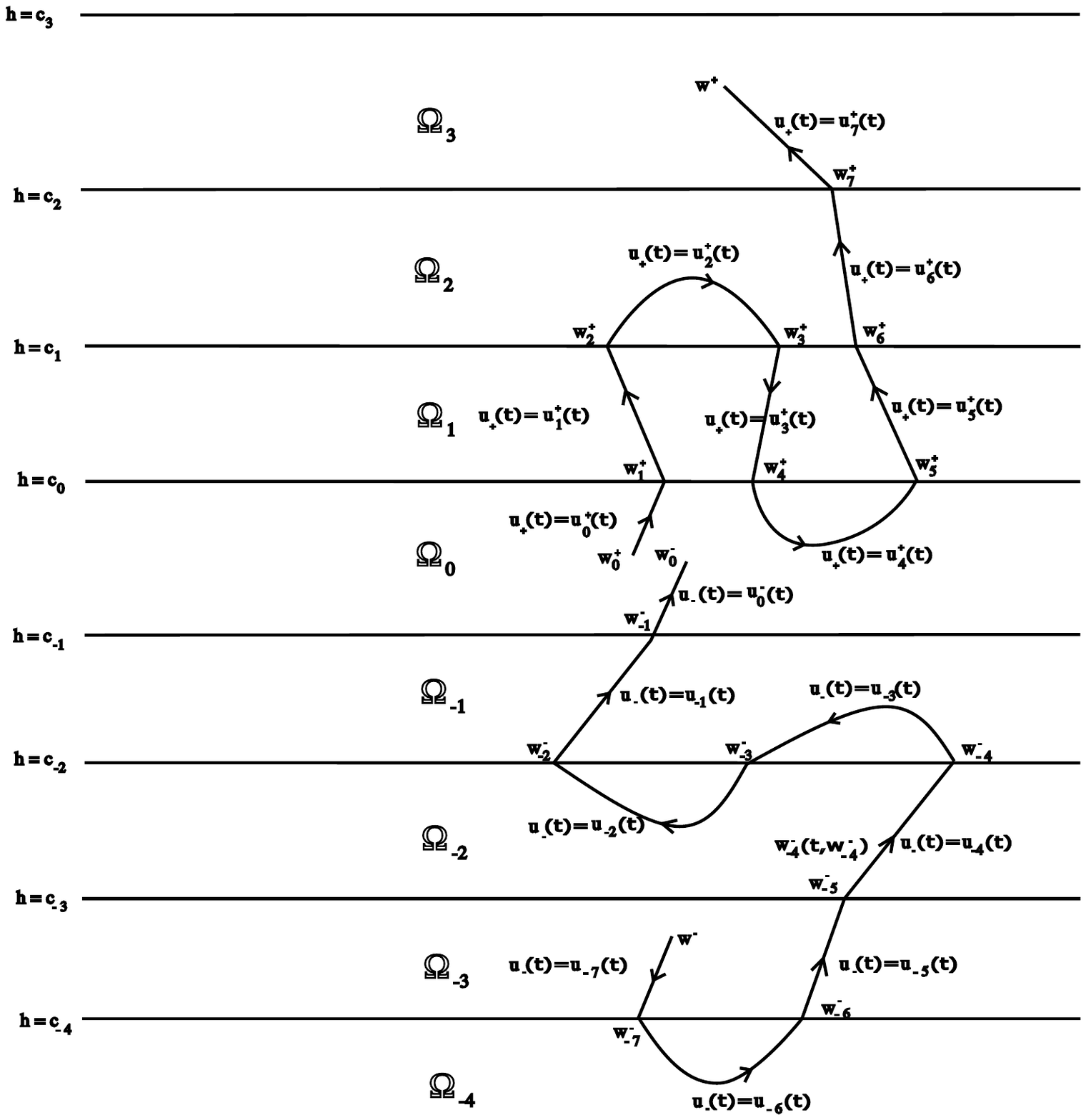}
	\caption{\small An example of the situation described in Remark \ref{moregenNew} with $N=3$ and $M=4$. Here we have
	$(i_{-1},i_{-2},i_{-3},i_{-4},i_{-5},i_{-6},i_{-7})=(-1,-2,-1,-2,-3,-4,-3)$ and $(i_0,i_1,i_2,i_3,i_4,i_5,i_6,i_7)=(0,1,2,1,0,1,2,3)$.
	Again in this figure we draw the  intersection of $\Om_i$ with $y=$constant.\label{fig2}}
\end{figure}
\begin{figure}\centering
	\includegraphics[scale = 0.6]{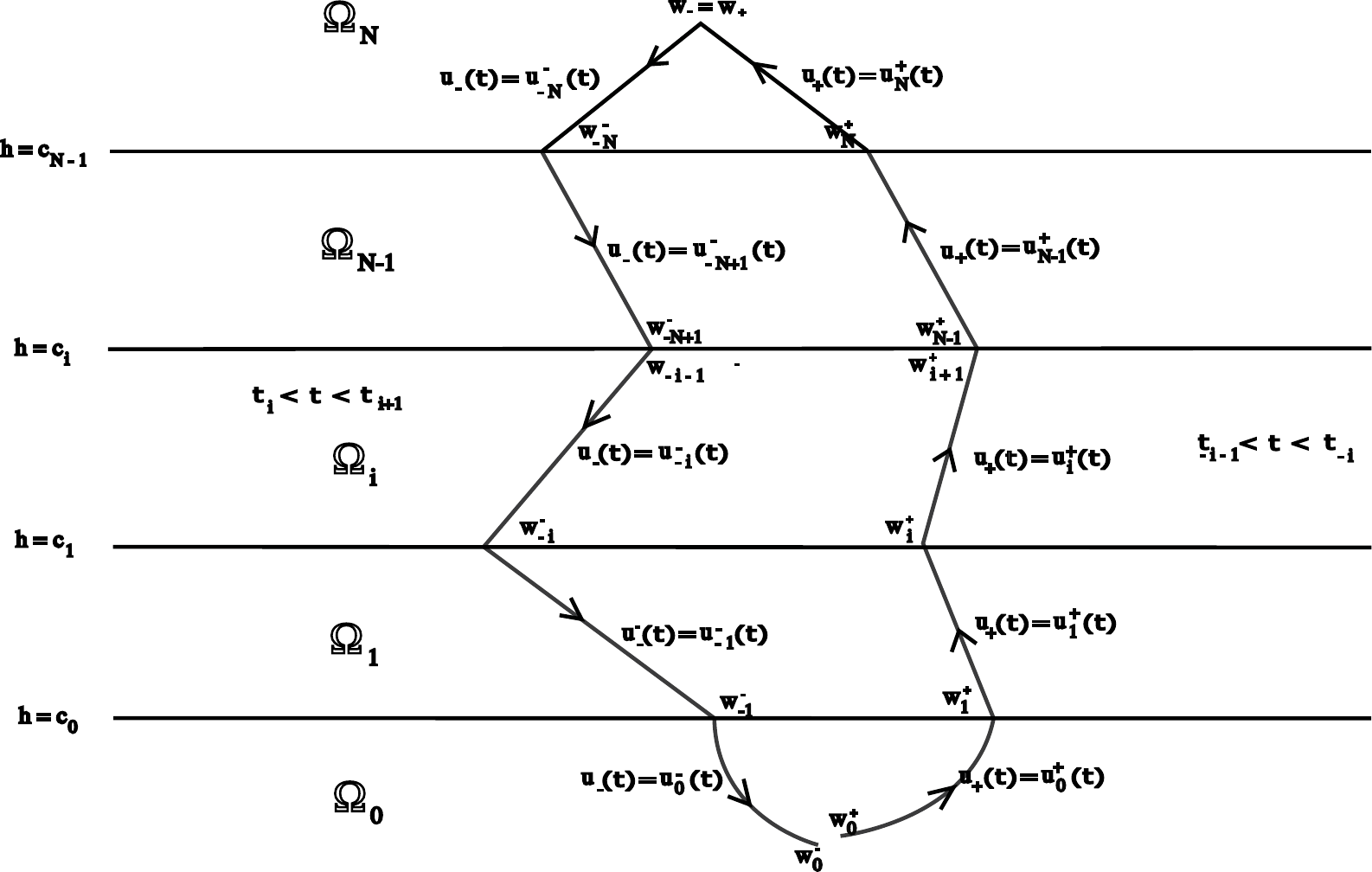}
	\caption{\small An example of homoclinic orbit as described in Remark \ref{moregenNew}. The homoclinic orbit arises at $y=y_0$ when
	$w^+_0(y_0)=w^-_0(y_0)$. Again in this figure we draw the intersection of $\Om_i$ with $y=$constant\label{fig3}}
\end{figure}

We have the following
\begin{lemma}\label{bddw} For $i=0\ldots,N$, $j=0,\ldots,-M$, let $w_i^+(y)=u_+(t_i(y),y)$, $w^-_j(y)=u_-(t_j(y),y)$. Then $w_i^+(y)$ and $w^-_j(y)$ are $C^r$-functions bounded together with their derivatives.
\end{lemma}
\begin{proof} We know that $w_0^\pm(y)$ is $C^r$ and bounded together with its derivatives. Now,
\[
	w_1^+(y) = u_+(t_1(y),y) = w_0^+(y) + \int_0^{t_1(y)} f_1(u_+(t,y),y)dt.
\]
Hence $w_1^+(y)$ is $C^r$ and  bounded since so are $w_0^+(y)$, $t_1(y)$ and $f_1(x,y)$. Next, from
\[
	h_x(w_1^+(y),y)f_1(w_1^+(y),y)>2\eta
\]
and the fact that $f_1(x,y)$ is bounded, we see that $\bar h>0$ exists such that $\inf_{y\in\R^m}|h_x(w_1^+(y),y)| > \bar h$. Then, differentiating
$h(w_1^+(y),y) = c_0$ we see that
\[
	\frac{dw_1^{+}}{dy}(y) = -\frac{h_y(w_1^+(y),y)}{h_x(w_1^+(y),y)}
\]
is bounded. More arguments of similar nature show that all derivatives of $w_1^+(y)$ are bounded.

Suppose now that $w^+_i(y)$ is bounded with its derivative. Then, since
\[
	w^+_{i+1}(y) = w^+_i(y) + \int_{t_i(y)}^{t_{i+1}(y)} f_i(u_+(t,y),y)dt
\]
we see that $w^+_{i+1}(y)$ is a bounded function and a similar argument as before shows that the derivatives of $w^+_{i+1}(y)$ are also bounded.
The proof that $w^-_j(y)$ is a $C^r$ function bounded with its derivatives for any $j=-1,\ldots, -M$ is similar.
\end{proof}
Let $i=0,1,\ldots,N$. For $t\ge 0$, let $u_i^+(t,y)$ be the solution of $\dot x = f_i(x,y)$ such that $u_i^+(t_i(y),y) =w^+_i(y)$. Similarly,  for $j=-M,\ldots,-1,0$ let $u_j^-(t,y)$ be the solution of $\dot x = f_j(x,y)$ such that $u_j^-(t_i(y),y) =w^-_j(y)$. Then
\begin{equation}\label{rel_u+-}\begin{array}{l}
	u(t,y) :=
	\left \{\begin{array}{ll}
		u_+(t,y)	& \hbox{$t\ge 0$}	\\
		u_-(t,y)	& \hbox{for $t\le 0$}
	\end{array}\right .		\\
	\qquad = \left \{\begin{array}{ll}
		u_j^-(t,y)	& \hbox{for $t_{j-1}(y)\le t\le t_j(y)$, $j=0,-1,\ldots -M$}	\\
		u_i^+(t,y)	& \hbox{for $t_i(y)\le t\le t_{i+1}(y)$, $i=0,\ldots,N$}
	\end{array}\right .
\end{array}
\end{equation}
where, for simplicity, we set $t_{N+1}(y)=\infty$ and $t_{-M-1}(y)=-\infty$
\[
	u_i^+(t_{i+1}(y),y) = u_+(t_{i+1}(y),y) = w^+_{i+1}(y) = u_{i+1}^+(t_i(y),y)
\]
and similarly,
\[
	u_j^-(t_{j-1}(y),y) = w^-_{j-1}(y) = u^-_{j-1}(t_{j-1}(y),y).
\]

Here we emphasize a convention that we will use throughout the whole paper. When we use the index $i$, such as in $t_i$ or $w^+_i(y)$, we always mean $i=1,\ldots,N$, (sometimes also $i=0$) while, when we use the index $j$, as in $t_j$ or $w^-_j(y)$, we always mean $j=-1,\ldots,-M$ (sometimes also $j=0$).

\section{Exponential dichotomy for piecewise discontinuous systems}
A basic tool in this paper is the notion of exponential dichotomy, whose definition we recall here. Let $J$ be either $[a,\infty), (-\infty,a]$, or
$\R$ and $A(t)$, $t\in J$, be a $n\times n$ continuous matrix. We say that the linear system
\begin{equation}\label{linsys}
\dot x = A(t) x, \quad x\in\R^n
\end{equation}
has an exponential dichotomy on $J$ if there exist a projection $P:\R^n\to \R^n$ and constants $\de>0$ and $K\ge 1$ such that the fundamental matrix
$X(t)$ of \eqref{linsys} satisfying $X(a)=\I$, when  $J=[a,\infty), (-\infty,a]$, or $X(0)=\I$ when $J=\R$, satisfies
\[\begin{array}{l}
	|X(t)PX(s)^{-1}|\le Ke^{-\de(t-s)}, 	\quad \hbox{for $s\le t,\; s,t\in J$}		\\
	|X(s)(\I-P)X(t)^{-1}|\le Ke^{-\de(t-s)}, 	\quad \hbox{for $s\le t,\; s,t\in J$.}	
\end{array}\]

In this section we extend the definition of exponential dichotomy to systems with discontinuities. To allow more generality we consent the solutions to have jump discontinuities at the discontinuity points of the coefficient matrix.

So, let $t_0<t_1<\ldots<t_N$ be real numbers, $B_1,\ldots,B_N$ be invertible $n\times n$ matrices and ${\cal A}(t)$, $t\ge t_0$ be a piecewise continuous matrix with possible discontinuity jumps at $t=t_1,\ldots,t_N$, that is
\begin{equation}\label{Ascr}
{\cal A}(t) = \left \{\begin{array}{ll}
	A_{i-1}(t) & \hbox{if $t_{i-1}\le t<t_i$,}	\\
	 A_N(t) & \hbox{if $t\ge t_N$}
\end{array}\right .\end{equation}
where $A_0(t),\ldots, A_N(t)$ are continuous matrices. Note that ${\cal A}(t)$ is continuous for $t\ge t_0$, $t\ne t_i$, and right-continuous at $t=t_i$, with possible jumps at $t=t_i$ given by the matrix  $A_i(t_i)-A_{i-1}(t_i)$. For $t\ge t_0$ we consider the linear, discontinuous, system
\begin{equation}\label{possys}\begin{array}{l}
	\dot x = {\cal A}(t) x 	\\
	x(t_i^+) = B_ix(t_i^-),
\end{array}\end{equation}
Similarly, if $t_{-M}<\ldots t_{-1}<t_0$ and
\begin{equation}\label{Ascr-}
{\cal A}(t) = \left \{\begin{array}{ll}
	A_{-M}(t) & \hbox{if $t\le t_{-M}$}	\\
	A_{j+1}(t) & \hbox{if $t_j<t\le t_{j+1}$,}
\end{array}\right .\end{equation}
where $A_0(t), A_{-1}(t),\ldots A_{-M}(t)$ are continuous matrices, we consider, for $t\le t_0$, the linear, discontinuous, system
\begin{equation}\label{negsys}\begin{array}{l}
	\dot x = {\cal A}(t) x 	\\
	x(t_j^+) = B_jx(t_j^-)
\end{array}\end{equation}
Note that ${\cal A}(t)$ is continuous for $t\le t_0$, $t\ne t_j$, left-continuous at $t=t_j$, with possible jumps at $t=t_j$, given by the matrix $A_{j+1}(t_j)-A_j(t_j)$.

\begin{remark}\label{onA}
i) As a matter of facts, for $t\ge t_0$, we will consider
\[
{\cal A}(t) = \left \{\begin{array}{ll}
	A_{i-1}(t) & \hbox{if $t_{i-1}\le t\le t_i$,}	\\
%	& \; i=1,\ldots,N		\\
	A_N(t) & \hbox{if $t\ge t_N$}
\end{array}\right .
\]
and similarly for $t\le t_0$. This may cause a duplicate definition of ${\cal A}(t)$ at $t=t_i$, however it will be always clear which among $A_i(t)$ will be taken into account at that point.

ii) The results of this section will be applied to the linear system $\dot x = A(t,y)x$ where $A(t,y)$ is given by
\begin{equation}\label{defA}
	A(t,y) := \left\{\begin{array}{ll}
		f_{-M,x}(u(t,y),y) & \hbox{if $t\le t_{-M}(y)$.}	\\
		f_{j+1,x}(u(t,y),y) & \hbox{if $t_j(y)<t\le t_{j+1}(y)$}		\\
%			& \hbox{$j=-1,\ldots,-M$}		\\
		f_{i-1,x}(u(t,y),y) & \hbox{if $t_{i-1}(y)\le t< t_i(y)$}		\\
%			&  \hbox{$i=1,\ldots,N$}		\\	
		f_{N,x}(u(t,y),y) & \hbox{if $t\ge t_N(y)$.}
	\end{array}\right .
\end{equation}
Note that, being $u(t,y)$ continuous for $t\ne 0$, $A(t,y)$ is continuous for $t\ne 0$, $t\ne t_i(y),t_j(y)$ with jump discontinuities at $t=0,t_i(y),t_j(y)$. More precisely
\[\begin{array}{l}
	A(0^-,y) = f_{0,x}(w_0^-(y),y)	\\
	A(0^+,y) = f_{0,x}(w_0^+(y),y),y)	\\
	A(t_i(y)^-,y) = f_{i-1,x}(w_i^+(y)),y) 		\\
	A(t_i(y)^+,y) = f_{i,x}(w_i^+(y)),y) 	\\
	A(t_j(y)^+,y) = f_{j+1,x}(w^-_i(y),y)		\\
	A(t_j(y)^-,y) = f_{j,x}(w^-_i(y),y).	
\end{array}\]
\end{remark}
Without loss of generality we may assume that $t_0=0$, so in the remaining part of this section we will take $t_0=0$.
\medskip

Let $U_i(t)$ be the fundamental matrix of the linear systems
\[
\dot x =A_i(t)x
\]
on $\R$, that is $\dot U_i(t) = A_i(t)U_i(t)$, $t\in\R$, and $U_i(0)=\I$. The fundamental matrix of \eqref{possys}, where $t\ge 0$, is given by the (discontinuous) invertible matrix
\begin{equation}\label{Xplus}
X_+(t) = \left \{\begin{array}{ll}
	U_0(t)	& \hbox{if $0\le t< t_1$}	\\
	U_i(t)U_i(t_i)^{-1}B_iX_+(t_i^-)	& \hbox{if $t_i\le t< t_{i+1}$}	\\
	U_N(t)U_N(t_N)^{-1}B_NX_+(t_N^-) & \hbox{if $t\ge t_N$}
\end{array}\right .
\end{equation}
that is
\[
	X_+(t) = \left \{ \begin{array}{ll}
	U_0(t)	& \hbox{for $0\le t<t_1$}		\\
	U_1(t)U_1(t_1)^{-1}B_1U_0(t_1) & \hbox{for $t_1\le t<t_2$}		\\
	U_i(t)U_i(t_i)^{-1}B_iU_{i-1}(t_i)U_{i-1}(t_{i-1})^{-1}\ldots B_1U_0(t_1)
	& \hbox{for $t_i\le t<t_{i+1}$}	\\
	\vdots	\\
	U_N(t) U_N(t_N)^{-1} B_NU_{N-1}(t_N)U_{N-1}(t_{N-1})^{-1}\ldots B_1U_0(t_1)	& \hbox{for $t\ge t_N$.}
\end{array}\right .
\]
Similarly, the fundamental matrix of \eqref{possys}, where $t\le 0$, is given by the (discontinuous) invertible matrix
\begin{equation}\label{Xmin}
X_-(t) = \left \{\begin{array}{ll}
	U_0(t) & \hbox{if $t_{-1}<t\le 0$}		\\
	U_j(t)U_j(t_j)^{-1}B_j^{-1}X_-(t_j^+) & \hbox{if $t_{j-1}<t\le t_j$}	\\
	U_{-M}(t)U_{-M}(t_{-M})^{-1}B_{-M}^{-1}X_-(t_{-M}^+) & \hbox{if $t\le t_{-M}.$}
\end{array}\right .
\end{equation}
Note that $X_+(t)$ is continuous for $t\ne t_1,\ldots, t_N$ and right-continuous at $t= t_1,\ldots, t_N$ and $X_-(t)$ is continuous for $t\ne t_{-1},\ldots, t_{-M}$ and left-continuous at $t\ne t_{-1},\ldots, t_{-M}$.

\medskip

It is clear that $\dot X_+(t) = {\cal A}(t)X_+(t)$, for $t\ge 0$, $t\ne t_i$, $\dot X_-(t) = {\cal A}(t)X_-(t)$, for $t\le 0$, $t\ne t_j$, $X_\pm(0)=\I$, the identity matrix, and
\[\begin{array}{l}
	X_+(t_i^+)=B_iX_+(t_i^-)		\\
	X_-(t_j^+)=B_jX_-(t_j^-).
\end{array}\]
Actually we can write
\[
	X_+(t_i)=B_iX_+(t_i^-), \quad X_-(t_j)=B_j^{-1}X_-(t_j^+)
\]
since $X_+(t)$ is right-continuous and $X_-(t)$ is left-continuous. Then, from \eqref{Xplus} and \eqref{Xmin} we see that
\begin{equation}\label{XU+}
	X_+(t)X_+(t_N^+) = U_N(t) U_N(t_N)^{-1}, \forall t\ge t_N
\end{equation}
and, similarly,
\begin{equation}\label{XU-}
	X_-(t)X_-(t_N^-) = U_N(t) U_N(t_N)^{-1}, \forall t\le t_{-M}.
\end{equation}

\begin{remark}\label{rem3.3}
i) Let $\tau\ge 0$ be a fixed number. For $t\ge 0$, $x(t) = X_+(t)X_+(\tau)^{-1}\tilde x$ is the right-continuous solution of
\begin{equation}\label{discA}
\left \{\begin{array}{l}
	\dot x = {\cal A}(t)x,\quad \hbox{for $t\ge 0$, $t\ne t_1,\ldots,t_N$}	\\
	x(t_i^+)=B_ix(t_i^-)		\\
	x(\tau^+) = \tilde x.
\end{array}\right .
\end{equation}
Indeed, it is obvious that $\dot x(t) = {\cal A}(t)x(t)$ for $t\ge 0$, $t\ne t_1,\ldots,t_N$ and that $x(t_i^+) = B_ix(t_i^-)$, since $X_+(t_i^+)=B_iX_+(t_i^-)$. Moreover $x(\tau^+) = X_+(\tau^+)X_+(\tau)^{-1}\tilde x = X_+(\tau)X_+(\tau)^{-1}\tilde x =\tilde x$, since $X_+(t)$ is right-continuous at any $t\ge 0$.

Similarly, for $t\le 0$ and any fixed $\tau\le 0$, $x(t) = X_-(t)X_-(\tau)^{-1}\tilde x$ is the left-continuous solution of
\begin{equation}\label{discB}
\left \{\begin{array}{l}
	\dot x = {\cal A}(t)x,\quad \hbox{for $t\le 0$, $t\ne t_{-1},\ldots,t_{-N}$}	\\
	x(t_{-i}^-)=B_i^{-1}x(t_{-i}^+)		\\
	x(\tau^-) = \tilde x.
\end{array}\right .
\end{equation}

ii) As $X_+(t)X_+(\tau)^{-1}\tilde x$ is right-continuous at $t=t_1,\ldots, t_N$ it is also clear that $X_+(t)X_+(\tau)^{-1}\tilde x$ satisfies
\begin{equation}\label{altdiscA}
\left \{\begin{array}{l}
	\dot x = {\cal A}(t)x,\quad \hbox{for $t\ge 0$, $t\ne t_1,\ldots,t_N$}	\\
	x(t_i)=B_ix(t_i^-)		\\
	x(\tau) = \tilde x
\end{array}\right .
\end{equation}
and, similarly, $X_-(t)X_-(\tau)^{-1}\tilde x$ satisfies
\begin{equation}\label{altdiscB}
\left \{\begin{array}{l}
	\dot x = {\cal A}(t)x,\quad \hbox{for $t\le 0$, $t\ne t_1,\ldots,t_N$}	\\
	x(t_{-i})=B_i^{-1}x(t_{-i}^+)		\\
	x(\tau) = \tilde x.
\end{array}\right .
\end{equation}
\end{remark}
We have the following
\begin{lemma}\label{extend} Suppose that the linear system
\[
	\dot x = A_N(t)x \quad  [resp. \dot x = A_{-M}(t)x]
\]
has an exponential dichotomy on $t\ge t_N$ (resp. $t\le t_{-M}$) with constant $K$, exponent $\delta$ and projection ${\cal P}_+$ (resp. ${\cal P}_-$). Then, the linear system \eqref{possys} (resp. \eqref{negsys}) with ${\cal A}(t)$ as in \eqref{Ascr} (resp. \eqref{Ascr-}) has an exponential dichotomy on $\R_+$, (resp. $\R_-$) with the same exponent $\delta$, constant $\tilde K\ge K$ and projection
\begin{equation}\label{Rtilde}\begin{array}{c}
	\tilde {\cal P}_+= X_+(t_N)^{-1}{\cal P}_+X_+(t_N)	\\
	(\hbox{resp. }\tilde {\cal P}_- = X_-(t_{-M})^{-1}{\cal P}_-X_-(t_{-M})).
\end{array}\end{equation}
\end{lemma}
\begin{proof}
As $X_+(t)$ is right continuous at $t=t_N$, for $t_N\le s\le t$ we have
\[\begin{array}{l}
	|X_+(t)\tilde {\cal P}_+X_+(s)^{-1}| = 	\\
	|U_N(t)U_N(t_N)^{-1}B_NX_+(t_N^-)X_+(t_N^+)^{-1}{\cal P}_+X_+(t_N^+)X_+(t_N^-)^{-1}B_N^{-1}U_N(t_N)U_N(s)^{-1}|	\\
	= |U_N(t)U_N(t_N)^{-1}{\cal P}_+U_N(t_N)U_N(s)^{-1}| \le Ke^{-\delta(t-s)}
\end{array}\]
since $U_N(t)U_N(t_N)^{-1}$ is the fundamental matrix of $\dot x = A_N(t)x$ on $t\ge t_N$. Similarly we see that
\[
	|X_+(s)(\I-\tilde {\cal P}_+)X_+(t)^{-1}|\le Ke^{-\delta(t-s)}
\]
for $t_N\le s\le t$. Being $0\le s\le t\le t_N$ compact and $X_+(t)\tilde {\cal P}_+X_+(s)^{-1}$ piecewise continuous with right and left limits at the discontinuity points, there exists $\hat K\ge K$ such that
\[\begin{array}{l}
	|X_+(t)\tilde {\cal P}_+X_+(s)^{-1}|\le \hat Ke^{-\delta(t-s)}		\\
	|X_+(s)(\I-\tilde {\cal P}_+)X_+(t)^{-1}|\le \hat Ke^{-\delta(t-s)}	
\end{array}\]
for $0\le s\le t\le t_N$. Finally, for $0\le s \le t_N< t$, we have, using right-continuity of $X_+(t)$:
\[\begin{array}{l}
	|X_+(t)\tilde {\cal P}_+X_+(s)^{-1}| \le |X_+(t)\tilde {\cal P}_+X_+(t_N)^{-1}| |B_NX_+(t_N^-)\tilde {\cal P}_+X_+(s)^{-1}|	\\
	\qquad \le Ke^{-\de(t-t_N)}|B_N| \hat K e^{-\de(t_N-s)} \le K|B_N|\hat K e^{-\de(t-s)}		\\
	|X_+(s)(\I-\tilde {\cal P}_+)X_+(t)^{-1}| 	\\
	\qquad \le |X_+(s)(\I-\tilde {\cal P}_+)X_+(t_N^-)^{-1}| |B_N^{-1}X_+(t_N^+)(\I-\tilde {\cal P}_+)X_+(t)^{-1}|	\\
	\qquad \le Ke^{-\de(s-t_N)}|B_N^{-1}| \hat K e^{-\de(t_N-t)} \le K|B_N^{-1}|\hat K e^{-\de(t-s)}	
\end{array}\]
proving the result for $t\ge 0$ with $\tilde K = \max\{K,\hat K, K|B_N^{-1}|\hat K\}$. A similar argument works when $t\le 0$. The Lemma is proved.
\end{proof}

The following result characterises ${\cal R}\tilde P_+$, resp. ${\cal N}\tilde P_-$, in terms of bounded solutions of system \eqref{possys}, resp. \eqref{negsys}, extending to the discontinuous case a simular result for continuous equations.
\begin{lemma}\label{bddlem} Let ${\cal A}(t)$ be either as in \eqref{Ascr} or \eqref{Ascr-}. Suppose that the condition of Lemma \ref{extend} hold and let $\tilde {\cal P}_\pm$ be as in \eqref{Rtilde}. Then $\xi_+\in {\cal R}\tilde {\cal P}_+$ if and only if the solution of the discontinuous system \eqref{possys} such that $x(0)=\xi_+$ is bounded for $t\ge 0$. Similarly, $\xi_-\in {\cal N}\tilde {\cal P}_-$ if and only if the solution of the discontinuous system \eqref{possys} such that
$x(0)=\xi_-$ is bounded for $t\le 0$.
\end{lemma}
\begin{proof}
If $\xi_+\in{\cal R}\tilde {\cal P}_+$ we have
\[
	|X_+(t)\xi_+| = |X_+(t)\tilde {\cal P}_+\xi_+| \le Ke^{-\de t}|\xi_+|
\]
that is the solution of \eqref{possys} starting from $\xi_+$ is bounded. Vice versa, suppose that $x(t)$ is a solution of \eqref{possys} bounded on $t\ge 0$. We have
\[\begin{array}{l}
	|[\I-\tilde {\cal P}_+]x(0)| = |[\I-\tilde {\cal P}_+]X_+(t)^{-1}x(t)| \le Ke^{-\de t}\sup_{t\ge 0}|x(t)| \to 0
\end{array}\]
as $t\to\infty$. Then $[\I-\tilde {\cal P}_+]x(0)=0$ and hence $x(0)\in{\cal R}\tilde {\cal P}_+(0)$. By a similar argument we prove the thesis when $\xi_-\in {\cal N}\tilde {\cal P}_-$ is concerned.
\end{proof}
\medskip

We conclude this section with the following
\begin{lemma}\label{prelem}
Let $B_i$, $B_j$, be invertile $n\times n$ matrices and $k(t)$ be a bounded integrable function for $t\ge 0$, (resp. $t\le 0$).
Suppose the condition of Lemma \ref{extend} hold and set
\[\begin{array}{l}
	\tilde {\cal P}_+^\tau = X_+(\tau)\tilde {\cal P}_+X_+(\tau)^{-1}		\\
	\tilde {\cal P}_-^\tau = X_-(-\tau)\tilde {\cal P}_-X_-(-\tau)^{-1}
\end{array}\]
where $\tilde {\cal P}_\pm$ is as in \eqref{Rtilde} and $0\le \tau\in\R$ is a fixed number. Then, for any $\xi_+\in{\cal R}\tilde {\cal P}_+^\tau$ (resp. $\xi_-\in{\cal N}\tilde {\cal P}_-^\tau$) the linear inhomogeneous system
\begin{equation}\label{eqplus}\begin{array}{l}
	\dot x = {\cal A}(t)x + k(t) 	\\
	x(t_i^+) = B_ix(t_i^-),		\\	%\quad \hbox{$i=1,\ldots,N$}	\\
	\tilde {\cal P}_+^\tau x(\tau) = \xi_+
\end{array}\end{equation}
with $t\ge 0$, [resp.
\[\begin{array}{l}
	\dot x = {\cal A}(t)x + k(t) 	\\
	x(t_j^-) = B_j^{-1}x(t_j^+)	\\
	(\I-\tilde {\cal P}_-^\tau)x(-\tau) = \xi_-
\end{array}\]
when $t\le 0$] has the unique right-continuous, [resp. left-continuous when $t\le 0$] bounded solution
\begin{equation}\label{xpos}\begin{array}{l}
	\dis x(t) = X_+(t)\tilde {\cal P}_+X_+(\tau)^{-1}\xi_+ + \int_\tau^t X_+(t)\tilde {\cal P}_+X_+(s)^{-1}k(s) ds 		\\
	\dis \qquad - \int_t^\infty X_+(t)(\I-\tilde {\cal P}_+)X_+(s)^{-1}k(s) ds
\end{array}\end{equation}
[resp.
\begin{equation}\label{xneg}\begin{array}{l}
	\dis x(t) = X_-(t)(\I-\tilde {\cal P}_-)X_-(-\tau)^{-1}\xi_- + \int_{-\infty}^t X_-(t)\tilde {\cal P}_-X_-(s)^{-1}k(s) ds 		\\
	\dis \qquad - \int_t^{-\tau} X_-(t)(\I-\tilde {\cal P}_-)X_-(s)^{-1}k(s) ds
\end{array}\end{equation}
if $t\le 0$]. Moreover such a solution satisfies
\begin{equation}\label{estpos}
	\sup_{t\ge\tau} |x(t)| \le K[|\xi_+| + 2\delta^{-1}\sup_{t\ge 0}|k(t)|]
\end{equation}
if $t\ge 0$ [resp.
\begin{equation}\label{estneg}
	\sup_{t\le -\tau}|x(t)| \le K[|\xi_-| + 2\delta^{-1}\sup_{t\le 0}|k(t)|]
\end{equation}
if $t\le 0$].
\end{lemma}
\begin{proof} We only give the proof for $t\ge 0$, the proof for $t\le 0$ being similar. We prove uniqueness, first. Suppose that $x_1(t), x_2(t)$ are two  solutions of \eqref{eqplus}, right-continuous and bounded for $t\ge 0$. Then $x(t)=x_1(t)-x_2(t)$ is a right-continuous, bounded solution of
\[\begin{array}{l}
	\dot x = {\cal A}(t)x		\\
	x(t_i^+) = B_ix(t_i^-)	\\
	\tilde {\cal P}_+^\tau x(\tau) = 0.
\end{array}\]
Then, as we have observed in Remark \ref{rem3.3}-i), $x(t) = X_+(t)X_+(\tau)^{-1}x(\tau)$, so:
\[\begin{array}{l}
	x(\tau) = (\I-\tilde {\cal P}_+^\tau)x(\tau) = (\I-\tilde {\cal P}_+^\tau)X_+(\tau)X_+(t)^{-1}x(t) 	\\
	\qquad = X_+(\tau)(\I-\tilde {\cal P}_+)X_+(t)^{-1}x(t)
\end{array}\]
from which we get
\[
	|x(\tau)|\le Ke^{-\delta (t-\tau)}\sup_{t\ge \tau}|x(t)| \to 0
\]
as $t\to\infty$, since $x(t)$ is bounded. This proves that $x(\tau)=0$ and then $x(t)=0$. Hence we have uniqueness. To show the existence we observe that the function given in \eqref{xpos} is right-continuous, bounded for $t\ge \tau$ (and hence also for $t\ge 0$) and satisfies \eqref{eqplus}. Finally, from \eqref{xpos} we get:
\[
|x(t)| \le K e^{-\delta(t-\tau)}|\xi_+| + \int_0^\infty Ke^{-\delta |t-s|}ds\, \sup_{t\ge 0}|k(t)|
\]
from which \eqref{estpos} easily follows.
\end{proof}
\section{Bounded solutions on the half lines}
In this section we prove the existence of continuous solutions $(x(t),y(t))$, with $t\ge 0$, of the perturbed linear system \eqref{pert} such that
\[
	\sup_{t\ge 0} |x(t) - u(t,y(t))|<\rho
\]
where $\rho>0$ is a sufficiently small positive real number. By a similar argument we can also prove the existence of continuous solutions of \eqref{pert}, with $t\le 0$, such that
\[
	\sup_{t\le 0} |x(t) - u(t,y(t))|<\rho.
\]

From $A_3)$ it follows that the number of the eigenvalues of $f_{N,x}(w_+(y),y)$, $f_{-M,x}(w_-(y),y)$ with negative (and then also positive) real parts, counted with multiplicities, is independent of $y\in\R^m$. Moreover it also follows that all eigenvalues are bounded functions of $y\in\R^m$. Indeed, since $f_{N,x}(w_+(y),y)$ is bounded, the matrix $\I-\la^{-1}f_{N,x}(w_+(y),y)$ is invertible for $|\la|>R$, sufficiently large and independent of $y$. Hence all eigenvalues have to satisfy $|\la|\le R$. The same arguments work as far as the eigenvalues of $f_{-M,x}(w_-(y),y)$ are concerned.

As in $A_3)$, let $k<n$ be the number of eigenvalues with negative real parts, counted with multiplicities, of the matrix $f_{N,x}(w_+(y),y)$ and $\de_0$ be any positive number strictly less than $\min\{|{Re}\la(y)|\}$, where $\la(y)$ are the eigenvalues of $f_{N,x}(w_+(y),y)$. According to \cite{Cop} the system $\dot x = f_{N,x}(w_+(y),y)x$ has an exponential dichotomy on $\R$ with exponent $\de_0$ and (spectral) projection
\[\begin{array}{l}
	\dis P^0_+(y) = \frac{1}{2\pi i}\int_\Gamma (z\I-f_{N,x}(w_+(y),y))^{-1} dz 	\\
	\qquad \dis = \sum_{Re\la(y)<0} {\rm Res}((z\I-f_{N,x}(w_+(y),y))^{-1}, z=\la(y))
\end{array}\]
where ${\rm Res}(F(z),z=z_0)$ is the residual of the meromorphic function $F(z)$ at $z_0$ and $\Gamma$ is a closed curve that contains in its interior all eigenvalues of $f_{N,x}(0,y)$ with negative real parts, but none of those with positive real parts. Hence $\sup_{y\in\R^m}|P_+^0(y)|\le M$, for some $M\ge 1$.  Similarly we see that  $\dot x = f_{-M,x}(w_-(y),y)x$ has an exponential dichotomy on $\R$ with exponent $\de_0$ and projection $P^0_-(y)$ such that $\sup_{y\in\R^m}|P_+^0(y)|\le M$, for some $M\ge 1$.

Now,
%recalling that
%\[\begin{array}{l}
%	u_i^+(t,y) = w_i(t-t_{i-1}(y),w_{i-1}^+(y),y)		\\
%	u_i^-(t,y) = w_i(t-t_{-i+1}(y),w^-_{i-1}(y),y)
%\end{array}\]
%(see \eqref{u_i+}-\eqref{u_i-}),
from $A_2)$ we know that
\[
	\lim_{t\to\pm\infty} u^\pm_N(t,y) = w_\pm(y)
\]
uniformly with respect to $y\in\R^m$.

Let $T_+>\sup_{y\in\R^m}t_N(y)$, $T_-<\inf_{y\in\R^m}t_{-M}(y)$ and take $0<\de<\de_0$. From the roughness of exponential dichotomies (cfr. \cite[Proposition 2, p. 34]{Cop}) the linear systems
\begin{equation}\label{variat+}
	\dot x = f_{N,x}(u^+_N(t+T_+,y),y)x
\end{equation}
and
\begin{equation}\label{variat-}
	\dot x = f_{-M,x}(u^-_{-M}(t+T_-,y),y)x	
\end{equation}
have an exponential dichotomy on $\R_+$, $\R_-$ resp., uniform with respect to $y\in\R^m$, with projections $P_+(y)$, resp. $P_-(y)$, of rank $k$, constant $K$ and exponent $\de$. Moreover, according to \cite[Proposition 2.3]{P2}, it can be assumed that, for $|y-y_0|$ sufficiently small it results: ${\cal N}P_+(y)={\cal N}P_+(y_0)$, ${\cal R}P_-(y)={\cal R}P_-(y_0)$ and in this case the projections are smooth with respect to $y$. Note that, ${\cal N}P_+(y)={\cal N}P_+(y_0)$ and ${\cal R}P_-(y)={\cal R}P_-(y_0)$ are equivalent to
\begin{equation}\label{commP}\begin{array}{l}
	P_+(y)=P_+(y)P_+(y_0), \quad P_+(y_0)=P_+(y_0)P_+(y)		\\
	P_-(y)=P_-(y_0)P_-(y), \quad P_-(y_0)=P_-(y)P_-(y_0) .
\end{array}\end{equation}
\medskip

For $i=0,1,\ldots,N$, $j=0,-1,\ldots,-M$, let $U_i^+(t,y)$, [resp. $U_j^-(t,y)$] be the fundamental matrix of
\[
	\dot x = f_{i,x}(u_i^+(t,y),y)x \quad \hbox{[resp. } \dot x = f_{j,x}(u_j^-(t,y),y)x\hbox{]}
\]
in $\R_+$ resp., $\R_-$ that is
\[\begin{array}{l}
	\dot U_i^+(t,y) =  f_{i,x}(u_i^+(t,y),y)U_i^+(t,y), \quad t\ge 0, 	\\
	U_i^+(0,y)=\I .
\end{array}\]
resp.
\[\begin{array}{l}
	\dot U_j^-(t,y) =  f_{j,x}(u_j^-(t,y),y)U_j^-(t,y), \quad t\le 0, 	\\
	U_j^-(0,y)=\I .
\end{array}\]

We have the following
\begin{lemma}\label{EDshift} For any $\tau\in\R$ the linear system
\begin{equation}\label{variat1}
	\dot x = f_{N,x}(u^+_{N}(t,y),y)x,
\end{equation}
resp.
\begin{equation}\label{variat2}
	\dot x = f_{-M,x}(u^-_{-M}(t,y),y)x,
\end{equation}
has an exponential dichotomy on $t\ge \tau$, resp. $t\le \tau$, with exponent $\de$, constant $\tilde K$ independent on $y$ and projections
\[\begin{array}{l}
	Q_+(y) = U^+_N(\tau,y)U^+_N(T_+,y)^{-1}P_+(y)U^+_N(T_+,y) U^+_N(\tau,y)^{-1} 		\\
	Q_-(y) = U^-_{-M}(\tau,y)U^-_{-M}(T_-,y)^{-1}P_-(y)U^-_{-M}(T_-,y) U^-_{-M}(\tau,y)^{-1} .
\end{array}\]
In particular, if $\tau=T_+$, resp. $\tau=T_-$, then $Q_+(y) = P_+(y)$, resp. $Q_-(y) = P_-(y)$, and $\tilde K=K$.
\end{lemma}
\begin{proof}
As $U_N^+(t,y)$ is the fundamental matrix of \eqref{variat1} in $\R_\pm$, the fundamental matrix of \eqref{variat+} on $\R_+$ is
\[
	\tilde U(t,y) := U^+_N(t+T_+,y)U^+_N(T_+,y)^{-1}
\]
while the fundamental matrix of \eqref{variat1} on $t\ge\tau$ is
\[
	\bar U(t,y) = U^+_N(t,y)U^+_N(\tau,y)^{-1}.
\]
Then
\[
	\bar U(t,y) U^+_N(\tau,y) = U^+_N(t,y) = \tilde U(t-T_+,y)U^+_N(T_+,y).
\]
So we get, for $t \ge s \ge T_+$:
\[\begin{array}{l}
	|\bar U(t,y) Q_+(y) \bar U(s,y)^{-1}| =	\\
	|\tilde U(t-T_+,y)U^+_N(T_+,y)U^+_N(\tau,y)^{-1}Q_+(y)U^+_N(\tau,y)U^+_N(T_+,y)^{-1}\tilde U(s-T_+,y)^{-1}| = 	\\
	|\tilde U(t-T_+,y)P_+(y)\tilde U(s-T_+,y)^{-1}| \le Ke^{-\de(t-s)}
\end{array}\]
and similarly
\[\begin{array}{l}
	|\bar U(s,y) (\I-Q_+(y)) \bar U(t,y)^{-1}| =	\\
	|\tilde U(s-T_+,y)(\I-P_+(y))\tilde U(t-T_+,y)^{-1}| \le Ke^{-\de(t-s)}.
\end{array}\]
Hence \eqref{variat1} has an exponential dichotomy on $t\ge T_+$ with projection $Q_+(y)$, constant $\tilde K=K$ and exponent $\delta$. If $\tau\ge T_+$ the proof is finished. Moreover if $\tau=T_+$ we have
\[
	Q_+(y) = P_+(y) \quad \hbox{and}\quad Q_-(y) = P_-(y).
\]
If $\tau<T_+$, for $\tau\le s\le t\le T_+$ let
\[
	v(t) := \bar U(t,y) Q_+(y) \bar U(s,y)^{-1} =\tilde U(t-T_+,y)P_+(y)\tilde U(s-T_+,y)^{-1}.
\]
Then $v(t)$ satisfies
\[
	\dot v(t) = f_{N,x}(u^+_N(t,y),y)v(t)
\]
with $v(T_+)=P_+(y)\tilde U(s-T_+,y)^{-1}$. Hence
\[
	v(t) = P_+(y))\tilde U(s-T_+,y)^{-1} - \int_t^{T_+} f_{N,x}(u^+_N(s,y),y)v(s) ds.
\]
So, using $|P_+(y)\tilde U(s-T_+,y)^{-1}|\le Ke^{-\de(T_+-s)}$:
\[
	|v(t)| \le Ke^{-\de(T_+-s)} + \int_t^{T_+} \bar F |v(s)| ds.
\]
where $\bar F$ is an upper bound for $|f_{N,x}(u^+_N(s,y),y)|$. Setting $r(t):\int_t^{T_+} |v(s)| ds$ we get
\[
	\frac{d}{dt}[r(t)e^{\bar F t}] = [r'(t)+\bar Fr(t)]e^{\bar F t} \ge -Ke^{\bar F t}.
\]
Integrating on $[t,T_+]$ we get, since $r(T_+)=0$, $- r(t)e^{\bar F t} \ge -K\bar F^{-1} [e^{\bar F T_+}-e^{\bar F t}]$. Hence
\[
	\int_t^{T_+} \bar F |v(s)| ds = \bar F r(t) \le K [e^{\bar F (T_+-t)}-1]
\]
and then
\[\begin{array}{l}
	|v(t)| \le Ke^{-\de(T_+-s)} + K e^{\bar F (T_+-t)} 		\\
	= K[ e^{-\de(T_+-t)} + e^{\bar F(T_+-t)+\de(t-s)}]e^{-\de(t-s)} \le \hat K e^{-\de (t-s)}
\end{array}\]
where
\[
	\hat K = K[1+e^{(\bar F+\de)(T_+-\tau)}].
\]
Note that $\hat K$ is independent of $y$. Finally, for $\tau\le s\le T_+\le t$ we have
\[\begin{array}{l}
	|\bar U(t,y) Q_+(y) \bar U(s,y)^{-1}| \le |\bar U(t,y) Q_+(y) \bar U(T_+,y)^{-1}| |\bar U(T_+,y) Q_+(y) \bar U(s,y)^{-1}| 	\\
	\le Ke^{-\de(t-T_+)} \hat Ke^{\de(s-T_+)} = K\hat Ke^{\de(t-s)}.
\end{array}\]
By a similar argument we prove that
\[
	|\tilde U(t-T_+,y)(\I-P_+(y))\tilde U(s-T_+,y)^{-1}|\le \tilde Ke^{\de(t-s)}
\]
for $\tau\le t\le s$.

The proof that \eqref{variat2} has an exponential dichotomy on $t\le\tau$, with projection $Q_-(y)$, exponent $\de$ and a constant independent of $y\in\R^m$ is similar.
\end{proof}
\medskip

From the proof of the parametric stable (resp. unstable) Theorem (see \cite[p. 18]{P1}) it follows that, all solutions $w(t,y)$ of
$\dot w = f_N(w,y)$ such that $\sup_{t\ge 0}|w(t,y)-w^+(y)|\le\Delta\ll 1$ (resp. $\sup_{t\le 0}|w(t,y)-w^-(y)|\le\Delta\ll 1$) can be obtained
as fixed points of a uniform contraction, $C^r$ with respect to $y$, on the space of bounded functions. It follows, then, that these solutions are bounded together with their derivatives.

The following Lemma states that the bounds for $u^+_N(t,y)$, $u^-_{-M}(t,y)$ and their derivatives with respect to $y$ can be taken independent of $y$. Due to its technical character the proof of Lemma \ref{Onu+} is postponed in the Appendix.

\begin{lemma}\label{Onu+} Assume $A_1)-A_3)$. Then $u^+_N(t,y)$ and its derivatives with respect to $y$ are bounded  uniformly with
respect to $y$, on $t\ge t_N(y)$. Similarly, $u^-_{-M}(t,y)$ and its derivatives with respect to $y$ are bounded, uniformly with respect to $y$, on $t\le t_{-M}(y)$.
\end{lemma}

\medskip

From \cite[Theorem 2]{B} we get the following
\begin{theorem}\label{bddsol} Suppose  that $A_1)-A_4)$  hold and let $\ga,\be>0$ be such that $r\beta <\gamma<\delta$. Then there exist $\rho>0$, $\bar\ep>0$ and $\bar\sigma>0$ such that for $(\al,\ep)\in\R^m\times\R$ and $\xi_+\in {\cal R}P_+(\al)$, with $|\al-y_0|<\bar\si$, $|\xi_+|<\bar\sigma$, and $|\ep|<\bar\ep$, system
\begin{equation}\label{lasteq}\begin{array}{l}
	\dot x = f_N(x,y)			\\
	\dot y = \ep g(x,y,\ep)	\\
	P_+(\al)[x(T_+)-u(T_+,\al)] = \xi_+	\\
	y(T_+)=\al,
\end{array}\end{equation}
has a unique solution $(x^+_N(t,\xi_+,\al,\ep),y^+_N(t,\xi_+,\al,\ep))$, $t\ge T_+$ such that
\begin{equation}\label{contplus}
	\sup_{t\ge T_+} |x^+_N(t,\xi_+,\al,\ep) - u^+_N(t,y_+(t,\xi_+,\al,\ep))| < \rho.
\end{equation}
Moreover,
\begin{equation}\label{contplus2}
	\sup_{t\ge T_+} |x^+_N(t,\xi_+,\al,\ep) - u^+_N(t,y_+(t,\xi_+,\al,\ep))| \to 0
\end{equation}
as $|\xi_+|+|\ep|\to 0$, uniformly with respect to $\al$. Similarly for $(\al,\ep)\in\R^m\times\R$ and $\xi_-\in{\cal N}P_-(\al)$ such that  $|\al-y_0|<\bar\si$, $|\xi_-|<\bar\si$ system
\[\begin{array}{l}
	\dot x = f_{-M}(x,y)			\\
	\dot y = \ep g(x,y,\ep)	\\
	(\I-P_-(\al))[x(T_-)-u(T_-,\al)] = \xi_-	\\
	y(T_-)=\al
\end{array}\]
has a unique solution $(x^-_{-M}(t,\xi_-,\al,\ep),y^-_{-M}(t,\xi_-,\al,\ep))$, $t\le T_-$ such that
\begin{equation}\label{contmin}
	\sup_{t\le T_-} |x^-_{-M}(t,\xi_-,\al,\ep) - u^-_{-M}(t,y_-(t,\xi_-,\al,\ep))| < \rho.
\end{equation}
Moreover
\begin{equation}\label{contmin2}
	\sup_{t\le T_-} |x^-_{-M}(t,\xi_-,\al,\ep) - u^-_{-M}(t,y_-(t,\xi_-,\al,\ep))| \to 0
\end{equation}
as $|\xi_-|+|\ep|\to 0$, uniformly with respect to $\al$. Moreover, for any $p\le r-1$, the function
\[
(\xi_+,\al,\ep) \mapsto (x^+_N(t,\xi_+,\al,\ep),y^+_N(t,\xi_+,\al,\ep))
\]
is of class $C^p$ and, for $t\ge T_+$, the $p$-th order derivatives are bounded above in absolute value by $K'_p e^{(p+1)\be|t-t_N(y)|}$, where $K'_p$ is a suitable constant. Similarly, for $t\le T_-$, the map
\[
(\xi_-,\al,\ep) \mapsto (x^-_{-M}(t,\xi_-,\al,\ep),y^-_{-M}t,\xi_-,\al,\ep))
\]
is of class $C^p$ and the $p$-th order derivatives are bounded above in absolute value by $K''_p e^{(p+1)\be|t-t_{-N}(y)|)}$, where $K''_p$ is a suitable constant.
\end{theorem}
\begin{proof} In \cite{B} the equations are considered for $t\ge 0$ or $t\le 0$ and instead of $P_\pm(\al)$ it is considered $P_\pm(y_0)$. To obtain the result as in Theorem \ref{bddsol} for $\xi_+\in{\cal R}P_+(\al)$, (resp. $\xi_-\in{\cal N}P_-(\al)$) we take $P_+(y_0)\xi_+\in{\cal R}P_+(y_0)$ (resp. $(\I-P_-(y_0))\xi_-\in{\cal N}P_-(y_0)$) and apply \cite[Theorem 1]{B} with $u^+_N(t+T_+,y)$, resp. $u^-_{-M}(t+T_-,y)$, instead of $u(t,y)$. Suppose, to fix ideas, that $t\ge 0$. We know that \eqref{variat+} has an exponential dichotomy on $\R_+$ with projection $P_+(y)$ and, from Lemma \ref{Onu+} that $u^+_N(t+T_+,y)$ and its derivatives with respect to $y\in\R^m$ are bounded uniformly with respect to $y\in\R^m$. From \cite[Theorem 1]{B} we obtain then a unique solution
\[
	(\tilde x(t,\xi_+,\al,\ep),\tilde y(t,\xi_+,\al,\ep))
\]
of \eqref{pert} such that
\[\begin{array}{l}
	P_+(y_0)[\tilde x(0)-u^+_N(T_+,\al)] = P_+(y_0)\xi_+	\\
	\tilde y(0)=\al
\end{array}\]
and
\[\begin{array}{l}
	\sup_{t\ge0} |\tilde x(t,\xi_+,\al,\ep)-u^+_N(t+T_+,\tilde y(t,\xi_+,\al,\ep))| <\rho,	\\
	\sup_{t\ge0} |\tilde x(t,\xi_+,\al,\ep)-u^+_N(t+T_+,\tilde y(t,\xi_+,\al,\ep))| \to 0	
\end{array}\]
as $|\xi_+|+|\ep|\to 0$, uniformly with respect to $\al$. Setting
\[\begin{array}{l}
	x^+_N(t,\xi_+,\al,\ep) = \tilde x(t-T_+,\xi_+,\al,\ep)	\\
	y^+_N(t,\xi_+,\al,\ep) = \tilde y(t-T_+,\xi_+,\al,\ep)
\end{array}\]
we see that $(x^+_N(t,\xi_+,\al,\ep), y^+_N(t,\xi_+,\al,\ep))$ satisfies
\[\begin{array}{l}
	\dot x = f_N(x,y)			\\
	\dot y = \ep g(x,y,\ep)	\\
	y(T_+)=\al
\end{array}\]
and
\[
	P_+(\al)[x(T_+)-u^+_N(T_+,\al)] = P_+(\al)P_+(y_0)\xi_+ =\xi_+
\]
since $P_+(\al)P_+(y_0)=P_+(\al)$.
A similar argument works for $t\le T_-$. Finally, we observe that, although in \cite{B} it is not explicitly stated that \eqref{contplus}, \eqref{contmin} hold uniformly with respect to $\al$, this fact easily follow from \cite[(20)-(23)]{B} and $\sup_{t\in\R_\pm} |\dot y_\pm(t,\xi_\pm,\al,\ep)|=O(\ep)$.
\end{proof}
\begin{remark}\label{rem1} According to assumption $A_3)$, $x=w_\pm(y)$ are normally hyperbolic manifolds for the system
\[
	\dot x = f_N(x,y), \quad \dot y = 0.
\]
These manifolds perturb to normally hyperbolic invariant manifolds $x = \tilde w_\pm(y,\ep)$ for the system
\[\begin{array}{l}
	\dot x = f_N(x,y), \\
	\dot y = \ep g(x,y,\ep)
\end{array}\]
(see, for example, \cite{BF-1,Sa}). Let $y_c(t,\al,\ep)$ be the solution of $\dot y = g(\tilde w_\pm(y,\ep),y,\ep)$ such that $y_c(0,\al,\ep) = \al$.

It follows from \cite[Theorem 1]{BP1} that, for $|\alpha-y_0|$ and $|\xi^0_+|$ sufficiently small, with $\al\in\R^m$ and $\xi^0_+\in{\cal R}P_+^0(y_0)$, there exists a unique solution of
\[\begin{array}{l}
	\dot x = f_N(x,y)			\\
	\dot y = \ep g(x,y,\ep)
\end{array}\]
such that
\[\begin{array}{l}
	\sup_{t\ge T_+}e^{\be t}|x(t,\xi^0_+,\al,\ep) - \tilde w_+(y(t,\xi^0_+,\al,\ep),\ep)| \le \mu_1	\\
	\sup_{t\ge T_+}e^{\be t}|y(t,\xi^0_+,\al,\ep) - y_c(\ep (t-T_+),\al,\ep)| \le \mu_2
\end{array}\]
and $P_+^0(y_0)[x(T_+,\xi^0_+,\al,\ep) - \tilde w_+(y(0,\xi^0_+,\al,\ep),\ep)] =\xi^0_+$.

As $\lim_{\ep\to 0} |\tilde w_+(y,\ep) - w_+(y)| = 0$ and $\lim_{t\to\infty}|u_N(t,y) - w_+(y)| = 0$ uniformly with respect to $y$ we conclude that
\[
	\sup_{t\ge T}|x(t,\xi^0_+,\al,\ep) - u_N(t,y(t,\xi^0_+,\al,\ep))| < \rho	
\]
where $T\ge T_+$ is sufficiently large. From continuous dependence we see that the above estimate holds with $\sup_{t\ge T_+}$ instead of $\sup_{t\ge T}$, provided $|\xi^0_+|+|\ep| \ll 1$ is sufficiently small, and then
\[
	x^+_N(t,\xi_+,\al,\ep) =  x(t,\xi^0_+,\al_0,\ep), \quad y^+_N(t,\xi_+,\al,\ep) =  y(t,\xi^0_+,\al_0,\ep),
\]
for some $\xi^0_+$ and $\al_0$. As a consequence the solutions given in Theorem \ref{bddsol} satisfy
\[\begin{array}{l}
	\sup_{t\ge T_+}e^{\be t}|x^+_N(t,\xi_+,\al,\ep) - \tilde w_+(y_N(t,\xi^0_+,\al_0,\ep),\ep)| \le \mu_1	\\
	\sup_{t\ge T_+}e^{\be t}|y^+_N(t,\xi_+,\al,\ep) - y_c(\ep (t-T_+),\xi^0_+,\ep)| \le \mu_2 .
\end{array}\]
Similarly we see that
\[\begin{array}{l}
	\sup_{t\le T_-}e^{-\be t}|x^-_{-M}(t,\xi_-,\al,\ep) - \tilde w_-(y_{-M}(t,\xi^0_-,\al_0,\ep),\ep)| \le \mu_1	\\
	\sup_{t\le T_-}e^{-\be t}|y^-_{-M}(t,\xi_-,\al,\ep) - y_c(\ep (t+T_-),\xi^0_-,\ep)| \le \mu_2.
\end{array}\]
\end{remark}

\medskip

In the remaining part of this section we extend the solutions obtained in Theorem \ref{bddsol} to continuous, piecewise $C^1$ solutions of \eqref{pert} in $\R_+$ resp. $\R_-$. We have the following
\begin{theorem}\label{extension} There exist $\rho>0$, bounded $C^r$-functions
\[
t^*_{-M}(\xi_-,\al,\ep)<\ldots<t^*_{-1}(\xi_-,\al,\ep)<0<t^*_1(\xi_+,\al,\ep) < \ldots < t ^*_N(\xi_+,\al,\ep),
\]
and continuous, piecewise $C^r$ solutions of \eqref{pert}
\[
(x_\pm(t,\xi_\pm,\al,\ep),y_\pm(t,\xi_\pm,\al,\ep))
\]
defined for $t\ge 0$ and $t\le 0$ resp., such that
\[\begin{array}{l}
	\lim_{(\xi_+,\ep)\to 0}|t^*_i(\xi_+,\al,\ep)-t_i(\al)| = 0, 		\\
	\lim_{(\xi_-,\ep)\to 0}|t^*_j(\xi_-,\al,\ep)-t_j(\al)| = 0
\end{array}\]
($i=1,\ldots,N$, $j=-1,\ldots,-M$) uniformly with respect to $\al\in\R^m$ and
\[\begin{array}{l}
	c_{i-2}<h(x_+(t,\xi_+,\al,\ep),y_+(t,\xi_+,\al,\ep))<c_{i-1}, \quad \hbox{ for $t^*_{i-1}(\xi_+,\al,\ep)<t<t^*_i(\xi_+,\al,\ep)$}	\\
	h(x_+(t,\xi_+,\al,\ep),y_+(t,\xi_+,\al,\ep))>c_{N-1}, \quad \hbox{ for $t>t^*_N(\xi_+,\al,\ep)$}		\\
	c_j<h(x_-(t,\xi_-,\al,\ep),y_-(t,\xi_-,\al,\ep))<c_{j+1}, \quad \hbox{ for $t^*_j(\xi_-,\al,\ep)<t<t^*_{j+1}(\xi_-,\al,\ep)$}	\\
	h(x_-(t,\xi_-,\al,\ep),y_-(t,\xi_-,\al,\ep))<c_{-M}, \quad \hbox{ for $t<t^*_{-M}(\xi_-,\al,\ep)$}		\\
	h(x_+(t^*_i(\xi_+,\al,\ep),\xi_+,\al,\ep),y_+(t^*_i(\xi_+,\al,\ep),\xi_+,\al,\ep)) = c_{i-1}	\\
	h(x_-(t^*_j(\xi_-,\al,\ep),\xi_-,\al,\ep),y_-(t^*_j(\xi_-,\al,\ep),\xi_-,\al,\ep)) = c_j	\\
	\frac{\partial}{\partial t} h(x_+(t,\xi_+,\al,\ep),y_+(t,\xi_+,\al,\ep))_{|t=t^*_i(\xi_+,\al,\ep)} > \eta 	\\
	\frac{\partial}{\partial t} h(x_-(t,\xi_-,\al,\ep),y_-(t,\xi_-,\al,\ep))_{|t=t^*_j(\xi_-,\al,\ep)} > \eta
\end{array}\]
where $t^*_0(\xi_\pm,\al,\ep)=0$. Moreover
\[\begin{array}{ll}
	(x_+(t,\xi_+,\al,\ep),y_+(t,\xi_+,\al,\ep)) = (x^+_N(t,\xi_+,\al,\ep), y^+_N(t,\xi_+,\al,\ep)) & \hbox{for $t\ge T_+$}		\\
	(x_-(t,\xi_+,\al,\ep),y_-(t,\xi_+,\al,\ep)) = (x^-_{-M}(t,\xi_+,\al,\ep), y^-_{-M}(t,\xi_+,\al,\ep)) & \hbox{for $t\le T_-$}
\end{array}\]
and
\begin{equation}\label{propsol-A}\begin{array}{ll}
	\rho > \sup_{t\ge 0} |x_+(t,\xi_+,\al,\ep) - u(t,y_+(t,\xi_+,\al,\ep))| \to 0 & \hbox{as $|\xi_+|+|\ep|\to 0$}	\\
	\rho > \sup_{t\le 0} |x_-(t,\xi_-,\al,\ep) - u(t,y_-(t,\xi_-,\al,\ep))| \to 0 & \hbox{as $|\xi_-|+|\ep|\to 0$}
\end{array}\end{equation}
uniformly with respect to $\al$.
\end{theorem}
\begin{proof}
Suppose $t\ge 0$. Since $y^+_N(T_+,\xi_+,\al,\ep)=\al$, from \eqref{contplus} we see that
\[
	|x^+_N(T_+,\xi_+,\al,\ep)-u^+_N(T_+,\al)| \to 0
\]
as $|\xi_+|+|\ep|\to 0$ (uniformly with respect to $\al$). Hence $h(x^+_N(T_+,\xi_+,\al,\ep),\al)>c_N$, provided $|\xi_+|+|\ep|$ is sufficiently small and uniformly with respect to $\al$. Then $(x^+_N(t,\xi_+,\al,\ep),y^+_N(t,\xi_+,\al,\ep)$ can be extended to a solution of
\[\left \{\begin{array}{l}
	\dot x = f_N(x,y)		\\
	\dot y = \ep g(x,y)
\end{array}\right .\]
which is defined for $t\ge 0$ and such that
\[
	\sup_{0\le t\le T_+} |x^+_N(t,\xi_+,\al,\ep)-u^+_N(t,y^+_N(t,\xi_+,\al,\ep))| \to 0
\]
as $|\xi_+|+|\ep|\to 0$, uniformly with respect to $\al$. Note that
\[
	\sup_{0\le t\le T_+} |y^+_N(t,\xi_+,\al,\ep) - \al| \to 0
\]
as $\ep\to 0$, uniformly with respect to $(\xi_+,\al)$, since $y^+_N(T_+,\xi_+,\al,\ep) = \al$ and $\dot y^+_N(t,\xi_+,\al,\ep) = O(|\ep|)$ uniformly with respect to $(\xi_+,\al)$. Now, from $A_1)$ and the implicit function theorem it follows that there exists a $C^r$-function $t^*_N(\xi_+,\al,\ep)$, bounded together with its derivatives, such that
\[\begin{array}{l}
	|t^*_N(\xi_+,\al,\ep)-t_N(\al)|\to 0, \quad \hbox{as $|\xi_+|+|\ep|\to 0$}		\\
	h(x^+_N(t,\xi_+,\al,\ep),y^+_N(t,\xi_+,\al,\ep)) > c_N, \quad \hbox{for $t>t^*_N(\xi_+,\al,\ep)$}		\\
	h(x^+_N(t^*_N(\xi_+,\al,\ep),\xi_+,\al,\ep),y^+_N(t^*_N(\xi_+,\al,\ep),\xi_+,\al,\ep)) = c_N	\\
	\frac{\partial}{\partial t} h(x^+_N(t,\xi_+,\al,\ep),y^+_N(t,\xi_+,\al,\ep))_{|t=t^*_N(\xi_+,\al,\ep)} > \eta
\end{array}\]
provided $|\xi_+|+|\ep|$ is sufficiently small, uniformly with respect to $\al$. Next, from the continuous dependence on the data, we see that the system
\[\begin{array}{l}
	\dot x = f_{N-1}(x,y)			\\
	\dot y = \ep g(x,y,\ep)	\\
	x(t^*_N(\xi_+,\al,\ep)) = x^+_N(t^*_N(\xi_+,\al,\ep),\xi_+,\al,\ep)	\\
	y(t^*_N(\xi_+,\al,\ep)) = y^+_N(t^*_N(\xi_+,\al,\ep),\xi_+,\al,\ep)
\end{array}\]
has a unique solution $(x^+_{N-1}(t,\xi_+,\al,\ep),y^+_{N-1}(t,\xi_+,\al,\ep))$, defined for $-1\le t \le t^*_N(\xi_+,\al,\ep)$, such that
\[
	\sup_{-1\le t\le t^*_N} |x^+_{N-1}(t,\xi_+,\al,\ep)-u^+_{N-1}(t,y^+_{N-1}(t,\xi_+,\al,\ep))| \le\rho
\]
(where $ t^*_N =  t^*_N(\xi_+,\al,\ep)$) for $|\xi_+|+ |\ep|$ sufficiently small and the following holds
\[\begin{array}{l}
	\sup_{-1\le t\le t^*_N} |x^+_{N-1}(t,\xi_+,\al,\ep)-u^+_{N-1}(t,y^+_{N-1}(t,\xi_+,\al,\ep))| \to 0 	\\
	\sup_{-1\le t\le t^*_N} |y^+_{N-1}(t,\xi_+,\al,\ep)-\al|\to 0			\\
\end{array}\]
as $|\xi_+|+|\ep|\to 0$, uniformly with respect to $\al$. Now, as
\[\begin{array}{l}
	y^+_{N-1}(t,\xi_+,\al,0)=y^+_N(t^*_N(\xi_+,\al,0),\xi_+,\al,0)=\al			\\
	c_{N-2} < h(u^+_{N-1}(t,\al),\al) < c_{N-1} \hbox{ for $t_{N-1}(\al)<t<t_N(\al)$,}	\\
	\frac{\partial}{\partial t} h(u^+_N(t,\al),\al)_{|t=t_{N-1}(\al)} > 2\eta
\end{array}\]
from the implicit function theorem we see that a $C^r$-function $t^*_{N-1}(\xi_+,\al,\ep)$, bounded together with its derivatives, exists such that
\[
	\lim_{(\xi_+,\ep)\to(0,0)} t^*_{N-1}(\xi_+,\al,\ep)-t_{N-1}(\al) = 0
\]
uniformly with respect to $\al$ and the following holds:
\[\begin{array}{l}
	c_{N-2}<h(x^+_{N-1}(t,\xi_+,\al,\ep),y^+_{N-1}(t,\xi_+,\al,\ep))<c_{N-1}, 			\\
	\qquad \hbox{for $t^*_{N-1}(\xi_+,\al,\ep) \le t<t^*_N(\xi_+,\al,\ep)$}		\\
	h(x^+_{N-1}(t^*_{N-1}(\xi_+,\al,\ep),\xi_+,\al,\ep),y^+_{N-1}(t^*_{N-1}(\xi_+,\al,\ep),\xi_+,\al,\ep))=c_{N-2}		\\
	\frac{\partial|}{\partial t} h(x^+_{N-1}(t,\xi_+,\al,\ep),y^+_{N-1}(t,\xi_+,\al,\ep))_{t = t^*_{N-1}(\xi_+,\al,\ep)} > \eta.
\end{array}\]
Proceeding this way we construct the solution $(x_+(t,\xi_+,\al,\ep),y_+(t,\xi_+,\al,\ep)$ with the properties stated in the Theorem. A similar argument works for $t\le 0$. The proof is complete
\end{proof}

\medskip
\begin{remark}\label{remt*}
According to Theorem \ref{extension} we have
\begin{equation}\label{inomega0}\begin{array}{l}
	h(x_+(t^*_i(\xi_+,\al,\ep),\xi_+,\al,\ep),y_+(t^*_i(\xi_+,\al,\ep),\xi_+,\al,\ep)) = c_{i-1}	\\
	h(x_-(t^*_j(\xi_-,\al,\ep),\xi_-,\al,\ep),y^*_j(t_*(\xi_-,\al,\ep),\xi_-,\al,\ep)) = c_j.
\end{array}\end{equation}
\end{remark}
Differentiating the above equalities with respect to $\xi_+$, $\xi_-$, at $\ep=0$ we obtain a formula for the derivatives
\[
\frac{\partial t^*_i}{\partial\xi_+}(\xi_+,\al,0), \quad \frac{\partial t^*_j}{\partial\xi_-}(\xi_-,\al,0).
\]
However we have to distinguish when $t\to t^*_i(\xi_+,\al,0)^+$ or $t\to t^*_i(\xi_+,\al,0)^-$ (resp. $t\to t^*_j(\xi_-,\al,0)^+$ or
$t\to t^*_j(\xi_-,\al,0)^-$). For example if $t\to t^*_i(\xi_+,\al,0)^+$, $x_+(t,\xi_+,\al,0)$ is the solution of $\dot x = f_i(x,\al)$ and then, differentiating \eqref{inomega0} with respect to $\xi_+$, we get, with $t^*_i = t^*_i(\xi_+,\al,0)$:
\[
h_x(x_+(t^*,\xi_+,\al,0),\al)[f_i(x_+(t^*_i,\xi_+,\al,0),\al)\frac{\partial t^*_i}{\partial\xi_+}(\xi_+,\al,0) + x_{+,\xi_+}(t^{*+}_i,\xi_+,\al,0)] = 0.
\]
Vice versa, when $t\to t^*_i(\xi_+,\al,0)^-$, $x_+(t,\xi_+,\al,0)$ is the solution of $\dot x = f_{i-1}(x,\al)$ and then
\[
h_x(x_+(t^*_i,\xi_+,\al,0),\al)[f_{i-1}(x_+(t^*_i,\xi_+,\al,0),\al)\frac{\partial t^*_i}{\partial\xi_+}(\xi_+,\al,0) + x_{+,\xi_+}(t^{*-}_i,\xi_+,\al,0)] = 0.
\]
Similarly we get
\[
h_x(x_-(t^*_j,\xi_-,\al,0),\al)[f_j(x_-(t^*_j,\xi_+,\al,0),\al)\frac{\partial t^*_j}{\partial\xi_-}(\xi_-,\al,0) + x_{-,\xi_-}(t^{*-}_j,\xi_-,\al,0)] = 0
\]
and
\[
h_x(x_-(t^*_j,\xi_-,\al,0),\al)[f_{j+1}(x_+(t^*_j,\xi_-,\al,0),\al)\frac{\partial t^*_j}{\partial\xi_-}(\xi_-,\al,0) + x_{-,\xi_-}(t^{*+}_j,\xi_-,\al,0)] = 0.
\]
We will use this remark in the next section.

\bigskip

\section{The variational equation}
Let $\ell \in\{-M,\ldots,-1,1,\ldots,N\}$. For fixed $\al\in\R^m$ we define linear operators $B_\ell(\al):\R^n \to \R^n$ as follows:
\begin{equation}\label{defB*}
	B_\ell(\al)x = x - \frac{h_x(u(t_\ell(\al),\al),\al)x}{h_x(u(t_\ell(\al),\al),\al)\dot u(t_\ell(\al)^-,\al)}[\dot u(t_\ell(\al)^-,\al) - \dot u(t_\ell(\al)^+,\al)].
\end{equation}
Note that (recall that $i\in\{1,\ldots,N\}$, $j\in\{-1,\ldots,-M\}$):
\[\begin{array}{l}
	h_x(u(t_i(\al),\al),\al)\dot u(t_i(\al)^-,\al) = h_x(u(t_i(\al),\al),\al)f_{i-1}(u(t_i(\al),\al),\al) \ne 0	\\
	h_x(u(t_j(\al),\al),\al)\dot u(t_j(\al)^-,\al) = h_x(u(t_j(\al),\al),\al)f_j(u(t_j(\al),\al),\al) \ne 0
\end{array}\]
\medskip

We have the following:
\begin{prop}\label{derxwrxi} For any $\ell$ and $\al\in\R^m$, $x\mapsto B_\ell(\al)x$
%and $x\mapsto B_{*,i}(\al)x$
are invertible linear maps. Moreover $x_{+,\xi_+}(t,0,\al,0)$ is a  solution of
\begin{equation}\label{var+}\begin{array}{l}
	\dot x = A(t,\al)x := \left\{\begin{array}{ll}
		f_{i-1,x}(u(t,\al),\al)x & \hbox{if $t_{i-1}(\al)< t < t_i(\al)$}		\\
%		&\quad 	i=1,\ldots,N		\\
		f_{N,x}(u(t,\al),\al)x &\hbox{if $t > t_N(\al)$}
	\end{array}\right .	\\
	x(t_i(\al)^+) = B_i(\al)x(t_i(\al)^-), %\quad i=1,\ldots,N	
\end{array}\end{equation}
which is $C^1$ for $t\ne t_i(\al)$, bounded for $t\ge 0$ and can be assumed to be right-continuous at $t=t_i(\al)$. Similarly
$x_{-,\xi_-}(t,0,\al,0)$ is a  solution of
\begin{equation}\label{var-}\begin{array}{l}
	\dot x = A(t,\al)x := \left\{\begin{array}{ll}
		f_{j+1,x}(u(t,\al),\al)x & \hbox{if $t_j(\al)<t< t_{j+1}(\al)$}		\\
%		&\quad 	j=-1,\ldots,-M	\\
		f_{-M,x}(u(t,\al),\al)x & \hbox{if $t<t_{-M}(\al)$}
	\end{array}\right .	\\
	x(t_j(\al)^+) = B_j(\al)x(t_j(\al)^-), %\quad j=-1,\ldots,-M
\end{array}\end{equation}
which is $C^1$ for $t\ne t_j(\al)$,bounded for $t\le 0$ and can be assumed to be left-continuous at $t=t_j(\al)$.
\end{prop}
\begin{proof}
First we prove that $B_\ell(\al):\R^n \to \R^n$ is invertible. If $\dot u(t_\ell(\al)^-,\al) = \dot u(t_\ell(\al)^+,\al)$, there is nothing to prove since $B_\ell(\al)x=x$. So, suppose that $\dot u(t_\ell(\al)^-,\al) \ne \dot u(t_\ell(\al)^+,\al)$ and $B_\ell(\al)x=0$. Then
\[
	x = \mu[\dot u(t_\ell(\al)^-,\al) - \dot u(t_\ell(\al)^+,\al)]	
\]
where $\mu = \frac{h_x(u(t_\ell(\al),\al),\al)x}{h_x(u(t_\ell(\al),\al),\al)\dot u(t_\ell(\al)^-,\al)}$. Hence:
\[\begin{array}{l}
	B_\ell(\al)x= \mu[\dot u(t_\ell(\al)^-,\al) - \dot u(t_\ell(\al)^+,\al)]	\\
	\qquad - \mu\frac{h_x(u(t_\ell(\al),\al),\al)[\dot u(t_\ell(\al)^-,\al) - \dot u(t_\ell(\al)^+,\al)]}{h_x(u(t_i(\al),\al),\al)\dot u(t_\ell(\al)^-,\al)}
	[\dot u(t_\ell(\al)^-,\al) - \dot u(t_\ell(\al)^+,\al)]	\\
	= \mu\frac{h_x(u(t_\ell(\al),\al),\al) \dot u(t_\ell(\al)^+,\al)}{h_x(u(t_\ell(\al),\al),\al)\dot u(t_\ell(\al)^-,\al)}[\dot u(t_\ell(\al)^-,\al) -
	\dot u(t_\ell(\al)^+,\al)].
\end{array}\]
But then $\mu=0$ since
\[\begin{array}{l}
h_x(u(t_\ell(\al),\al),\al)\dot u(t_\ell(\al)^+,\al) 	\\
	= \left \{ \begin{array}{ll}
		h_x(u(t_i(\al),\al),\al) f_i(u(t_i(\al),\al),\al) & \hbox{if $\ell=i>0$}	\\
		h_x(u(t_j(\al),\al),\al) f_{j+1}(u(t_j(\al),\al),\al) & \hbox{if $\ell=j<0$}
 	\end{array}\right . \ne 0.
\end{array}\]

Next, $x_+(t,\xi_+,\al,0)$ is a continuous, piecewise $C^r$, solution of the differential equation
\[
	\dot x  = \left\{\begin{array}{ll}
		f_{i-1}(x,\al) & \hbox{if $t^*_{i-1}(\xi_+,\al,0)< t< t^*_i(\xi_+,\al,0)$}	\\
%			& i=1,\ldots,N		\\
		f_N(x,\al) & \hbox{if $t>t^*_N(\xi_+,\al,0)$.}
	\end{array}\right .
\]
Hence, for $t^*_i(\xi_+,\al,0)<t<t^*_{i+1}(\xi_+,\al,0)$, when $i\le N-1$, or $t\ge t^*_N(\xi_+,\al,0)$ when $i=N$, we have
\[
	x_+(t,\xi_+,\al,0) = x_+(t^*_i(\xi_+,\al,0)^-,\xi_+,\al,0) + \int_{t^*_i(\xi_+,\al,0)}^t f_i(x_+(s,\xi_+,\al,0),\al)ds.
\]
Differentiating with respect to $\xi_+$ we get, for the same values of $t$:
\[\begin{array}{l}
	x_{+,\xi_+}(t,\xi_+,\al,0) = x_{+,\xi_+}(t^*_i(\xi_+,\al,0)^-,\xi_+,\al,0)		\\
	\dis [\dot x_+(t^*_i(\xi_+,\al,0)^-,\xi_+,\al,0) - \dot x_+(t^*_i(\xi_+,\al,0)^+,\xi_+,\al,0)]\frac{\partial t^*_i}{\partial\xi_+}(\xi_+,\al,0) 	\\
	\dis + \int_{t_i^*(\xi_+,\al,0)}^t f_{i,x}(x_+(s,\xi_+,\al,0),\al)x_{\xi_+}(s,\xi_+,\al,0)ds.
\end{array}\]
and then,
\begin{equation}\label{matrix2}
\begin{array}{l}
	x_{+,\xi_+}(t^*_i(\xi_+,\al,0)^+,\xi_+,\al,0) = x_{+,\xi_+}(t^*_i(\xi_+,\al,0)^-,\xi_+,\al,0) +	\\
	\dis + [f_{i-1}(\hat x^*_i(\xi_+,\al),\al) - f_i(\hat x^*_i(\xi_+,\al),\al)] \frac{\partial t^*_i}{\partial\xi_+}(\xi_+,\al,0)
\end{array}
\end{equation}
where we write for simplicity
\begin{equation}\label{defxhat}
	\hat x^*_i(\xi_+,\al) := x_+(t^*_i(\xi_+,\al,0),\xi_+,\al,0).
\end{equation}

Now, from Remark \ref{remt*} we see that
\[
	\frac{\partial t^*_i}{\partial\xi_+}(\xi_+,\al,0) =
	-\frac{h_x(\hat x^*_i(\xi_+,\al),\al)x_{+,\xi_+}(t^*(\xi_+,\al,0)^-,\al,0)}{h_x(\hat x^*_i(\xi_+,\al),\al) f_{i-1}(\hat x^*_i(\xi_+,\al),\al)}.
\]
Hence
\begin{equation}\label{jump2}
	x_{\xi_+}(t^*_i(\xi_+,\al,0)^+,\xi_+,\al,0) = B_i(\xi_+,\al)x_{\xi_+}(t^*_i(\xi_+,\al,0)^-,\xi_+,\al,0)
\end{equation}
where
\[\begin{array}{l}
	B_i(\xi_+,\al)x = 	\\
	x - [f_{i-1}(\hat x^*(\xi_+,\al),\al) - f_i(\hat x^*(\xi_+,\al),\al)]
	\frac{h_x(\hat x^*_i(\xi_+,\al),\al)x}{h_x(\hat x^*_i(\xi_+,\al),\al)f_{i-1}(\hat x^*_i(\xi_+,\al),\al)}.
\end{array}\]
Taking $\xi_+=0$, we see that $x_{\xi_+}(t,0,\al,0)$ is a solution, for $t\ge 0$ of \eqref{var+} where $B_i(\al)$ is as in \eqref{defB*}.

Following a similar argument we see that $x_{\xi_-}(t,0,\al,0)$ is a solution, for $t\le 0$ of \eqref{var-} where $B_j(\al)$ is as in \eqref{defB*}.

Finally we prove that $x_{\xi_+}(t,0,\al,0)$ is bounded for $t\ge 0$. It is enough to prove this for $t\ge T_+$. From  Theorem \ref{bddsol} we know that $x_{\xi_+}(t,0,\al,0)$ is a solution of
\[
\dot x = f_{N,x}(u_N^+(t,\al),\al)x
\]
such that
\[
	\sup_{t\ge T_+} |x(t)|e^{-\beta(t-T_+)}<\infty.
\]
So, $x_{\xi_+}(t,0,\al,0)e^{-\be(t-T_+)}$ is a bounded solution of the linear system
\begin{equation}\label{shiftbe}
\dot x = [f_{N,x}(u_N^+(t,\al),\al)-\be\I]x
\end{equation}
whose fundamental matrix on $t\ge T_+$ is $U^+_N(t,\al)U^+_N(T_+,\al)^{-1}e^{-\be (t-T_+)}$. According to Lemma \ref{EDshift} \eqref{shiftbe} has an exponential dichotomy on $t\ge T_+$ with projection $P_+(\al)$ and exponent $\de-\be$. Then we have $x_{\xi_+}(T_+,0,\al,0)\in {\cal R}P_+(\al)$. But then $x_{\xi_+}(t,0,\al,0)$ is bounded for $t\ge T_+$ because ${\cal R}P_+(\al)$ is the space of initial conditions of solutions of $\dot x = f_{N,x}(u_N^+(t,\al),\al)x$ that are bounded for $t\ge T_+$. A similar argument shows that $x_{\xi_-}(t,0,\al,0)$ is bounded for $t\le 0$. The proof is complete.
\end{proof}
In the next proposition we show that $\dot u(t,\al)$ is a nontrivial bounded solution of \eqref{var+} for $t\ge 0$ (resp. \eqref{var-} for $t\le 0$).
\begin{prop}\label{propudot} For $t\ge 0$, resp. $t\le 0$, the function
\[
	\dot u(t,\al) = \left \{
	\begin{array}{ll}
		\dot u_{i-1}^+(t,\al) & \hbox{for $t_{i-1}(\al)<t<t_i(\al)$}		\\
		\dot u^+_N(t,\al) & \hbox{for $t > t_N(\al)$,}
	\end{array}\right .
\]
resp.
\[
	\dot u(t,\al) = \left \{
	\begin{array}{ll}
		\dot u_{j+1}^-(t,\al) & \hbox{for $t_j(\al)<t<t_{j+1}\al)$}		\\
		\dot u^-_{-M}(t,\al) & \hbox{for $t < t_{-M}(\al)$,}
	\end{array}\right .
\]
is a solution of \eqref{var+} (resp. \eqref{var-}) bounded on $t\ge 0$ (resp, $t\le 0$) where $B_\ell(\al)$ is as in \eqref{defB*}.
\end{prop}
\begin{proof} We already know that $\dot u(t)$ satisfies \eqref{var+}, for $t\ge 0$ and \eqref{var-} for $t\le 0$,
$t\ne t_\ell$. We prove that $\dot u(t_\ell(\al)^+,\al) = B_\ell(\al)\dot u(t_\ell(\al)^-,\al)$. We have
\[\begin{array}{l}
	B_\ell(\al)\dot u(t_\ell(\al)^-,\al) = \dot u(t_\ell(\al)^-,\al)		\\
	\qquad - \frac{h_x(u(t_\ell(\al),\al),\al)\dot u(t_\ell(\al)^-,\al)}{h_x(u(t_\ell(\al),\al),\al)\dot u(t_\ell(\al)^-,\al)}
	[\dot u(t_\ell(\al)^-,\al) - \dot u(t_\ell(\al)^+,\al)]		\\
	= \dot u(t_\ell(\al)^+,\al)
\end{array}\]
The proof is complete.
\end{proof}
\section{The Melnikov condition}
First we recall that $P_+(y)$ is the projections of the exponential dichotomy on $t\ge 0$, of the linear system \eqref{variat+} with constant $K$ and exponent $\delta$. Then, from Lemma \ref{EDshift}, we see that \eqref{variat1} has an exponential dichotomy on $t\ge t_N(y)$ with exponent $\de$ and projection
\begin{equation}\label{pro+}\begin{array}{c}
	U^+_N(t_N(y),y)U^+_N(T_+,y)^{-1}P_+(y)U^+_N(T_+,y)U^+_N(t_N(y),y)^{-1}	\\
	= X_+(t_N(y),y)X_+(T_+,y)^{-1}P_+(y)X_+(T_+,y)X_+(t_N(y),y)^{-1}
\end{array}\end{equation}
the equality following from \eqref{XU+} and $T_+>t_N(y)$.

Similarly, the linear system \eqref{variat2} has an exponential dichotomy on $t\le t_{-M}(y)$ with exponent $\delta$ and projection
\begin{equation}\label{pro-}\begin{array}{c}
	U^-_{-M}(t_{-M}(y),y)U^-_{-M}(T_-,y)^{-1}P_-(y)U^-_{-M}(T_-,y)U^-_{-M}(t_{-M}(y),y)^{-1}	\\
	= X_-(t_{-M}(y),y)X_-(T_-,y)^{-1}P_-(y)X_-(T_-,y)X_-(t_{-M}(y),y)^{-1}
\end{array}\end{equation}
where
\[
X_+(t,y) = \left \{\begin{array}{ll}
	U_0(t,y)	& \hbox{if $0\le t< t_1(y)$}	\\
	U_i(t,y)U_i(t_i,y)^{-1}B_i(y)X_+(t_i^-,y)	& \hbox{if $t_i(y)\le t< t_{i+1}(y)$}	\\
	U_N(t,y)U_N(t_N,y)^{-1}B_N(y)X_+(t_N^-,y) & \hbox{if $t\ge t_N(y)$}
\end{array}\right .
\]
and
\[
X_-(t,y) = \left \{\begin{array}{ll}
	U_0(t,y) & \hbox{if $t_{-1}(y)<t\le 0$}		\\
	U_j(t,y)U_j(t_j,y)^{-1}B_j^{-1}(y)X_-(t_j^+,y) & \hbox{if $t_{j-1}(y)<t\le t_j(y)$}	\\
	U_{-M}(t,y)U_{-M}(t_{-M},y)^{-1}B_{-M}(y)^{-1}X_-(t_{-M}^+,y) & \hbox{if $t\le t_{-M}(y)$}
\end{array}\right .
\]
are the fundamental matrix of $\dot x = A(t,y)x$, where $A(t,y)$ is as in \eqref{defA}.

From Lemma \ref{extend}--\ref{bddlem} and \eqref{pro+}-\eqref{pro-} we obtain the following
\begin{prop}\label{basicprop} For any $\al\in\R^m$, the discontinuous linear system \eqref{var+} (resp. \eqref{var-})
has an exponential dichotomy on $\R_+$, (resp. $\R_-$) with projections $Q_+(\al)$, resp. $Q_-(\al)$, given by
\[\begin{array}{l}
	Q_+(\al) = X_+(T_+,\al)^{-1} P_+(\al)X_+(T_+,\al)	\\	\\
	Q_-(\al) = 	X_-(T_-,\al)^{-1}P_-(\al)X_-(T_-,\al).
\end{array}\]
Moreover ${\cal R}Q_+(\al)$ (resp. ${\cal N}Q_-(\al)$) is the space of initial conditions of solutions of \eqref{var+}, resp. \eqref{var-},
right-continuous, when $t\ge 0$ (resp. left-continuous, when $t\le 0$) and bounded on $\R_+$, (resp, on $\R_-$).
\end{prop}

For simplicity we write $Q_\pm = Q_\pm(y_0)$.

We assume the following condition holds:
\begin{itemize}
	\item[$A_5)$] $\dim {\cal R}Q_+ \cap {\cal N}Q_- = d\le m$.
\end{itemize}
From Proposition \ref{propudot} we see that $\dot u(0,y_0)\in{\cal R}Q_+\cap{\cal N}Q_-$ so
\[
1\le \dim [{\cal R}Q_+ + {\cal N}Q_-]^\perp = d.
\]
Next, from $A_3)$ we know that $\dim{\cal R}Q_+ = k$ and $\dim{\cal N}Q_-=n-k$, hence $d\le \min\{k,n-k\}$.

Let $\psi_1,\ldots,\psi_d\in\R^n$ be such that $[{\cal R}Q_+ + {\cal N}Q_-]^\perp =\span\{\psi_1,\ldots,\psi_d\}$. Without loss of generality we assume that $(\psi_1,\ldots,\psi_d)$ is an orthonormal set.

The purpose of this section is the to prove the following
\begin{theorem}\label{main}
Suppose that $A_1)-A_5)$ hold. Suppose further that the matrix $[\psi_j^T[w^-_{0,y}(y_0)- w^+_{0,y}(y_0)]_{j=1,\ldots,d}$ has rank $d$.
Then there  exists $\rho>0$ and $\ep_0>0$ such that for $0\le \ep\le\ep_0$ system \eqref{pert} has a $(m-d)$-dimensional manifold of continuous,
piecewise $C^r$ solutions $(x(t,\ep),y(t,\ep))$ such that
\[\begin{array}{l}
	\sup_{t\in\R}|x(t,\ep) - u(t,y(t,\ep)) <\rho,	\\
	\dis \sup_{t\in\R}|x(t,\ep) - u(t,y(t,\ep))| \to 0
\end{array}\]
as $\ep\to 0$.
\end{theorem}
\begin{proof} First, we apply Lemma \ref{prelem} to obtain another expression of $x_+(t,\xi_+,\al,\ep)$ (resp. $x_-(t,\xi_-,\al,\ep)$).  We know that, for $t\ge 0$,
\[
	z_+(t) = x_+(t,\xi_+,\al,\ep) - u(t,y_+(t,\xi_+,\al,\ep))
\]
is a bounded and continuous solution of the differential equation
\[\begin{array}{l}
	\dot z = f(x_+(t,\xi_+,\al,\ep),y_+(t,\xi_+,\al,\ep)) - f(u(t,y_+(t,\xi_+,\al,\ep)),y_+(t,\xi_+,\al,\ep)) 	\\
	\hskip 20pt - \ep u_y(t,y_+(t,\xi_+,\al,\ep))g(x_+(t,\xi_+,\al,\ep),y_+(t,\xi_+,\al,\ep))
\end{array}\]
that we write:
\[\left\{\begin{array}{l}
	\dot z - A(t,\al)z = b_+(t) 	\\
	z(t_i(\al)^+) = z(t_i(\al)^-), \quad i=1,\ldots ,N
\end{array}\right .\]
where $A(t,y)$ has been defined in \eqref{defA} and
\begin{equation}\label{defbt}\begin{array}{l}
	b_+(t) = f(x_+(t,\xi_+,\al,\ep),y_+(t,\xi_+,\al,\ep)) - f(u(t,y_+(t,\xi_+,\al,\ep)),y_+(t,\xi_+,\al,\ep)) 	\\
	\hskip 20pt  - A(t,\al)x_+(t,\xi_+,\al,\ep) - \ep u_y(t,y_+(t,\xi_+,\al,\ep))g(x_+(t,\xi_+,\al,\ep),y_+(t,\xi_+,\al,\ep)).
\end{array}\end{equation}
Note that
\[
(u(t,y_+(t,\xi_+,\al,0)),y_+(t,\xi_+,\al,0)) = (u(t,\al),\al)
\]
and, according to $a_2)$,
\[\begin{array}{l}
	h_x(u(t_i(\al),\al) f_{i-1}(u(t_i(\al),\al),\al) > 2\eta , 	\\
	h_x(u(t_i(\al),\al) f_i((u(t_i(\al),\al) > 2\eta.
\end{array}\]
Hence, for $\ep$ sufficiently small there exist $C^r$-functions $\tilde t_i(\xi_+,\al,\ep)$ such that
\[\begin{array}{l}
	|\tilde t_i(\xi_+,\al,\ep)-t_i(\al)|\to 0, \hbox{ as $\ep\to 0$}	\\
	h(u(\tilde t_i(\xi_+,\al,\ep),y_+(\tilde t_i(\xi_+,\al,\ep),\xi_+,\al,\ep)),y_+(\tilde t_i(\xi_+,\al,\ep),\xi_+,\al,\ep)) = c_i	\\
	c_{i-1}< h(u(t,y_+(t,\xi_+,\al,\ep)),y_+(t,\xi_+,\al,\ep))<c_i, \\
	\qquad \hbox{ for $\tilde t_{i-1}(\xi_+,\al,\ep)<t<\tilde t_i(\xi_+,\al,\ep)$} 	\\
	h_x(x,y)f_{i-1}(x,y)_{\vrule_{x=u(\tilde t_i(\xi_+,\al,\ep),y_+(\tilde t_i(\xi_+,\al,\ep),\xi_+,\al,\ep))\atop
	y = y_+(\tilde t_i(\xi_+,\al,\ep)\hfill }} >\eta 	\\
	h_x(x,y)f_i(x,y)_{\vrule_{x=u(\tilde t_i(\xi_+,\al,\ep), y_+(\tilde t_i(\xi_+,\al,\ep),\xi_+,\al,\ep))\atop
	y = y_+(\tilde t_i(\xi_+,\al,\ep)\hfill }} >\eta
\end{array}\]
for any $i=1,\ldots,N$ and uniformly with respect to $(\xi_+,\al)$. Then
\[\begin{array}{l}
	f(x_+(t,\xi_+,\al,\ep),y_+(t,\xi_+,\al,\ep)) 	\\
	= \left \{\begin{array}{ll}
		f_{i-1}(x_+(t,\xi_+,\al,\ep),y_+(t,\xi_+,\al,\ep)) & \hbox{if $t^*_{i-1}(\xi_+,\al,\ep)\le t<t^*_i(\xi_+,\al,\ep)$}	\\
		f_N(x_+(t,\xi_+,\al,\ep),y_+(t,\xi_+,\al,\ep)) & \hbox{if $t\ge t^*_N(\xi_+,\al,\ep)$}
	\end{array}\right .
\end{array}\]
and
\[\begin{array}{l}
	f(u(t,\xi_+,\al,\ep),y_+(t,\xi_+,\al,\ep)) 	\\
	= \left \{\begin{array}{ll}
		f_{i-1}(u(t,\xi_+,\al,\ep),y_+(t,\xi_+,\al,\ep))	& \hbox{if $\tilde t_{i-1}(\xi_+,\al,\ep)\le t<\tilde t_i(\xi_+,\al,\ep)$}	\\
		f_N(u(t,\xi_+,\al,\ep),y_+(t,\xi_+,\al,\ep))	& \hbox{if $t\ge \tilde t_N(\xi_+,\al,\ep)$.}
	\end{array}\right .
\end{array}\]
According to Lemma \ref{prelem}, with $\tau=0$, we see that
\begin{equation}\label{express1}
	\begin{array}{l}
		x_+(t,\xi_+,\al,\ep) = u(t,y_+(t,\xi_+,\al,\ep)) + X_+(t,\al)\tilde\xi_+ 	\\
		\qquad \dis + \int_0^t X_+(t,\al)Q_+(\al)X_+(s,\al)^{-1}b_+(s)ds 		\\
		\qquad \dis - \int_t^\infty X_+(t,\al)(\I-Q_+(\al))X_+(s,\al)^{-1}b_+(s)ds
	\end{array}
\end{equation}
where
\[
\tilde\xi_+ = Q_+(\al)[x_+(0,\xi_+,\al,\ep) - u(0^+,y_+(0,\xi_+,\al,\ep)) ] \in {\cal R}Q_+(\al).
\]
Note that
\[
	|y_+(0,\xi_+,\al,\ep) -\al| \le |\ep| \int_0^{T_+} g(x_+(t,\xi_+,\al,\ep),y_+(t,\xi_+,\al,\ep),\ep)|dt \to 0
\]
as $\ep\to 0$, uniformly with respect to $(\xi_+,\al)$. We prove the following
\medskip

\noindent {\bf Claim}: For $\ep$ sufficiently small, the map $(\xi_+,\al)\mapsto (\tilde\xi_+,\tilde\al)$, with $\tilde\al=y_+(0,\xi_+,\al,\ep)$ from
${\cal R}P_+(\al)\times\R^m$ into ${\cal R}Q_+(\al)\times\R^m$ is linearly invertible.
\medskip

Indeed, for $\ep=0$ the above map reduces to
\[
	(\xi_+,\al)\mapsto (\tilde\xi_+,\al)
\]
where
\[
	\tilde\xi_+ = Q_+(\al)[x_+(0,\xi_+,\al,0) - u(0^+,\al)].
\]
Hence
\[
	\frac{\partial\tilde\xi_+}{\partial\xi_+}(0,\al,0) = Q_+(\al)x_{+,\xi_+}(0,0,\al,0).
\]
Now, from Proposition \ref{derxwrxi} we know that, for any $\xi_+\in {\cal R}P_+(\al)$,
\[
z(t):=x_{+,\xi_+}(t,0,\al,0)\xi_+
\]
is a right-continuous solution, bounded on $t\ge 0$, of
\[\begin{array}{l}
	\dot z(t) = A(t,\al)z(t)		\\
	z(t_i(\al)^+) = B_i(\al)z(t_i^-(\al))	\\
	P_+(\al)z(T_+)=\xi_+
\end{array}\]
the last relation following differentiating the equality
\[
	P_+(\al)[x_+(T_+,\xi_+,\al,0)-u(T_+,\al)]=\xi_+
\]
with respect to $\xi_+$ at $\xi_+=0$ (see Theorem \eqref{bddsol}). In particular we have $z(t) = X_+(t,\al)X_+(\tau,\al)^{-1}z(\tau)$ for any $\tau\ge 0$ (see Remark \ref{rem3.3}). From Lemma \ref{EDshift} it follows that the linear system \eqref{variat1}, with $y=\al$, has an exponential dichotomy on $t\ge t_N(\al)$ with projection (see also \eqref{XU+})
\[\begin{array}{l}
	{\cal P}_+(\al) = U^+_N(t_N(\al),\al)U^+_N(T_+,\al)^{-1}P_+(\al)U^+_N(T_+,\al)U^+_N(t_N(\al),\al)^{-1}	\\
	= X_+(t_N(\al),\al)X_+(T_+,\al)^{-1}P_+(\al)X_+(T_+,\al)X_+(t_N(\al),\al)^{-1}
\end{array}\]
as $T_+<t_N(y)$. Hence
\[\begin{array}{l}
	X_+(T_+,\al)X_+(t_N(\al),\al)^{-1}{\cal P}_+(\al) = P_+(\al)X_+(T_+,\al)X_+(t_N(\al),\al)^{-1}.
\end{array}\]
As a consequence
\[\begin{array}{l}
	X_+(T_+,\al)X_+(t_N(\al),\al)^{-1}{\cal P}_+(\al)z(t_N(\al))		\\
	=  P_+(\al)X_+(T_+,\al)X_+(t_N(\al),\al)^{-1}z(t_N(\al)) = P_+(\al)z(T_+) = \xi_+ .
\end{array}\]
So, $z(t):=x_{+,\xi_+}(t,0,\al,0)\xi_+$ is a right-continuous solution, bounded on $t\ge 0$
\[\begin{array}{l}
	\dot z(t) = A(t,\al)z(t)		\\
	z(t_i(\al)^+) = B_i(\al)z(t_i^-(\al))	\\
	{\cal P}_+(\al)z(t_N(\al)) = X_+(t_N(\al),\al)X_+(T_+,\al)^{-1}\xi_+
\end{array}\]
for any $\xi_+\in {\cal R}P_+(\al)$. Now, with reference to Lemmas \ref{extend}, \ref{prelem} with $\tau=t_N(\al)$, we have
\[\begin{array}{l}
	\tilde {\cal P}_+^{t_N(\al)}(\al) = X_+(t_N(\al),\al)Q_+(\al) X_+(t_N(\al),\al)^{-1} 	\\
	= X_+(t_N(\al),\al)X_+(T_+,\al)^{-1} P_+(\al)X_+(T_+,\al)X_+(t_N(\al),\al)^{-1}	\\
	= {\cal P}_+(\al).
\end{array}\]
Hence $x_{+,\xi_+}(t,0,y_0,0)$ is a bounded solution of
\[\begin{array}{l}
	\dot z(t) = A(t,y_0)z(t)		\\
	z(t_i(\al)^+) = B_i(\al)z(t_i^-(\al))			\\
	\tilde {\cal P}_+^{t_N(\al)}(\al)z(t_N(\al)) = X_+(t_N(\al),\al)X_+(T_+,\al)^{-1}\xi_+.
\end{array}\]
From Lemma \ref{prelem}, equation \eqref{xpos}, we get then, for any $\xi_+\in{\cal R}P_+(\al)$, using again the right-continuity of $X_+(t,\al)$:
\[\begin{array}{l}
	\dis x_{+,\xi_+}(t,0,\al,0)\xi_+ = X_+(t,\al)Q_+(\al)X_+(t_N(\al),\al)^{-1}X_+(t_N(\al),\al)X_+(T_+,\al)^{-1}\xi_+	\\
	\dis = X_+(t,\al)Q_+(\al)X_+(T_+,\al)^{-1}\xi_+
\end{array}\]

So, %using $U^+_N(T_+,\al)U^+_N(t_N(\al),\al)^{-1} = X_+(T_+,\al)X_+(t_N(\al),\al)^{-1}$:
\[
	\frac{\partial\tilde\xi_+}{\partial\xi_+}(0,\al,0) = Q_+(\al)x_{+,\xi_+}(0,0,\al,0) = Q_+(\al)X_+(T_+,\al)^{-1}\xi_+
	= X_+(T_+,\al)^{-1} P_+(\al)\xi_+ .
\]
Hence $\frac{\partial\tilde\xi_+}{\partial\xi_+}(0,\al,0)$ is an isomorphism from ${\cal R}P_+(\al)$ into ${\cal R}Q_+(\al)$ and the Claim is proved.
\medskip

Similarly we see that
\begin{equation}\label{express2}
	\begin{array}{l}
		x_-(t,\xi_-,\al,\ep) = u(t,y_-(t,\xi_-,\al_-,\ep)) + X_-(t,\al)\tilde\xi_- 	\\
		\quad \dis + \int_{-\infty}^t X_-(t,\al)Q_-(\al)X_-(s,\al)^{-1}b_-(s)ds 	\\
		\quad \dis - \int_t^0 X_-(t,\al)(\I-Q_-(\al))X_-(s,\al)^{-1}b_-(s)ds
	\end{array}
\end{equation}
where
\[\begin{array}{l}
	b_-(t) = f(x_-(t,\xi_-,\al,\ep),y_-(t,\xi_-,\al,\ep)) - f(u_0(t,y_-(t,\xi_-,\al,\ep)),y_-(t,\xi_-,\al,\ep)) 	\\
	\hskip 20pt  - {\cal A}(t,\al)x_-(t,\xi_-,\al,\ep) - \ep u_{0,y}(t,y_-(t,\xi_-,\al,\ep))g(x_-(t,\xi_-,\al,\ep),y_-(t,\xi_-,\al,\ep))
\end{array}\]
and
\[
\tilde\xi_- = [\I-Q_-(\al)][x_-(0,\xi_-,\al,\ep) - u_0(0^-,y_-(0,\xi_-,\al,\ep)) ] \in {\cal N}Q_-(\al).
\]
Again we see that
\[
	|y_-(0,\xi_-,\al,\ep) -\al| \le |\ep| \int_0^1 g(x_-(t,\xi,\al_-,\ep),y_-(t,\xi_-,\al,\ep),\ep)|dt \to 0
\]
as $\ep\to 0$, uniformly with respect to $(\xi_-,\al)$ and the map $(\xi_-,\al)\mapsto (\tilde\xi_-,y_-(0,\xi_-,\al,\ep))$
from ${\cal N}P_-(\al)\times\R^m$ into ${\cal N}Q_-(\al)\times\R^m$ is linearly invertible.
\medskip

From \eqref{express1}-\eqref{express2} we get, for $|\al_+-y_0|+|\al_--y_0|$ sufficiently small
\begin{equation}\label{Bif}\begin{array}{l}
	x_+(0,\xi_+,\al_+,\ep) - x_-(0,\xi_-,\al_-,\ep)	\\
	= u_0(0^+,y_+(0,\xi_+,\al_+,\ep))-u_0(0^-,y_-(0,\xi_-,\al_-,\ep)) + \tilde\xi_+ -\tilde\xi_- 	\\
	\dis - \int_0^\infty (\I-Q_+(\al_+))X_+(s,\al_+)^{-1}b_+(s)ds - \int_{-\infty}^0 Q_-(\al_-)X_-(s,\al_-)^{-1}b_-(s)ds.
\end{array}\end{equation}
Hence the system
\[\begin{array}{l}
	x_+(0,\xi_+,\al_+,\ep) = x_-(0,\xi_-,\al_-,\ep)	\\
	y_+(0,\xi_+,\al_+,\ep) = y_-(0,\xi_-,\al_-,\ep)
\end{array}\]
is equivalent to
\begin{equation}\label{Xbifeq}\left \{\begin{array}{l}
	 \tilde\xi_+ -\tilde\xi_- = u(0^-,y_-(0,\xi_-,\al_-,\ep))-u(0^+,y_+(0,\xi_+,\al_+,\ep))		\\
	 \dis \qquad + \int_0^\infty (\I-Q_+(\al_+))X_+(s,\al_+)^{-1}b_+(s)ds + \int_{-\infty}^0 Q_-(\al_-)X_-(s,\al_-)^{-1}b_-(s)ds	\\
	 y_+(0,\xi_+,\al_+,\ep) - y_-(0,\xi_-,\al_-,\ep) = 0.
\end{array}\right .\end{equation}

Let
\[\begin{array}{l}
	\dis k(\xi_+,\xi_-,\al_+,\al_-,\ep) = \int_0^\infty (\I-Q_+(\al_+))X_+(s,\al_+)^{-1}b_+(s)ds 	\\
	\qquad \dis+ \int_{-\infty}^0 Q_-(\al_-)X_-(s,\al_-)^{-1}b_-(s)ds.	
\end{array}\]
Differentiating $b_+(t)=b_+(t,\xi_+,\al_+,\ep)$ with respect to $\xi_+$ at $\xi_+=0$, $\ep=0$ and using $x_+(t,0,\al_+,0)=u(t,\al_+)$,
$y_+(t,0,\al_+,0)=\al_+$, we see that, for  $t_{i-1}(\al_+)<t<t_i(\al_+)$, we have
\[\begin{array}{l}
\frac{\partial b_+}{\partial\xi_+}(t,0,\al_+,0) = [f_{i-1,x}(u(t,\al_+),\al_+) - A(t,\al_+) ]x_{+,\xi_+}(t,0,\al_+,0) = 0.
\end{array}\]
and for $t>t_N(\al_+)$:
\[\begin{array}{l}
	\frac{\partial b_+}{\partial\xi_+}(t,\xi_+,\al_+,0) = [f_{N,x}(u(t,\al_+),\al_+) -  A(t,\al_+) ]x_{+,\xi_+}(t,0,\al_+,0) = 0.
\end{array}\]
%It is also obvious that
%\[
%\frac{\partial}{\partial\xi_-} b_+(s,\xi_+,\al_+,0) = \frac{\partial}{\partial\xi_+}b_-(s,\xi_-,\al_-,0) = 0.
%\]
Then
\[
	\frac{\partial}{\partial\xi_+} \int_0^\infty (\I-P_+)X_+(s)^{-1}b_+(s)ds = 0
\]
and similarly
\[
	\dis \frac{\partial}{\partial\xi_-} \int_{-\infty}^0 P_-X_-(s)^{-1} b_-(s)ds = 0.
\]
As a consequence, on account of $u(0^\pm,y)=w^\pm_0(y)$, \eqref{Xbifeq} reads:
\[\begin{array}{l}
	\tilde\xi_+ - \tilde\xi_- = w^-_0(\al_-)-w^+_0(\al_+) + R_1(\tilde\xi_+, \tilde\xi_-,\al_+,\al_-,\ep)	\\
	\al_+ = \al_- + R_2(\tilde\xi_+, \tilde\xi_-,\al_+,\al_-,\ep)
\end{array}\]
where $R_1(\tilde\xi_+, \tilde\xi_-,\al_+,\al_-,\ep)=O(|\xi_+|^2 +|\xi_-|^2 +|\ep|)$ and $R_2(\tilde\xi_+, \tilde\xi_-,\al_+,\al_-,\ep)=O(|\ep|)$, uniformly wth respect to $(\xi_+,\xi_-,\al_+,\al_-)$. Since $(\xi_+,\al_+)\mapsto (\tilde\xi_+,\tilde\al_+)$ and $(\xi_-,\al_-)\mapsto (\tilde\xi_-,\tilde\al_-)$ are linearly invertible we see that $|\xi_\pm|= O(|\tilde\xi_\pm|)$ and hence \eqref{Xbifeq} reads:
\begin{equation}\label{bifeq2a}\begin{array}{l}
	\tilde\xi_+ - \tilde\xi_- = w^-_0(\al_-)-w^+_0(\al_+) + \tilde R_1(\tilde\xi_+, \tilde\xi_-,\al_+,\al_-,\ep)	\\
	\al_+=\al_-+ \tilde R_2(\tilde\xi_+, \tilde\xi_-,\al_+,\al_-,\ep)
\end{array}\end{equation}
where $\tilde R_1(\tilde\xi_+, \tilde\xi_-,\al_+,\al_-,\ep)=O(|\tilde\xi_+|^2 +|\tilde \xi_-|^2 +|\ep|)$ and
$\tilde R_1(\tilde\xi_+, \tilde\xi_-,\al_+,\al_-,\ep)=O(|\ep|)$, uniformly wth respect to $(\tilde\xi_+,\tilde\xi_-,\al_+,\al_-)$. Now we can write
$\tilde\xi_+=Q_+(\al)\tilde\xi_+ = Q_+\tilde\xi_+ + (Q_+(\al)-Q_+)\tilde\xi_+$ and hence
\[
\frac{1}{2}|\tilde\xi_+|\le |Q_+\tilde\xi_+|\le 2|\tilde\xi_+|
\]
provided $|\al_+ - y_0|$ is sufficiently small. Similarly, for $|\al_- - y_0|$ is sufficiently small, $\frac{1}{2}|\tilde\xi_-|\le |(\I-Q_-)\tilde\xi_+|\le 2|\tilde\xi_-|$. In particular the map $\tilde\xi_+\mapsto Q_+\tilde\xi_+$ from ${\cal R}Q_+(\al)$ into ${\cal R}Q_+$, and $\tilde\xi_+\mapsto (\I-Q_-)\tilde\xi_+$ from
${\cal N}Q_+(\al)$ into ${\cal N}Q_+$ are linearly invertible. Then, setting
\[
	\bar\xi_+=Q_+\tilde\xi_+, \quad \bar\xi_-=(\I-Q_-)\tilde\xi_-,
\]
\eqref{bifeq2a} can be written as
\begin{equation}\label{bifeq2}\begin{array}{l}
	\bar\xi_+ - \bar\xi_- = w^-_0(\al_-)-w^+_0(\al_+) + \bar R_1(\bar\xi_+, \bar\xi_-,\al_+,\al_-,\ep)	\\
	\al_+=\al_-+ \bar R_2(\bar\xi_+, \bar\xi_-,\al_+,\al_-,\ep)
\end{array}\end{equation}
where $\bar R_1(\bar\xi_+,\bar\xi_-,\al_+,\al_-,\ep)=O(|\bar\xi_+|^2 +|\bar\xi_-|^2 +|\ep|)$ and $\bar R_2(\bar\xi_+,\bar\xi_-,\al_+,\al_-,\ep)=O(|\ep|)$, uniformly with respect to $(\xi_+,\xi_-,\al_+,\al_-)$. Now the map $(\bar\xi_+,\bar\xi_-)\mapsto \bar\xi_+-\bar\xi_-$ is a linear map from
${\cal R}Q_+\times{\cal N}Q_-$ into ${\cal R}Q_+\times {\cal N}Q_-$ whose kernel is ${\cal R}Q_+\cap{\cal N}Q_-$ which, by assumption $A_5)$, is $d$-dimensional.

Let $W\subset{\cal R}Q_+$ be a complement of ${\cal R}Q_+\cap{\cal N}Q_-$ in ${\cal R}Q_+$, so that
\[
	{\cal R}Q_++{\cal N}Q_- = W\oplus{\cal N}Q_-.
\]
Note that $\dim W = k-d$ and $\R^n = [{\cal R}Q_++{\cal N}Q_-]\oplus\span\{\psi_1,\ldots,\psi_d\}$. Recall that we assumed that $(\psi_1,\ldots,\psi_d)$ is orthonormal. Then, let $Q:\R^n\to\R^n$ be the orthogonal projection such that ${\cal R}Q={\cal R}Q_++{\cal N}Q_-$ and ${\cal N}Q=\span\{\psi_1,\ldots,\psi_d\}$. Since $(\I-Q)x\in {\cal N}Q = \span\{\psi_1,\ldots,\psi_d\}$ and $(\psi_1,\ldots,\psi_d)$ is orthonormal we get
\[
	(\I-Q)x = \sum_{j-1}^d \langle \psi_j, (\I-Q)x\rangle \psi_j = \sum_{j-1}^d \langle (\I-Q)\psi_j, x\rangle \psi_j  = \sum_{j-1}^d (\psi_j^T x)\psi_j .
\]
Hence we replace \eqref{bifeq2} with
\begin{equation}\label{bifeq3}\begin{array}{l}
	\bar\xi_+ - \bar\xi_- = Q[ w^-_0(\al_-)-w^+_0(\al_+) + \bar R_1(\bar\xi_+, \bar\xi_-,\al_+,\al_-,\ep)],	\\
	\al_+-\al_- = \bar R_2(\bar\xi_+, \bar\xi_-,\al_+,\al_-,\ep) 		\\
	\psi_j^T[ w^-_0(\al_-)-w^+_0(\al_+) + \bar R_1(\bar\xi_+, \bar\xi_-,\al_+,\al_-,\ep)] = 0.
\end{array}\end{equation}
Since ${\rm dim}[{\cal R}Q_++{\cal N}Q_-]=n-d$, for any $\ep$
\begin{equation}\label{redbif}
	\begin{array}{l}
		\bar\xi_+ - \bar\xi_- - Q[ w^-_0(\al_-)-w^+_0(\al_+) ] =Q\bar R_1(\bar\xi_+, \bar\xi_-,\al_+,\al_-,\ep)],	\\
		\al_+-\al_- = \bar R_2(\bar\xi_+, \bar\xi_-,\al_+,\al_-,\ep)
\end{array}\end{equation}
is essentially a system of $n-d+m$ equations in the $n-d+2m$ variables $(\bar\xi_+,\bar\xi_-,\al_+,\al_-)$ such that, when $\ep=0$, has the solution
\[
	(\bar\xi_+, \bar\xi_-)=(0,0),\quad \al_+=\al_-=y_0.
\]
The Jacobian matrix at this point is
\[
	J=\begin{pmatrix}
		L & -Q w^-_{0,y}(y_0) & Qw^+{0,y}(y_0)	\\
		0 & \I_{\R^m} & - \I_{\R^m}
	\end{pmatrix}
\]
where $L : W\times{\cal N}Q_-\to W\oplus{\cal N}Q_-$ is the invertible linear map given by $L(\bar\xi_+, \bar\xi_-) = \bar\xi_+ - \bar\xi_-$. We have
\[
\rank J = \rank \begin{pmatrix}
		L & Q[w^+_{0,y}(y_0) - w^-_{0,y}(y_0)] & Qw^+{0,y}(y_0)	\\
		0 & 0 & - \I_{\R^m}
	\end{pmatrix} = n-d+m
\]
hence, for $\ep\ne 0$ and sufficiently small \eqref{redbif} has a $m$-dimensional manifold of solutions
\[
	\bar\xi_+=\bar\xi_+(\al_+,\ep), \quad \bar\xi_-=\bar\xi_-(\al_+,\ep), \quad \al_-=\al_-(\al_+,\ep)
\]
where
\[\begin{array}{l}
	|\bar\xi_\pm(\al_+,\ep)| = O(|\al_+-y_0|+|\ep|)	\\
	|\al_-(\al_+,\ep)-y_0|=O(|\al_+-y_0|+|\ep|)
\end{array}\]
Note also that
\[
	\lim_{\ep\to 0} |\al_-(\al_+,\ep)-\al_+| = 0
\]
uniformly with respect to $\al_+$. Then we plug this solution in the third equation in \eqref{bifeq3} and obtain the system of equations
\[
	\psi_j^T[ w^-_0(\al_+) - w^+_0(\al_+) + O(|\al_+-y_0|^2+|\ep|]) = 0, \quad j=1,\ldots,d.
\]
Let
\[
	{\cal M}(\al_+,\ep) = \left (\psi_j^*[ w^-_0(\al_+) - w^+_0(\al_+) + O(|\al_+-y_0|^2+|\ep|)]\right )_{j=1,\ldots,d} .
\]
We have ${\cal M}:\R^m\times\R \to \R^d$, ${\cal M}(y_0,0) = 0$ and
\[\begin{array}{l}
	{\cal M}_{\al_+}(y_0,0) = \left (\psi_j^T[ w^-_{0,y}(y_0) - w^+_{0,y}(y_0)] \right )_{j=1,\ldots,d}.
\end{array}\]
Hence from the Implicit Functions Theorem the existence follows of $\ep_0>0$ such that for any $|\ep|<\ep_0$ there exists a $(m-d)$-dimensional submanifold ${\cal S}$ of $\R^m$ such that when $\alpha_+\in{\cal S}$ we have ${\cal M}(\al_+,\ep) = 0$. For $\al_+\in{\cal S}$, we take
\[
	\al_-=\al_-(\al_+,\ep), \quad \bar\xi_+=\bar\xi_+(\al_+,\ep), \quad \bar\xi_-=\bar\xi_-(\al_+,\ep),
\]
and $\tilde\xi_\pm=\tilde\xi_\pm(\al_+,\ep)$, so that
\[
\bar\xi_+(\al_+,\ep)=Q_+\tilde\xi_+(\al_+,\ep), \quad \bar\xi_-(\al_+,\ep)=(\I-Q_-)\tilde\xi_-(\al_+,\ep).
\]
Then
\[\begin{array}{l}
	x(t,\ep) = x(t,\tilde\xi_+(\al_+,\ep),\tilde\xi_-(\al_+,\ep),\al_+,\al_-(\al_+,\ep),\ep)	\\
	y(t,\ep) = y(t,\tilde\xi_+(\al_+,\ep),\tilde\xi_-(\al_+,\ep),\al_+,\al_-(\al_+,\ep),\ep)
\end{array}\]
with $\al_+\in{\cal S}$ satisfies the conclusion of the Theorem. The proof is complete.
\end{proof}
\begin{remark}\label{remrem} i) According to Remark \ref{rem1} we see that $(x(t,\ep),y(t,\ep))$ satisfies
\[\begin{array}{l}
	\sup_{t\ge T_+}e^{\be t}|x(t,\ep) - \tilde w_+(y(t,\ep),\ep)| \le \tilde \mu_1	\\
	\sup_{t\le T_-}e^{-\be t}|x(t,\ep) - \tilde w_-(y(t,\ep),\ep)| \le \tilde \mu_1	\\
\end{array}\]
where $\tilde\mu_1,\tilde\mu_2$ do not depend on $\ep$.

ii) We can replace the orthonormal basis $(\psi_1,\ldots,\psi_d)$ of $[{\cal R}Q_+ + {\cal N}Q_-]^\perp$ with any independent set $(\tilde\psi_1,\ldots,\tilde\psi_d)$
such that
\[
	\R^n = [{\cal R}Q_+ + {\cal N}Q_-] \oplus \span\{\tilde\psi_1,\ldots,\tilde\psi_d\}.
\]
Indeed, let $\langle \cdot, \cdot\rangle$ be a scalar product on $\R^n$ such that
\[
	[{\cal R}Q_+ + {\cal N}Q_-]^\perp = \span \{ \tilde\psi_1,\ldots,\tilde\psi_d\}.
\]
and let $(\psi_1,\ldots,\psi_d)$ be an orthonormal basis of $\span\{\tilde\psi_1,\ldots,\tilde\psi_d\}$. Then an invertible $d\times d$ matrix $M$ exists such that
\[
	(\tilde\psi_1 \ldots \tilde\psi_d) = (\psi_1 \ldots \psi_d)M.
\]
Hence
\[\begin{array}{l}
	[\tilde\psi_j^T [ w^-_{0,y}(y_0) - w^+_{0,y}(y_0)]_{j=1,\ldots,d} = (\tilde\psi_1 \ldots \tilde\psi_d)^T [ w^-_{0,y}(y_0) - w^+_{0,y}(y_0)]		\\
	= M^T(\psi_1 \ldots \psi_d)^T [ w^-_{0,y}(y_0) - w^+_{0,y}(y_0)] =  M^T[\psi_j^T [ w^-_{0,y}(y_0) - w^+_{0,y}(y_0)]_{j=1,\ldots,d}
\end{array}\]
that is $[\tilde\psi_j^T [w^-_{0,y}(y_0) - w^+_{0,y}(y_0)]$ has rank $d$ if and only if $[\psi_j^T [w^-_{0,y}(y_0) - w^+_{0,y}(y_0)]_{j=1,\ldots,d}$
has rank $d$.
\end{remark}

We conclude this section giving another expression for ${\cal M}_{\al_+}(y_0,0)$ that can be useful in the applications of Theorem \ref{main}.
\begin{prop}\label{altMel} Let $u(t,y)$ be the $C^1_b$-function defined in \eqref{rel_u+-} and let
$\psi \in [{\cal R}Q_++{\cal N}Q_-]^\perp = {\rm span}\{\psi_1,\ldots,\psi_d\}$. Then
\[
	\psi^T[w_{0,y}^-(y_0) - w_{0,y}^+(y_0)] = \int_{-\infty}^\infty \psi(t)^Tf_y(u(t,y_0),y_0)dt
\]
where
\begin{equation}\label{defpsi}\begin{array}{l}
\psi(t) = \left \{\begin{array}{ll}
	(X_-(t,y_0)^T)^{-1}Q_-^T\psi & \hbox{for $t\le 0$}	\\
	(X_+^T(t,y_0))^{-1}(\I-Q_+^T)\psi & \hbox{for $t\ge 0$.}
\end{array}\right .
\end{array}\end{equation}
 Hence, the Melnikov conditon in Theorem \ref{main} reads
\begin{equation}\label{Melcnd}
{\rm rank} \left [ \int_{-\infty}^\infty \psi_j(t)^T f_y(u(t,y_0),y_0) dt\right ]_{j=1,\ldots,d} = d
\end{equation}
where $\psi_j(t)$ is as in \eqref{defpsi} with $\psi_j$ instead of $\psi$.
\end{prop}
\begin{proof}
As $u_y(t,y_0)$ is a bounded solution of
\[
	\dot x = A(t)x +  f_y(u(t,y_0),y_0)
\]
where $A(t)$ is as in \eqref{defA} with $y=y_0$,  and
\[
f_y(u(t,y_0),y_0) = \left \{
\begin{array}{ll}
	f_{-M,y}(u(t,y_0),y_0) & \hbox{if $t<t_{-M}(y_0)$}		\\
	f_{j+1,y}(u(t,y_0),y_0) & \hbox{if $t_j(y_0)<t<t_{j+1}(y_0)$} 	\\
	f_{i-1,y}(u(t,y_0),y_0) & \hbox{if $t_{i-1}(y_0)<t<t_i(y_0)$} 	\\
	f_{N,y}(u(t,y_0),y_0) & \hbox{if $t>t_N(y_0)$}	\\	
\end{array}\right .\]
from to Lemma \ref{prelem}, equation \eqref{xneg}, with $\tau=0$ we get:
\[\begin{array}{l}
	\dis u_y(t,y_0) = X_-(t,y_0)(\I-Q_-)u_y(0^-,y_0)		\\
	\dis + \int_{-\infty}^t X_-(t,y_0)Q_-X_-(s,y_0)^{-1} f_y(u(s,y_0),y_0) ds 		\\
	\dis - \int_t^0 X_-(t,y_0)(\I-Q_-)X_-(s,y_0)^{-1} f_y(u(s,y_0),y_0) ds.
\end{array}\]
Taking the $\lim_{t\to 0^-}$ and using $u_y(0^-,y_0) = w^-_{0,y}(y_0)$, we get
\[
	w^-_{0,y}(y_0) = (\I-Q_-)w^-_{0,y}(y_0) + \int_{-\infty}^0 Q_-X_-(s)^{-1} f_y(u(s,y_0),y_0) ds
\]
that is
\[
	Q_-w^-_{0,y}(y_0) = \int_{-\infty}^0 Q_-X_-(s)^{-1} f_y(u(s,y_0),y_0) ds .
\]
Similarly
\[
	(\I-Q_+)w^+_{0,y}(y_0) = -\int_0^\infty (\I-Q_+)X_+(s)^{-1} f_y(u(s,y_0),y_0) ds .
\]
Now, since
\[
[{\cal R}Q_++{\cal N}Q_-]^\perp = [{\cal R}Q_+]^\perp\cap[{\cal N}Q_-]^\perp = {\cal N}Q_+^T \cap {\cal R}Q_-^T
\]
for any $\psi\in[{\cal R}Q_++{\cal N}Q_-]^\perp$ we get
\[\begin{array}{l}
	\psi^T(\I-Q_-)  = 0	\\
	\psi^TQ_+ = 0.
\end{array}\]
Hence, for any $\psi\in [{\cal R}Q_++{\cal N}Q_-]^\perp$ we have:
\begin{equation}\label{alt}\begin{array}{l}
	\psi^T[w^-_{0,y}(y_0) - w^+_{0,y}(y_0)] = \psi_j^T[Q_-w^-_{0,y}(y_0) - (\I-Q_+)w^+_{0,y}(y_0)] 	\\
	\qquad \dis = \int_{-\infty}^\infty \psi(t)^T f_y(u(t,y_0),y_0) dt
\end{array}\end{equation}
where $\psi(t)$ is as in \eqref{defpsi} with $\psi\in [{\cal R}Q_++{\cal N}Q_-]^\perp$ instead of $\psi_j$. The proof is complete.
\end{proof}

The adjoint system to \eqref{var+} and \eqref{var-} is given by \cite{BF0}
\begin{equation}\label{advar}
\begin{array}{l}
	\dot w = -A^T(t,\al)w \quad \hbox{if $t\ge 0$}	\\
	B_\ell(\al)^Tw(t_\ell(\al)^+) = w(t_\ell(\al)^-)
\end{array}
\end{equation}
where $\ell \in\{-M,\ldots,N\}\setminus\{0\}$.

It is easy to check that, if $\psi\in[{\cal R}Q_++{\cal N}Q_-]^\perp$, the function $\psi(t)$ defined in \eqref{defpsi} is a bounded solution of \eqref{advar} for $\al=y_0$. We prove that if ${\rm span}\{\psi_1,\ldots,\psi_d\}=[{\cal R}Q_++{\cal N}Q_-]^\perp$ then $\{\psi_1(t),\ldots,\psi_d(t)\}$ is a basis for the space of the bounded solutions of \eqref{advar}. Indeed, the fundamental matrix of \eqref{advar} on $t\ge 0$ is $[X_+(t)^T]^{-1}$, and the fundamental matrix of \eqref{advar} on $t\le 0$ is $[X_-(t)^T]^{-1}$. As a consequence \eqref{advar} has an exponential dichotomy on $\R_+$ and $\R_-$ with projections $(\I-Q_+^T)$ and $(\I-Q_-^T)$ respectively. So, the space of bounded solutions of \eqref{advar}, $C^1$ for
$t\ne t_\ell(\al)$, are those whose initial conditions belong to
\[
{\cal R}(\I-Q_+^T)\cap {\cal N}(\I-Q_-^T) = ({\cal R}Q_+)^\perp \cap ({\cal N}Q_-)^\perp = ({\cal R}Q_+ + {\cal N}Q_-)^\perp .
\]
Then the dimension of the space of solutions of \eqref{advar}, bounded on $\R$, is $d$ and $(\psi_1(t),\ldots,\psi_d(t))$ span this space.
\medskip

Now suppose that $x(t)$ and $\psi(t)$ are bounded solution on $\R$ of \eqref{var+}-\eqref{var-} and \eqref{advar} resp., both continuous for
$t\ne t_\ell(y_0)$. For $t\ne t_\ell(y_0)$ we have
\[
\frac{d}{dt}[\psi(t)^Tx(t)] = \dot\psi(t)^Tx(t) + \psi(t)^T\dot x(t) = - \psi(t)^TA(t,\al)x(t) + \psi(t)^T\dot x(t) = 0.
\]
Moreover
\[\begin{array}{l}
	\psi(t_\ell(\al)^+)^Tx(t_\ell(\al)^+)) = [(B_\ell(\al)^{-1})^T\psi(t_\ell(\al)^-)]^Tx(t_\ell(\al)^+)	\\
	=\psi(t_\ell(\al)^-)^T[B_\ell(\al)^{-1}x(t_\ell(\al)^+] = \psi(t_\ell(\al)^-)^Tx(t_\ell(\al)^-)
\end{array}\]
Thus we conclude that $\psi(t)^Tx(t)$ is constant on $\R$ (see also \cite{SP}).

\section{An example}
An interesting application of Theorem \ref{main} is when $d=1$ that is when
\[
{\cal R}Q_+\cap{\cal N}Q_- =\span\{\dot u(0,y_0)\}.
\]
This condition is trivially satisfied when $n=2$ since in this case $k=n-k=1$. Moreover, when $n=2$, we also have $\dim{\cal R}Q_+=
\dim{\cal N}Q_-=1$ and hence
\begin{equation}\label{RN}
{\cal R}Q_+={\cal N}Q_- =\span\{\dot u(0,y_0)\}.
\end{equation}

In this section we consider examples of applications of Theorem \ref{main} with $n=2$, $m=1$ and $d=1$. First we prove some general facts concerning two-dimensional differential equations depending on a slowly varying variable. So the system is
\begin{equation}\label{genex}\begin{array}{l}
	\dot x_1 = F_1(x_1,x_2,y)	\\
	\dot x_2 = F_2(x_1,x_2,y)	\\
	\dot y = \ep g(x_1,x_2,y) .
\end{array}\end{equation}
Suppose $u(t,y)=(u_1(t,y), u_2(t,y))$ is a piecewise smooth solution of \eqref{genex} for $\ep=0$ satisfying assumptions $A_1)-A_4)$.
To write the Melnikov condition \eqref{Melcnd}, that in this case reads
\[
	\int_{-\infty}^\infty \psi(t)^T\begin{pmatrix} F_{1,y}(u(t,y_0),y_0) \\ F_{2,y}(u(t,y_0),y_0) \end{pmatrix} dt \ne 0,
\]
we need to know the (unique) bounded solution $\psi(t)$ of the adjoint system \eqref{advar}.

Let
\[
	J = \begin{pmatrix}	0 & -1 \\ 1 & 0	\end{pmatrix}
\]
 and, again, $\ell\in\{-M,\ldots,-1,1,\ldots,N\}$.

We prove the following
\begin{prop}\label{solAd2dim}
Let $A(t) = [a_{jk}(t)]_{1\le j,k\le 2} := [F_{j,x_k}(u_1(t),u_2(t),y_0)]_{1\le j,k\le 2}$, $B_\ell$ as in \eqref{defB*}, with $\al=y_0$ and
\[
	v(t) : =  e^{-\int_0^t a_{11}(s)+a_{22}(s) ds}J\dot u(t,y_0) = e^{-\int_0^t a_{11}(s)+a_{22}(s) ds}
	\begin{pmatrix} -\dot u_2(t,y_0) \\ \dot u_1(t,y_0) \end{pmatrix}.
\]
Then the space of bounded solution of the adjoint variational system are of the form
\[
	\psi(t) = \left \{\begin{array}{ll}
		\mu_{-M} v(t) &  \hbox{for $t\le t_{-M}$}				\\
		\mu_{j+1}v(t), & \hbox{for $t_j<t\le t_{j+1}$}		\\
		\mu_iv(t), & \hbox{for $t_i\le t<t_{i+1}$}		\\
		\mu_N v(t) &  \hbox{for $t\ge t_N$}	
	\end{array}\right .
\]
where $\mu_{-M}\ne 0$ is arbitrary and
\begin{equation}\label{jmp}\begin{array}{l}
	\mu_{j+1}v(t_j^+)  = \mu_j[B_j^T]^{-1}v(t_j^-) 	\\
	\mu_iv(t_i^+)  = \mu_{i-1}[B_i^T]^{-1}v(t_i^-) .
\end{array}\end{equation}
\end{prop}
\begin{proof}
First we show that constants satisfying \eqref{jmp} exist. Indeed recall that for any $i=1,\ldots,N$
\[\begin{array}{l}
[[B_i^T]^{-1}v(t_i^-)]^T\dot u(t_i^+,y_0) = v(t_i^-)^T[B_i^{-1}\dot u(t_i^+,y_0)] = v(t_i^-)^T\dot u(t_i^-,y_0) = 0	\\
v(t_i^+)^T\dot u(t_i^+,y_0) = 0,
\end{array}\]
where $t_i=t_i(y_0)$, and then
\[
v(t_i^+)  = \tilde\mu_{i}[B_i^T]^{-1}v(t_i^-)
\]
for some $\tilde \mu_{i}\in \R$, since both vectors are orthogonal to $\dot u(t_i^+,y_0)$. In  a similar way we see that constants satisfying \eqref{jmp} for $j=-1,\ldots,-M$ exist.

Next, using $JA(t) = {\rm det}A(t) [A(t)^T]^{-1}J$:
\[\begin{array}{l}
\dot v(t) = e^{-\int_0^t a_{11}(s)+a_{22}(s) ds} J A(t) \begin{pmatrix} \dot u_1(t,y_0) \\ \dot u_2(t,y_0) \end{pmatrix} - (a_{11}(t)+a_{22}(t))v(t) 	\\
	= e^{-\int_0^t a_{11}(s)+a_{22}(s) ds}{\rm det} A(t)[A(t)^T]^{-1}J \begin{pmatrix} \dot u_1(t,y_0) \\ \dot u_2(t,y_0) \end{pmatrix} - (
	a_{11}(t)+a_{22}(t))v(t)		\\
	= {\rm det} A(t) [A(t)^T]^{-1}v(t) - (a_{11}(t)+a_{22}(t))v(t) = - A(t)^Tv(t).
\end{array}\]
Hence
\[
	\dot \psi(t) =  - A(t)^T\psi (t),
\]
for any $t\ne t_\ell$. Moreover
\[\begin{array}{l}
	\psi(t_j^-) = \mu_jv(t_j^-) = B_j^T \mu_{j+1}v(t_j^+) = B_j^T \psi(t_j^+)		\\
	\psi(t_i^+) = \mu_iv(t_i^+) = \mu_{i-1}[B_i^T]^{-1}v(t_i^-) = [B_i^T]^{-1}\psi(t_i^-) .
\end{array}\]
Finally, we prove that $\psi(t)$ is bounded. As the adjoint system has an exponential dichotomy on $\R_+$, resp. $\R_-$, with projections $\I-Q_+^T$, resp. $I-Q_-^T$, it is enough to prove that
\[
	\psi(0) \in {\cal R}(\I-Q_+^T) \cap {\cal N}(I-Q_-^T) = {\cal R}(Q_+)^\perp \cap {\cal N}(Q_-)^\perp = \{\dot u(0)\}^\perp
\]
the last equality following from \eqref{RN}. But
\[
	\langle \psi(0), \dot u(0,y_0)\rangle = \langle J\dot u(0,y_0),\dot u(0,y_0)\rangle = 0
\]
and then $\psi(t)$ is bounded concluding the proof of the Proposition.
\end{proof}
\begin{remark}\label{remmui}
i) From \eqref{jmp} we have
\[\begin{array}{l}
	\mu_i J\dot u(t_i^+,y_0) = \mu_{i-1} [B_i^T]^{-1}J\dot u(t_i^-,y_0)		\\
	\mu_{j+1} J\dot u(t_j^+,y_0) = \mu_j [B_j^T]^{-1}J\dot u(t_j^-,y_0)	
\end{array}\]
and then
\[
	\mu_i\|\dot u(t_i^+,y_0) \|^2 = \mu_i\langle J\dot u(t_i^+,y_0) ,J\dot u(t_i^+,y_0) \rangle =
	\mu_{i-1}\langle [B_i^T]^{-1}J\dot u(t_i^-,y_0),J\dot u(t_i^+,y_0) \rangle
\]
and similarly
\[
	\mu_{j+1}\|\dot u(t_j^+,y_0) \|^2 = \mu_{j+1}\langle J\dot u(t_j^+,y_0) ,J\dot u(t_j^+,y_0) \rangle =
	\mu_j\langle [B_j^T]^{-1}J\dot u(t_j^-,y_0), J\dot u(t_j^+,y_0) \rangle.
\]
Hence all $\mu_\ell$'s can be computed in terms of $\dot u(t_\ell^\pm,y_0)$.

ii) Since $\mu_{-M}\ne 0$ and all $B_j$, $B_i$ are invertible, we see that $\mu_\ell\ne 0$ for all $\ell$.
\end{remark}
\medskip

The case where all $\mu_\ell$ are equal is of particular interest, since in this case we can take $\psi(t)=v(t)$ and the Melnikov condition  reads
\[
	\Delta : = \int_{-\infty}^\infty e^{-\int_0^t a_{11}(s)+a_{22}(s) ds} \begin{pmatrix} -\dot u_2(t,y_0) \\ \dot u_1(t,y_0) \end{pmatrix}^T
	\begin{pmatrix} F_{1,y}(u(t,y_0),y_0) \\ F_{2,y}(u(t,y_0),y_0) \end{pmatrix} dt \ne 0.
\]
If, moreover, $a_{11}(t)+a_{22}(t)=0$ we have
\begin{equation}\label{Delta}\begin{array}{l}
	\dis \Delta = \int_{-\infty}^\infty F_{2,y}(u(t,y_0),y_0) \dot u_1(t,y_0) - F_{1,y}(u(t,y_0),y_0)\dot u_2(t,y_0) dt \\
	\dis =  \int_{\Gamma} F_{2,y}(u_1,u_2,y_0) du_1 - F_{1,y}(u_1,u_2,y_0)du_2
\end{array}\end{equation}
where $\Gamma = \{ (u_1(t),u_2(t)) |t\in\R\}$.

We have the following
\begin{prop}\label{nice} Equations \eqref{jmp} are satisfied with $\mu_0=1$ and $\mu_\ell=1$, if and only if there exist $\nu_\ell$ such that
\begin{equation}\label{redcond}
	J[\dot u(t_\ell(y_0)^+,y_0)  - \dot u(t_\ell(y_0)^-,y_0)] = \nu_\ell h_x(u(t_\ell(y_0),y_0),y_0)^T.
\end{equation}
\end{prop}
\begin{proof}
We have $\mu_0=1$, $\mu_\ell=1$ for all $\ell$, if and only if the following holds.
\begin{equation}\label{redcnd}%\begin{array}{l}
	B_\ell^TJ\dot u(t_\ell(y_0)^+,y_0) = J\dot u(t_\ell(y_0)^-,y_0) 		%\\
%	B_j^T J\dot u(t_j^+,y_0) = J\dot u(t_j^-,y_0)
%\end{array}
\end{equation}
%or, equivalently
%\begin{equation}\label{redcnd}\begin{array}{l}
%	B_i^TJf_i(u(t_i,y_0),y_0) = Jf_{i-1}(u(t_i,y_0),y_0)		\\
%	B_j^TJf_{j+1}(u(t_j,y_0),y_0) = Jf_j(u(t_j,y_0),y_0)
%\end{array}\end{equation}
%for all $i=1,\ldots,N$.
We check that
\[\begin{array}{l}
	\langle x ,B_\ell y\rangle =
	\langle x - \frac{\langle \dot u(t_\ell(y_0)^-,y_0) - \dot u(t_\ell(y_0)^+,y_0),x \rangle}{h_x(u(t_\ell,y_0),y_0)\dot u(t_\ell(y_0)^-,y_0)}
	h_x(u(t_\ell(y_0),y_0),y_0)^T, y\rangle
\end{array}\]
%\[\begin{array}{l}
%	\langle x + \frac{\langle f_i(u(t_{-i}),y_0) - f_{i+1}(u(t_{-i}),y_0), x \rangle}{h_x(u(t_{-i}),y_0)f_{i+1}(u(t_{-i}),y_0)}h_x(u(t_{-i}),y_0)^T,y
%	\rangle = \langle x, B_j y\rangle	\\	\\
%	\langle x - \frac{\langle f_i(u(t_i),y_0) - f_{i+1}(u(t_i),y_0), x \rangle}{h_x(u(t_i),y_0)f_i(u(t_i),y_0)}h_x(u(t_i),y_0)^T,y\rangle =
%	\langle x, B_i y\rangle	
%\end{array}\]
so that
\[
	B_\ell^T x = x - \frac{\langle \dot u(t_\ell(y_0)^-,y_0) - \dot u(t_\ell(y_0)^+,y_0),x \rangle}{h_x(u(t_\ell,y_0),y_0)\dot u(t_\ell(y_0)^-,y_0)}
	h_x(u(t_\ell(y_0),y_0),y_0)^T
%\begin{array}{l}
%	\dis B_j^Tx = x + \frac{\langle f_i(u(t_{-i}),y_0) - f_{i+1}(u(t_{-i}),y_0), x \rangle}{h_x(u(t_{-i}),y_0)f_{i+1}(u(t_{-i}),y_0)}h_x(u(t_{-i}),y_0)^T,	\\
%		\\
%	\dis B_i^Tx = x - \frac{\langle f_i(u(t_i),y_0) - f_{i+1}(u(t_i),y_0), x \rangle}{h_x(u(t_i),y_0)f_i(u(t_i),y_0)}h_x(u(t_i),y_0)^T .
%\end{array}
\]
Then \eqref{redcnd} is equivalent to :
\begin{equation}\label{redcnd1}\begin{array}{l}
	J\dot u(t_\ell^-,y_0)
	%= B_\ell^TJ\dot u(t_\ell^+,y_0)
	=  J\dot u(t_\ell^+,y_0) 	\\
	- \frac{\langle \dot u(t_\ell(y_0)^-,y_0) - \dot u(t_\ell(y_0)^+,y_0),
	J\dot u(t_\ell(y_0)^+,y_0) \rangle}{h_x(u(t_\ell,y_0),y_0)\dot u(t_\ell(y_0)^-,y_0)} h_x(u(t_\ell(y_0),y_0),y_0)^T
\end{array}\end{equation}
%\begin{equation}\label{redcnd1}\begin{array}{l}
%	Jf_i(u(t_{-i}),y_0) + \frac{\langle f_i(u(t_{-i}),y_0) - f_{i+1}(u(t_{-i}),y_0), Jf_i(u(t_{-i}),y_0)\rangle}{h_x(u(t_{-i}),y_0)f_{i+1}(u(t_{-i}),y_0)}
%	h_x(u(t_{-i}),y_0)^T 			\\ 	\qquad = Jf_{i+1}(u(t_{-i}),y_0)		\\	\\
%	Jf_{i+1}(u(t_i),y_0) - \frac{\langle f_i(u(t_i),y_0) -f_{i+1}(u(t_i),y_0) , Jf_{i+1}(u(t_i),y_0) \rangle}{h_x(u(t_i),y_0)f_i(u(t_i),y_0)}
%	h_x(u(t_i),y_0)^T 		\\ 		\qquad = Jf_i^+(u(t_i),y_0)
%\end{array}\end{equation}
or else, as $\langle Jx , x\rangle =0$, %and $J^T=-J$, \eqref{redcnd1} is equivalent to
\begin{equation}\label{redcond2} %\begin{array}{l}
	J\dot u(t_\ell^+,y_0) - J\dot u(t_\ell^-,y_0) = \frac{\langle \dot u(t_\ell(y_0)^-,y_0) ,
	J\dot u(t_\ell(y_0)^+,y_0) \rangle}{h_x(u(t_\ell(y_0),y_0)\dot u(t_\ell(y_0)^-,y_0)}
	h_x(u(t_\ell(y_0),y_0)^T 	%	\\	\\
%	J (f_{i+1}(u(t_i),y_0) - f_i(u(t_i),y_0)) = \frac{\langle Jf_{i+1}(u(t_i),y_0) , f_i(u(t_i),y_0) \rangle}
%	{h_x(u(t_i),y_0)f_i(u(t_i),y_0)}h_x(u(t_i),y_0)^T .
%	\end{array}
\end{equation}
which is \eqref{redcond} with
\[
	\nu_\ell = \frac{\langle \dot u(t_\ell(y_0)^-,y_0) , J\dot u(t_\ell(y_0)^+,y_0) \rangle}{h_x(u(t_\ell(y_0),y_0)\dot u(t_\ell(y_0)^-,y_0)}.
\]
%\[\begin{array}{l}
%	\nu_i^- =  \frac{\langle Jf_{i+1}(u(t_{-i}),y_0) , f_i(u(t_{-i}),y_0) \rangle}{h_x(u(t_{-i}),y_0)f_{i+1}(u(t_{-i}),y_0)}	\\	\\
%	\nu_i^+ =  \frac{\langle Jf_{i+1}(u(t_i),y_0) , f_i(u(t_i),y_0) \rangle}{h_x(u(t_i),y_0)f_i(u(t_i),y_0)} .
%\end{array}\]
On the other hand, if \eqref{redcond} holds, taking the scalar product with $\dot u(t_\ell(y_0),y_0)$ we get
\[\begin{array}{l}
	\nu_\ell \langle h_x(u(t_\ell(y_0),y_0)^T, \dot u(t_\ell(y_0)^-,y_0)\rangle =
	\langle J[\dot u(t_\ell(y_0)^+,y_0)  - \dot u(t_\ell(y_0)^-,y_0)] , \dot u(t_\ell(y_0)^-,y_0)\rangle
%	\dis \frac{\langle Jf_{i+1}(u(t_{-i}),y_0),f_i(u(t_{-i}),y_0)\rangle}{h_x(u(t_{-i}),y_0)f_{i+1}(u(t_{-i}),y_0)} =
%	-\frac{\langle f_{i+1}(u(t_{-i}),y_0), Jf_i(u(t_{-i}),y_0)\rangle}{h_x(u(t_{-i}),y_0)f_{i+1}(u(t_{-i}),y_0)}		\\
%	\dis - \frac{\langle f_{i+1}(u(t_{-i}),y_0) ,Jf_{i+1}(u(t_{-i}),y_0) - \nu_i^- h_x(u(t_{-i}),y_0)^T\rangle}{h_x(u(t_{-i}),y_0)f_{i+1}(u(t_{-i}),y_0)}	\\
%	\dis = \nu_i^- \frac{\langle h_x(u(t_{-i}),y_0)^T , f_{i+1}(u(t_{-i}),y_0)\rangle}{h_x(u(t_{-i}),y_0)f_{i+1}(u(t_{-i}),y_0)} = \nu_i^-
\end{array}\]
that is
\[
	\nu_\ell = \frac{\langle J\dot u(t_\ell(y_0)^+,y_0), \dot u(t_\ell(y_0)^-,y_0)\rangle}{h_x(u(t_\ell(y_0),y_0)\dot u(t_\ell(y_0)^-,y_0)}
\]
%and similarly
%\[
%	\frac{\langle Jf_{i+1}(u(t_i),y_0) , f_i(u(t_i),y_0) \rangle}{h_x(u(t_i),y_0)f_i(u(t_i),y_0)} = \nu_i^+
%\]
and then \eqref{redcnd} follows, given the equivalence between \eqref{redcnd} and \eqref{redcond2}
\end{proof}
\begin{remark} As $J[\dot u(t_\ell(y_0)+,y_0)  - \dot u(t_\ell(y_0)^-,y_0)] $ is orthogonal to
$\dot u(t_\ell(y_0)^+,y_0)  - \dot u(t_\ell(y_0)^-,y_0)$ and $h_x(u(t_\ell(y_0),y_0)^T$ is orthogonal to the tangent space to ${\cal S}_\ell(y_0)$
at $u(t_\ell(y_0),y_0)$, say $T_{u(t_\ell(y_0),y_0)}{\cal S}_\ell(y_0)$, condition \eqref{redcond} is equivalent to the fact that
$\dot u(t_\ell(y_0)^+,y_0)  - \dot u(t_\ell(y_0)^-,y_0)$ belongs to $T_{(u(t_\ell(y_0),y_0,y_0)}{\cal S}_\ell(y_0)$.
\end{remark}
For example, suppose
\[
	h(x,y) = x_k
\]
where either $k=1$ or $k=2$. Recalling that
\[
	f_\ell(x,y) = \begin{pmatrix}
		F_{1,\ell}(x_1,x_2,y)	\\	F_{2,\ell}(x_1,x_2,y)
	\end{pmatrix}
\]
we get, omitting the argument $u(t_\ell(y_0),y_0)$ for simplicity:
\[
	J[\dot u(t_\ell^+(y_0),y_0) - \dot u(t_\ell^-(y_0),y_0)] = \left \{
	\begin{array}{ll}
		\begin{pmatrix}
			F_{2,i-1} - F_{2,i} 	\\ F_{1,i} - F_{1,i-1}
		\end{pmatrix} & \hbox{if $\ell=i$}	\\
		\begin{pmatrix}
			F_{2,j} - F_{2,j+1} 	\\ F_{1,j+1} - F_{1,j}
		\end{pmatrix} & \hbox{if $\ell=j$}
	\end{array} \right .
\]
and then \eqref{redcond} holds if and only if
\[\begin{array}{l}
	F_{k,i}(u(t_\ell),y_0) = F_{k,i-1} (u(t_\ell),y_0), \quad i= 1,\ldots, N-1	\\
	F_{k,j}(u(t_\ell),y_0) = F_{k,j+1} (u(t_\ell),y_0), \quad j=-1,\ldots, -M
\end{array} \]
\medskip

To give a specific example, consider the second order, discontinuous equation with slowly varying coefficients:
\begin{equation}\label{disKL}\begin{array}{ll}
	\ddot u + u(u-a_-(\ep t + \bar y))(1-u) = 0 	& \hbox{if $u<c$}	\\
	\ddot u + u(u-a_+(\ep t + \bar y))(1-u) = 0 	& \hbox{if $u>c$}	\\
\end{array}\end{equation}
where $0<\inf\{a_\pm(y), y\in\R\}<\sup\{a_\pm(y), y\in\R\}<1$.
\medskip

We prove the following
\begin{prop}\label{discKL} Let $a_\pm(y)$ be $C^1$-functions such that
\[
	0<a_+(y)<\sup\{a_+(y):y\in\R\}\le \frac{1}{2}\le \inf\{a_-(y):y\in\R\}<a_-(y)<1
\]
and let $0<c<1$ be a fixed number. Then equation \eqref{disKL} with $\ep=0$ has a family of  $C^1$-solutions $u_\pm(t,y)$ defined for $t\le 0$, $t\ge 0$ resp., bounded together with their derivatives, such that $u_-(0,y)=u_+(0,y)=c$ and
\[
	\lim_{t\to-\infty} u_-(t,y)=0, \quad \lim_{t\to\infty} u_+(t,y)=1
\]
uniformly with respect to $y$. Next, let $D(y) = c^2(2c-3)[a_+(y)-a_-(y)] + a_+(y)-\frac{1}{2}$ and suppose that $y_0$ and $0<c<1$ exist such that
\begin{equation}\label{generic}\begin{array}{l}
	D(y_0) = 0,	\quad D'(y_0) \ne 0 .
\end{array}\end{equation}
Then $\dot u_-(0,y_0)=\dot u_+(0,y_0)$ and there exists $\ep_0>0$ such that for $0<\ep<\ep_0$ there exists a $C^1$-function $y_0(\ep)$,
such that $\lim_{\ep\to 0}y_0(\ep)=y_0$, and a $C^1$-solution $u(t,\ep)$ of equation \eqref{disKL} with $\bar y=y_0(\ep)$, bounded with its
derivatives and such that
\[\begin{array}{l}
	\dis \lim_{\ep\to 0}\sup_{t\in\R_\pm} |u(t,\ep)-u_\pm(t,\ep t + y_0(\ep))| = 0	\\
	\dis \lim_{\ep\to 0}\sup_{t\in\R_\pm} |\dot u(t,\ep)-\dot u_\pm(t,\ep t + y_0(\ep))| = 0
\end{array}\]
\end{prop}
\begin{proof} Writing $x = \begin{pmatrix} u \\ \dot u \end{pmatrix}$, $y=\ep t + \bar y$, equation \eqref{disKL} reads
\begin{equation}\label{ex2d}
	\dot x = f(x,y) = \begin{pmatrix}
		x_2	\\	x_1(x_1-a(x_1,y))(x_1-1)
	\end{pmatrix},\quad \dot y = \ep
\end{equation}
where
\[
	a(x_1,y) = \left \{\begin{array}{ll}
		a_-(y)	&	\hbox{if $x_1<c$}		\\
		a_+(y)	&	\hbox{if $x_1>c$.}
	\end{array}\right .
\]
Then $h(x,y)=x_1$ and
\[
	f_\pm(x,y) = \begin{pmatrix}
		x_2	\\	x_1(x_1-a_\pm(y))(x_1-1)
	\end{pmatrix} .
\]
Note that
\[
	f_+(x,y) - f_-(x,y) = \begin{pmatrix}
		0	\\	(a_-(y) - a_+(y))x_1(x_1-1)
	\end{pmatrix}
\]
is tangent to the manifold $h(x,y)=c$ and \eqref{disKL} has the fixed points $w_-(y)=(0,\; 0)^T$ and $w_+(y)=(1,\; 0)^T$.
\medskip

Let \eqref{ex2d}$_0$ be equation \eqref{ex2d} with $\ep=0$ and $(u_\pm(t,y),\dot u_\pm(t,y))$, $t\in\R_\pm$, be solutions of \eqref{ex2d}$_0$ such that
\[\begin{array}{l}
	\dis\lim_{t\to-\infty}(u_-(t,y),\dot u_-(t,y))=(0,\; 0),	\\
	\dis\lim_{t\to\infty}(u_+(t,y),\dot u_+(t,y))=(1,\; 0).
\end{array}\]
Multiplying \eqref{disKL} (with $\ep=0$) by $\dot u_-(t,y)$ and integrating from $-\infty$ to $t<0$ we get
\[\begin{array}{l}
	\dis \frac{1}{2}\dot u_-(t,y)^2 = \int_0^{u_-(t,y)} u(u-a_-)(u-1)du 	\\
	\dis \qquad = \frac{u_-(t,y)^2 (3u_-(t,y)^2 -4(a_-(y)+1)u_-(t,y) + 6 a_-(y)}{12}
\end{array}\]
hence, if $u_-(t,y)>0$, it satisfies
\begin{equation}\label{eqmin}
	\frac{du}{dt} =  \frac{u\sqrt{3u^2 -4(a_-(y)+1)u + 6 a_-(y)}}{\sqrt{6}}.
\end{equation}
It is easy to check that, for $1<2a_-(y)<2$ we have $3u^2 -4(a_-(y)+1)u + 6 a_-(y)>0$.

To simplify notation in the following we write $a_-$ for $a_-(y)$. Integrating \eqref{eqmin} we get (see \cite[eq. 2.266]{GR})
\begin{equation}\label{impl_0}\begin{array}{l}
	\sinh^{-1}\left (\frac{\sqrt{2}[3a_--(a_-+1)u_-(t,y)]}{u_-(t,y)\sqrt{(2-a_-)(2a_--1)}} \right ) =
	\sinh^{-1}\left (\frac{\sqrt{2}[3a_--(a_-+1)c]}{c\sqrt{(2-a_-)(2a_--1)}} \right ) - t\sqrt{a_-} .
\end{array}\end{equation}
where $c=u_-(0,y)$. Then, using $\sinh(\al-\be)=\sinh\al\cosh\be - \sinh\be\cosh\al$ and $\cosh\al = \sqrt{1+\sinh^2\al}$:
\[\begin{array}{l}
	\frac{3a_-\sqrt{2}}{u_-(t,y)\sqrt{(2-a_-)(2a_--1)}} =
	\sinh\left (\sinh^{-1}\left (\frac{\sqrt{2}[3a_- - (a_-+1)c]}{c\sqrt{(2-a_-)(2a_--1)}} \right ) - t\sqrt{a_-} \right ) +
	\frac{(a_-+1)\sqrt{2}}{\sqrt{(2-a_-)(2a_--1)}}		\\
	= \sinh\left (\sinh^{-1}\left (\frac{\sqrt{2}[3a_- - (a_-+1)c]}{c\sqrt{(2-a_-)(2a_--1)}} \right ) \right )\cosh (t\sqrt{a_-}) 		\\
		\qquad - \sinh(t\sqrt{a_-})\cosh\left ( \sinh^{-1}\left (\frac{\sqrt{2}[3a_- - (a_-+1)c]}{c\sqrt{(2-a_-)(2a_--1)}}\right )\right )	
		+ \frac{(a_-+1)\sqrt{2}}{\sqrt{(2-a_-)(2a_--1)}}	\\
	= \frac{\sqrt{2}[3a_- - (a_-+1)c]}{c\sqrt{(2-a_-)(2a_--1)}}\cosh(t\sqrt{a_-}) -
	\sqrt{1+\left (\frac{\sqrt{2}[3a_- - (a_-+1)c]}{c\sqrt{(2-a_-)(2a_--1)}}\right )^2}\sinh(t\sqrt{a_-}) + \frac{(a_-+1)\sqrt{2}}{\sqrt{(2-a_-)(2a_--1)}}		\\
	= \frac{\sqrt{2}[3a_- - (a_-+1)c]}{c\sqrt{(2-a_-)(2a_--1)}}\cosh(t\sqrt{a_-}) -
	\sqrt{\frac{3a_-(3c^2-4(a_-+1)c+6a_-)}{c^2(2-a_-)(2a_--1)}}\sinh(t\sqrt{a_-}) + \frac{(a_-+1)\sqrt{2}}{\sqrt{(2-a_-)(2a_--1)}}
\end{array}\]
Hence
\begin{equation}\label{defu-}
	\frac{1}{u_-(t,y)} = \mu_1(a_-)\cosh(t\sqrt{a_-}) - \mu_2(a_-)\sinh(t\sqrt{a_-}) + \frac{a_-+1}{3a_-}.
\end{equation}
where
\begin{equation}\label{defmu_i}
	\begin{array}{l}
		\mu_1(a) = \frac{3a - (a+1)c}{3ac} = \frac{1}{c} -\frac{1}{3}\left ( 1 + \frac{1}{a}\right ).	 	\\
		\mu_2(a) = \sqrt{\frac{3c^2-4(a+1)c+6a}{6ac^2}} = \sqrt{\frac{1}{a}\left ( \frac{1}{2} +\frac{2}{3c}\right )
		+\frac{1}{c^2}-\frac{2}{3c}}
	\end{array}
\end{equation}
We pause for a while to observe that $\mu_2(a)$ is well defined if $3c^2-4(a+1)c+6a\ge 0$. This holds for sure if $\frac{1}{2}\le a<1$. However, if
$0<a<\frac{1}{2}$ we need that
\[
	3c<2(a+1) - \sqrt{4a^2-10 a+4}.
\]
We will consider this issue in Section 8.

As $a_-(y)>0$ and its derivative are bounded on $\R$, so are $\mu_1(a_-(y))$ and $\mu_2(a_-(y))$, Moreover, from \eqref{defu-} it follows that
$u_-(t,y)$ is $C^1$ in $(t,y)$.

We have
\[\begin{array}{l}
	\frac{1}{|u_-(t,y)|} \ge  |\mu_1(a_-)\cosh(t\sqrt{a_-}) - \mu_2(a_-)\sinh(t\sqrt{a_-}) |	\\
	= \frac{1}{2} | (\mu_1(a_-)-\mu_2(a_-))e^{t\sqrt{a_-}} + (\mu_1(a_-)+\mu_2(a_-))e^{-t\sqrt{a_-}}) |		\\
	\ge \frac{1}{2} |(\mu_1(a_-)+\mu_2(a_-))e^{-t\sqrt{a_-}}) | -  \frac{1}{2} | (\mu_1(a_-)-\mu_2(a_-))e^{t\sqrt{a_-}}  |
\end{array}\]
As
\begin{equation}\label{relmu}\begin{array}{l}
	(\mu_2(a) +\mu_1(a))(\mu_2(a)-\mu_1(a) = \mu_2(a)^2 - \mu_1(a)^2 = \frac{(2a - 1)(2-a)}{18a^2} >0
\end{array}\end{equation}
for $1<2a\le 2$, we see that $\mu_2(a) +\mu_1(a)$ and $\mu_2(a)-\mu_1(a)$ do not change sign in  $1<2a\le 2$. But, since for $a=1$ we have
\[
	\mu_2(1) = \frac{1}{c}-\frac{2}{3} = \frac{3-c}{3c} >0, \quad
	\mu_1(1) = \frac{1}{c}\sqrt{\frac{3c^2-8c+6}{6}}>0,
\]
we see that $\inf_{1<2a\le 2}[\mu_2(a)+\mu_1(a)]>0$. Moreover $\mu_1(\frac{1}{2})=\frac{1}{c}-1$ and
$\mu_2(\frac{1}{2})=\sqrt{\frac{c^2-2c+1}{c^2}} = \frac{1-c}{c}>0$ and that $\mu_1(\frac{1}{2})+\mu_2(\frac{1}{2})=2\frac{1-c}{c}>0$.
So
\[
	\inf_{1\le 2a \le 2}[\mu_2(a)+\mu_1(a)]>0
\]
and then
\[
	\frac{1}{u_-(t,y)} \ge \frac{1}{2}\inf_{y\in\R}\{ \mu_2(a_-) +\mu_1(a_-)\} e^{-t\inf_{y\in\R}a_-}.
\]
As a consequence $u_-(t,y)>0$ for $t\le 0$ and
\[
	\lim_{t\to-\infty}u_-(t,y) = 0
\]
uniformly with respect to $y\in\R$. From \eqref{eqmin} we also see that $\dot u_-(t,y)>0$ and $\dis\lim_{t\to-\infty}\dot u_-(t,y)=0$ uniformly with respect to $y$. In particular, for $t\le 0$, $u_-(t,y), \dot u_-(t,y)$ are bounded, positive, functions and $u_-(t,y)$ is strictly increasing from $0$ to
$u_-(0,y)=c$.
\medskip

Now, we look for a (strictly) increasing solution of the second equation in \eqref{disKL} on $t\ge0$. To this end we observe that, $u_+(t,y)$ is a (strictly) increasing solution, on $t\ge 0$, of the second equation in \eqref{disKL} such that
\[
	u_+(0,y)=c, \quad \lim_{t\to\infty}u_+(t,y)=1
\]
if and only if $v(t)=1-u_+(-t)$ is a (strictly) increasing solution, on $t\le 0$, of
\[
	\ddot v + v(v -(1-a_+))(1-v) = 0
\]
such that
\[
	v(0)=1-c, \quad \lim_{t\to-\infty}v(t)=0.
\]
Since $0<a_+(y)\le \frac{1}{2}$ is equivalent to $\frac{1}{2}\le 1-a_+(y)<1$, from the previous part we conclude that the second equation in \eqref{disKL} has a unique strictly increasing solution $u_+(t,y)$ such that
\[
	u_+(0,y)=c, \quad \lim_{t\to\infty}u_+(t,y) = 1
\]
the limit being uniform with respect to $y\in\R$.

Then for any $y\in\R$, equation \eqref{disKL} with $\ep=0$ has a pair of solutions $(u_-(t,y),\dot u_-(t,y))$, $(u_+(t,y),\dot u_+(t,y))$
defined for $t\le 0$ and $t\ge 0$ resp., such that $u_\pm(t)$ are increasing in their interval of definition and
\[\begin{array}{l}
	0<u_-(t,y)<c, \;\hbox{for $t<0$ and  $u_-(0,y)=c$}			\\
	c<u_+(t,y)<1, \;\hbox{for $t>0$ and  $u_+(0,y)=c$}			\\
	\dis \lim_{t\to-\infty}u_-(t,y) = \lim_{t\to\infty}1-u_+(t,y) = 0	\\
	\dis \lim_{t\to-\infty}\dot u_-(t,y) = \lim_{t\to\infty}\dot u_+(t,y) = 0
\end{array}\]
uniformly with respect to $y$.

Now, since $\dis\lim_{t\to-\infty} (u_-(t,y),\dot u_-(t,y))= (0,0)$ and $\dis\lim_{t\to\infty} (u_+(t,y),\dot u_+(t,y))= (1,0)$ we get:
\begin{equation}\label{cndatinf}
\begin{array}{l}
	\dis \frac{1}{2}\dot u_-(0,y)^2 = \int_0^c u(u-a_-(y))(u-1)du		\\
	\dis \frac{1}{2}\dot u_+(0,y)^2 = \int_c^1 u(u-a_+(y))(1-u)du.
\end{array}
\end{equation}
So equation \eqref{disKL}, with $\ep=0$, has a $C^1$ solution heteroclinic to the fixed points $w_-(y)$ and $w_+(y)$ if and only if
\[
	\int_0^c u(u-a_-(y))(u-1)du +  \int_c^1 u(u-a_+(y))(u-1)du = 0.
\]
It is easy to check that
\begin{equation}\label{one}
	\int_0^c u(u-a_-(y))(u-1)du +  \int_c^1 u(u-a_+(y))(u-1)du = \frac{1}{6}D(y).
\end{equation}
and hence equation \eqref{disKL} has a $C^1$ solution heteroclinic to the fixed points $w_-(y)$ and $w_+(y)$ if and only if
\begin{equation}\label{join}
	D(y) = 0.
\end{equation}
From \eqref{generic} we see that equation \eqref{join} has a unique solution only for $y=y_0$.

Then equation \eqref{ex2d}$_0$ has a unique $C^1$ solution $u_0(t)=u(t,y_0)$ asymptotic to $u= 0$ as $t\to-\infty$ and to $u = 1$ as $t\to\infty$ only for $y=y_0$ and this solution breaks when $y\ne y_0$ (in the sense that it is no longer $C^1$). Since $f_+(x,y) - f_-(x,y)$ is tangent to the manifold $h(x,y)=c$ and
\[
	f_y(u_0(t),\dot u_0(t),y) = \left \{\begin{array}{ll} \begin{pmatrix}
		0	\\	-a_-'(y_0)u_0(t)(u_0(t)-1)
	\end{pmatrix} & \hbox{if $t\le 0$}	\\	\\
		\begin{pmatrix}
		0	\\	-a_+'(y_0)u_0(t)(u_0(t)-1)
	\end{pmatrix} & \hbox{if $t>0$}
	\end{array}\right .
\]
the Melnikov function reads:
\[\begin{array}{l}
	\dis -\int_{-\infty}^0 a_-'(y_0)\dot u_0(t)u_0(t)(u_0(t)-1) dt - \int_0^\infty a_+'(y_0)\dot u_0(t)u_0(t)(u_0(t)-1) dt		\\
	= \dis -a_-'(y_0)\int_0^c u(u-1) du - a_+'(y_0)\int_c^1 u(u-1) du = \frac{1}{6}D'(y_0) \ne 0.	%\\	
\end{array}\]
From Theorem \ref{main} the existence follows of $0<\rho<1$ and $\ep_0>0$ such that for $0<\ep<\ep_0$ there exist a
$C^r$-function $y_0(\ep)$ with $\lim_{\ep\to 0} y_0(\ep)=y_0$ and a solution $\tilde u(t,\ep)$ of \eqref{ex2d} such that
\begin{equation}\label{cont}\begin{array}{l}
	\sup_{t\in\R}|\tilde u(t,\ep)-u(t,\ep t+y_0(\ep))| + \sup_{t\in\R}|\dot{\tilde u}(t,\ep)-\dot u(t,\ep t+y_0(\ep))| < \rho,	\\
	\lim_{\ep\to 0} \{\sup_{t\in\R}|\tilde u(t,\ep)-u(t,\ep t+y_0(\ep))| + \sup_{t\in\R}|\dot {\tilde u}(t,\ep)-\dot u(t,\ep t+y_0(\ep))|\} = 0.
\end{array}\end{equation}
\end{proof}

\begin{remark} Suppose that $a_-(y)-a_+(y)=\frac{1}{2}$ with $0<a_+(y)<\frac{1}{2}$. Then \eqref{generic} reads:
\[\begin{array}{l}
	2a_+(y_0) = 2c^3 - 3c^2+1 		\\
	a'_+(y_0) \ne 0.
\end{array}
\]
Note that $0<2c^3 - 3c^2+1<1$ for $0<c<1$.
\end{remark}
\section{Concluding remark}
The assumption $0<a_+(y)\le \frac{1}{2}\le a_-(y)<1$ can be slightly weakened. Indeed, suppose that
\[
	0<a_{\min}=\min\{a_-(y):y\in\R\}\le a_-(y)<1
\]
where $a_{\min}<\frac{1}{2}$.

By uniqueness of analytical continuation, the function defined in \eqref{defu-} is a solution of the first equation in \eqref{discKL} (with $\ep=0$) for any value of $a_-=a_-(y)$ for which $\mu_2(a_-)$ has a meaning, that is such that
\begin{equation}\label{CR1}
	3c^2-4(a_-+1)c+6a_-\ge 0
\end{equation}

Moreover to prove that $u_-(t,y)\to 0$ as $t\to -\infty$, uniformly with respect to $y$, following the same argument of Proposition \ref{discKL}, we need that $\mu_1(a)+\mu_2(a)>0$ that is\begin{equation}\label{CR2}
	3a-(a+1)c + 3a\sqrt{\frac{3c^2-4(a+1)c+6a}{6a}} > 0.
\end{equation}
We prove that \eqref{CR1} and \eqref{CR2} hold if and only if
\begin{equation}\label{cndonc}
	c<\frac{1}{3} \left ( 2(a_{\min}+1) - \sqrt{4a_{\min}^2 - 10a_{\min} + 4}\right ).
\end{equation}
\medskip

As the function of $a$: $3c^2-4(a+1)c+6a$ is increasing, for $2c<3$, \eqref{CR1} holds if and only if
\[
	3c^2-4(a_{\min}+1)c+6a_{\min}\ge 0
\]
that is
\[
	\hbox{either} \; c\le \frac{2(a_{\min}+1)-\sqrt{2\Delta}}{3} \; \hbox{ or }\; c\ge \frac{2(a_{\min}+1)+\sqrt{2\Delta}}{3}
\]
where $\Delta=2a_{\min}^2-5a_{\min}+2$. However it is easy to check that, for $a_{\min}<\frac{1}{2}$, $\frac{2(a_{\min}+1)+\sqrt{2\Delta}}{3}>1$.
Hence \eqref{CR1} is equivalent to
\begin{equation}\label{CR4}
	c\le \frac{2(a_{\min}+1)-\sqrt{2\Delta}}{3}
\end{equation}
Next, the function of $a$
\[
	\mu(a) := (a+1)c - \sqrt{\frac{3a}{2}}\sqrt{3c^2-4(a+1)c+6a} - 3a
\]
is convex and its values at $a=a_{\min}<\frac{1}{2}$ and $a=\frac{1}{2}$ are
\[\begin{array}{l}
	\mu(a_{\min}) = (a_{\min}+1)c - \sqrt{\frac{3a_{\min}}{2}}\sqrt{3c^2-4(a_{\min}+1)c+6a_{\min}} - 3a_{\min}	\\
	\mu(\frac{1}{2}) = \frac{3}{2}c - \sqrt{\frac{3}{4}}\sqrt{3c^2-6c+3} = \frac{3}{2}(c-|c-1|) = \frac{3}{2}(2c-1) - 3 = 3c-\frac{9}{2} <0
\end{array}\]
(since $c<1$). Condition $\mu(a_{\min})<0$ for $a_{\min}\le \frac{1}{2}$ is equivalent to:
\begin{equation}\label{CR5}
	c < \frac{3a_{\min}}{a_{\min}+1}
\end{equation}
Then \eqref{CR2} hods if and only if \eqref{CR5} holds. So we only need to prove that \eqref{CR4} implies \eqref{CR5}.

However the function
\[
	\mu(a) := \frac{9a}{a+1} - 2(a+1) + \sqrt{4a^2 - 10a + 4}
\]
is concave in $0<a<\frac{1}{2}$ and $\mu(0)=\mu(\frac{1}{2})=0$. So $\mu(a)>0$ on $0<a<\frac{1}{2}$ and then
\[
	\frac{2(a_{\min}+1)-\sqrt{2\Delta}}{3} <  \frac{3a_{\min}}{a_{\min}+1}.
\]

So, assuming \eqref{cndonc} the fact that  $\lim_{t\to-\infty}u_-(t,y)=0$ uniformly with respect to $y\in\R$ goes as in Proposition \ref{discKL}.
\medskip

Next, the argument given to prove the existence of $u_+(t,y)$ shows that such a solution exists if $0<a_+(y)<1$ and
\begin{equation}\label{2ndcnd}
	3(1-c)<2(1-a_{\max}+1) - \sqrt{4(1-a_{\max})^2 - 10(1-a_{\max}) + 4}
\end{equation}
where we assume that $a_{\max} := \sup\{a_+(y): y\in\R\}\ge \frac{1}{2}$ . So in order to have both $u_-(t,y)$ and $u_+(t,y)$ we need that
\[
	2a_{\max} -1 + \sqrt{4a_{\max}^2 +2a_{\max}-2} < 3c < 2(a_{\min}+1) - \sqrt{4a_{\min}^2 - 10a_{\min} + 4}.
\]
For example Let $a_{\min} = \frac{1}{2}-\kappa$ and $a_{\max} = \frac{1}{2}+\kappa$, $\kappa>0$. Then
\[
	2(a_{\min}+1) - \sqrt{4a_{\min}^2 - 10a_{\min} + 4} = 3 -2\kappa - \sqrt{2\kappa(2\kappa+3)}
\]
and similarly:
\[
	2a_{\max} -1 + \sqrt{4a_{\max}^2 +2a_{\max}-2} = 2\kappa + \sqrt{2\kappa(2\kappa+3)}.
\]
Then, the set of those $c$ satisfying both \eqref{cndonc} and \eqref{2ndcnd} is not empty if and only if
\[
	2\kappa + \sqrt{2\kappa(2\kappa+3)} < 3 -2\kappa - \sqrt{2\kappa(2\kappa+3)}
\]
or, equivalently, $0<\kappa<\frac{3}{16}$. We conclude this section giving a geometrical interpretation of \eqref{CR4}. Equation $\ddot u =
u(u-a_{\min})(1-u)$ has a homoclinic orbit to $(u,\dot u)=(0,0)$ that intersects the $u$-axis ($\dot u = 0$) at the point $u=\bar u$ where $\bar u$
is the right hand side of \eqref{CR4}. So, if \eqref{CR4} does not hold the portion of the unstable manifold of the fixed point $(0,0)$ of equation
$\ddot u = u(u-a_{\min})(1-u)$ such that $u\ge 0$, lies entirely on the left of the line $u=c$. Hence we cannot have heteroclinic solutions of the discontinuous equation \eqref{discKL} joining $(0,0)$ with $(1,0)$ and such that $0<u(t)<1$.

\section{Proof of Lemma \eqref{Onu+}}

In this appendix we give the proof of Lemma \eqref{Onu+} for $u^+_N(t,y)$, the proof for  $u^-_{-M}(t,y)$ being similar. Let $0<\rho\le 1$. Since $u^+_N(t,y)\to w_+(y)$ uniformly with respect to $y$, there exists $T\ge T_+$ such that
\[
	\sup_{t\ge T}|u^+_N(t,y)-w_+(y)|\le \rho.
\]
First we prove that $u^+_N(t,y)$ satisfies the statement of the Lemma for $t_N^+(y)\le t\le T$. Indeed, for such values of $t$ we have
\[
	u^+_N(t,y) = w^+_N(y) + \int_{t_N(y)}^t f_N(u(s,y),y) ds
\]
As $w^+_N(y)$, $t_N(y)$ and $f_N(x,y)$ are bounded together with their derivatives we see that, for any $k=0,\ldots,n$ there exists a constant $M_k>0$, independent of $y$, such that
\[
	\sup_{t_N(y)\le t\le T}\left |\frac{\partial^k u^+_N}{\partial y^k}(t,y)\right |\le M_k.
\]
We conclude that
\[
	\sup_{t\ge t_N(y)}|u^+_N(t,y)|\le \max\{M_0,\sup\{w_+(y)|y\in\R^n\}+1\}
\]
which is independent of $y\in\R^m$. Now we prove that $u^+_{N,y}(t,y)$ is bounded uniformly with respect to $y\in\R^m$. Let $\bar y\in\R^m$ be fixed, $C^0_b(\R_+)$ be the space of bounded continuous functions on $\R_+$ and $U_0(t,y)$ be the fundamental matrix of
$\dot x = f_{N,x}(w_+(y),y)x$.

Arguing as in the proof of the parametric stable Theorem (see \cite[p. 18]{P1}), there exists $\tilde\rho>0$, $\mu>0$ and $0<\Delta\ll 1$ such that the map
\[\begin{array}{l}
	(\xi,y,w(t)) \mapsto U_0(t,\bar y)P_+^0(\bar y)\xi  	\\
	\dis +\int_0^t U_0(t,\bar y)P_+^0(\bar y)U_0(s,\bar y)^{-1} [f_N(w^+(y)+w(s),y) -f_{N,x}(w^+(\bar y),\bar y)w(s)]ds \\
	\dis -\int_t^\infty U_0(t,\bar y)(\I-P_+^0(\bar y))U_0(s,\bar y)^{-1} [f_N(w^+(y)+w(s),y) -f_{N,x}(w^+(\bar y),\bar y)w(s)]ds
\end{array}\]
where $\xi\in {\cal R}P^0_+(\bar y)$, $|\xi|<\tilde\rho$, $|y-\bar y|<\mu$ is a $C^r$-contraction on the space of bounded continuous functions $w(t)\in C^0_b(\R_+)$ with $\sup_{t\ge 0}|w(t)|\le\Delta$, uniform with respect to $(\xi,y)$. Let $z(t,\xi,y)$ be the fixed point of such a contraction. Then the map $(\xi,y)\mapsto z(t,\xi,y)$, $|\xi|\le \tilde\rho$, $|y-\bar y|<\mu$, is a $C^r$-map into the space of bounded continuous functions on $\R_+$. In particular all derivatives of $z(t,\xi,y)$ are bounded functions (but the bounds may depend on $(\xi,y)$). Now, $x(t)= z(t,\xi,y)+w^+(y)$ satisfies
\begin{equation}\label{par}\begin{array}{l}
	\dot x = f_N(x,y)		\\
	P^0_+(\bar y)[x(0)-w^+(y)]=\xi	\\
	\sup_{t\ge 0}|x(t)-w^+(y)|\le \Delta.
\end{array}\end{equation}
To proceed with the proof we modify $\rho$, and hence $T$, so that $0<K\rho\le\tilde\rho$ and the previous conditions still hold. As
$u^+_N(t+T,y)$ satisfies \eqref{par} with $\xi = P^0_+(\bar y)[u^+_N(T,y)-w^+(y)]$ and
\[
	|P^0_+(\bar y)[u^+_N(T,y)-w^+(y)]|\le K\rho \le \tilde\rho
\]
we conclude, by uniqueness of the fixed point, that
\[
	u^+_N(t+T,y) = w^+(y) + z(t,P^0_+(\bar y)[u^+_N(T,y)-w^+(y)],y).
\]
As $\bar y\in\R^m$ is arbitrary we see that $u^+_{N,y}(t+T,y)$ is bounded on $t\ge 0$. So
\[
	\sup_{t\ge t_N(y)} |u^+_{N,y}(t,y)| <\infty
\]
however the bound may depend on $y$. To prove that it can be taken independent of $y$ we observe that, for $t\ge 0$, $u^+_N(t+T_+,y)$ is a bounded solution of
\[\begin{array}{l}
	\dot x = f_{N,x}(u^+_N(t+T_+,y),y)x + f_{N,y}(u^+_N(t+T_+,y),y)	\\
	P_+(y)x(0) = P_+(y)u^+_{N,y}(T_+,y)
\end{array}\]
and $|u^+_{N,y}(T_+,y)|\le M_1$ since $0\le T_+\le T$. Now, \eqref{variat+} has an exponential dichotomy on $\R_+$ with projections $P_+(y)$ and its fundamental matrix on $\R_+$ is $X(t,y):=U_N^+(t+T_+,y)U_N^+(T_+,y)^{-1}$. Hence,
\[\begin{array}{l}
	\dis u^+_{N,y}(t+T_+,y) = X(t,y)P_+(y)u^+_{N,y}(T_+,y) 	\\
	\dis + \int_0^t X(t,y)P_+(y)X(s,y)^{-1}f_{N,y}(u^+_{N,y}(s+T_+,y),y)	ds	\\
	\dis - \int_t^\infty X(t,y)\I-P_+(y))X(s,y)^{-1}f_{N,y}(u^+_{N,y}(s+T_+,y),y) ds
\end{array}\]
and then, setting $\bar F=\sup_{(x,y)}|f_{N,y}(x,y)|$:
\[
	|u^+_{N,y}(t+T_+,y)| \le Ke^{-\de t}M_1 + \int_0^\infty K\bar Fe^{-\de|t-s|}ds \le K(M_1 + 2\bar F\de^{-1})
\]
which is independent of $y\in\R^m$. Since $T\ge T_+$ and $K\ge 1$, we obtain
\[
	\sup_{t\ge t_N(y)} |u^+_{N,y}(t,y)| \le K(M_1 + 2\bar F\de^{-1}).
\]
More arguments of similar nature prove the Lemma as far as the higher order derivatives of $u^+_N(t,y)$ are concerned. This completes the proof of Lemma \ref{Onu+}.
\medskip

\begin{remark}
We can also give a better estimate of the difference $w(t,y):=u^+_N(t+T,y)-w_+(y)$. Indeed here we prove that for $\rho>0$ sufficiently small,
there exists $\be>0$ such that $w(t,y)e^{\be t}$ and its derivatives with respect to $y$ are bounded on $t\ge 0$, uniformly with respect to $y\in\R^m$.

First, as $\dot w(t,y) = f_N(w(t,y)+w_+(y),y)$ and $f_N(w_+(y),y)=0$, we see that $w(t,y)$ is a bounded solution of
\[
	\dot x = f_{N,x}(w_+(y),y)x + b(x,y)
\]
where $b(x,y) = f_N(x+w_+(y),y) - f_N(w_+(y),y) - f_{N,x}(w_+(y),y)x$. Note that
\[
	|b(x,y)| \le \frac{1}{2}L|x|^2
\]
where $L$ is a Lipschitz constant for $f_{N,x}(x,y)$. Then
\[\begin{array}{l}
	w(t,y) = U_0(t,y)P^0_+(y)w(0,y) + 	\\
	\dis \quad \int_0^t U_0(t,y)P^0_+(y)U_0(s,y)^{-1} b(w(s,y),y) ds -	\\
	\dis \quad \int_t^\infty U_0(t,y)(\I-P^0_+(y))U_0(s,y)^{-1} b(w(s,y),y)ds
\end{array}\]
where $U_0(t,y)$ is the fundamental matrix of $\dot x = f_{N,x}(w_+(y),y)x$. Then, using $\sup_{t\ge 0}|w(t,y)|<\rho$, with $\rho$ as in Lemma \ref{Onu+}, we get:
\[
	|w(t,y)|\le Ke^{-\de_0t}\rho + \frac{1}{2}KL\rho \int_0^\infty e^{-\de_0|t-s|} |w(s,y)|ds .
\]
Let $\rho>0$ be such that it also satisfies $KL\rho<\de_0$ and take
\[\begin{array}{l}
	\th=KL\rho\de_0^{-1}		\\	\be = \de_0(1-\theta)^{\frac{1}{2}}.
\end{array}\]
According to \cite[Lemma 1, p.28]{Cop} we get $|w(t,y)| \le 2\de_0L^{-1}(1-(1-\th)^{1/2})e^{-\be t}$ that is
\[
	\sup_{t\ge 0}|u^+_N(t+T,y)-w_+(y)|e^{\be t} \le 2\de_0L^{-1}(1-(1-\th)^{1/2})\le \de_0\th L^{-1} = K\rho
\]
and the bound is independent of $y\in\R^m$.
\medskip

Now we consider $w_y(t,y)=u^+_{N,y}(t+T,y)-w'_+(y)$.

Differentiating $\dot u^+_N(t+T,y) = f_N(u^+_N(t+T,y) ,y)$ with respect to $y$ and using $f_N(w_+(y),y)=0$ we see that
\[\begin{array}{l}
	\dot w_y(t,y) 	\\
	\qquad = f_{N,x}(u_N^+(t+T,y),y) [w_y(t,y)+w'_+(y)] +  f_{N,y}(u_N^+(t+T,y),y)	\\
	\qquad = f_{N,x}(w_+(y),y)w_y(t,y) 	\\
	\qquad\qquad + [f_{N,x}(u_N^+(t+T,y),y) -f_{N,x}(w_+(y),y)]w_y(t,y)	\\
	\qquad\qquad  + [f_{N,x}(u_N^+(t+T,y),y)-f_{N,x}(w_+(y),y)]w'_+(y) 	\\
	\qquad\qquad  + f_{N,y}(u_N^+(t+T,y),y)  - f_{N,y}(w_+(y),y)
\end{array}\]
that is $w_y(t,y)$ is a solution of:
\[
	\dot x = f_{N,x}(w_+(y),y)x + b(t,x,y)
\]
where
\[\begin{array}{l}
	b(t,x,y) := [f_{N,x}(u_N^+(t+T,y),y) -f_{N,x}(w_+(y),y)]x			\\
	\qquad  + [f_{N,x}(u_N^+(t+T,y),y)-f_{N,x}(w_+(y),y)]w'_+(y) 	\\
	\qquad  + f_{N,y}(u_N^+(t+T,y),y)  - f_{N,y}(w^+(y),y)
\end{array}\]
bounded on $t\ge 0$. Let $L$ be a Lipschitz constant for $f_{N,x}(x,y)$ and $f_{N,y}(x,y)$. Now we assume that $3KL\rho<\de_0$.
We have
\[
	|b(t,x,y)|\le L\rho |x|  + KL\rho (|w'_+(y)|+1)e^{-\be t}
\]
and
\[\begin{array}{l}
	\dis w_y(t,y) = U_0(t,y)P^0_+(y)w_y(0,y) + \int_0^t U_0(t,y)P^0_+(y)U_0(s,y)^{-1} b(s,w_y(s,y),y) ds		\\
	\dis\qquad - \int_t^\infty U_0(t,y)(\I-P^0_+(y))U_0(s,y)^{-1} b(s,w_y(s,y),y) ds.
\end{array}\]
Then
\[\begin{array}{l}
	\dis |w_y(t,y)| \le Ke^{-\de_0 t}|w_y(0,y)| + 		\\
	\qquad \qquad \dis \int_0^\infty KL\rho e^{-\de_0(t-s)} (|w_y(s,y)| + K(|w'_+(y)|+1)e^{-\be s})  ds	\\
	\dis \le Ke^{-\de_0 t}|w_y(0,y)| + 2K^2L\rho(|w'_+(y)|+1)\de(\de^2-\be^2)e^{-\be t} 	\\
	\qquad \qquad \dis + KL\rho\int_0^\infty e^{-\de_0|t-s|)}|w_y(s,y)|ds	\\
	\le \dis Ce^{-\be t} + KL\rho\int_0^\infty e^{-\be|t-s|} |w_y(s,\bar y)| ds	
\end{array}\]
where
\[
	C = K|w_y(0,y)| + 2\de K^2L\rho(\de^2-\be^2)^{-1}(|w_+'(y)| +1).
\]
We claim that $2KL\rho\be^{-1} < 1$. First, as $1-x<(1-x)^{\frac{1}{2}}$, for $0<x<1$, we have
\[
	\de_0-KL\rho < \beta.
\]
Then, since $\frac{x}{1-x}$ is increasing on $0\le x\le 1$ and $KL\rho\de^{-1}<\frac{1}{3}$:
\[\begin{array}{l}
	2KL\rho\be^{-1} < \frac{2KL\rho}{\de_0 - KL\rho} = \frac{2KL\rho\de_0^{-1}}{1 - KL\rho\de_0^{-1}}\le
	2 \frac{\frac{1}{3}}{1-\frac{1}{3}} = 1.
\end{array}\]
Then, applying again \cite[Lemma 1, p.28]{Cop} we get
\[
	|w_y(t,y)|\le \tilde\mu C e^{-\tilde\be t}
\]
for $t\ge 0$, where
\[\begin{array}{l}
	\tilde\mu = \tilde\th^{-1}[1-(1-2\tilde\th)^\frac{1}{2}]	\\
	\tilde\th = KL\rho\be^{-1}	\\	\tilde\be = \be(1-2\tilde\th)^{-1}.
\end{array}\]
Hence $|w_y(t,y)|e^{\tilde\be t}$ is bounded on $\R_+$, uniformly with respect to $y$ since
\[
	w(0,y)=u_N(T,y)-w_+(y) = \int_{t_N(y)}^T f_N(u(s,y),y)ds
\]
and $w_+(y)$ are bounded together with their derivatives, uniformly with respect to $y\in\R^m$. The proof for the higher order derivatives follows the same line.
\end{remark}


\begin{thebibliography}{99}
\bibitem{BF-1} F. Battelli and M. Fe\v ckan, \emph{Global centre manifolds in singular systems}, NoDEA., vol. 3, (1996), 19-34

\bibitem{BF} F. Battelli and M. Fe\v ckan, \emph{Homoclinic trajectories in discontinuous systems}, J. Dyn. Diff. Eqs., vol. 20, n. 2, (2008), 337-376

\bibitem{BF0} F. Battelli and M. Fe\v ckan, \emph{ Chaos in forced impact systems}, Disc. Cont. Dyn. Syst S. \textbf{6} (2013), 861-890.

\bibitem{BF2} F. Battelli and M. Fe\v ckan, \emph{Periodic solutions in slowly varying discontinuous differential equations: the generic sase}, Mathematics, {\bf 9} (19), (2021), article n. 2449.

\bibitem{BF3} F. Battelli and M. Fe\v ckan, \emph{Periodic solutions in slowly varying discontinuous differential equations: a non-generic case}, to appear in J. Dyn. Diff. Eqs.

\bibitem{B} F. Battelli, \emph{Heteroclinic orbits in singular systems: a unifying approach}, J. Dyn. Diff. Eqs. \textbf{6} n. 1 (1994), 147-173.

\bibitem{BL} F. Battelli  and C. Lazzari, \emph{Heteroclinic orbits in systems with slowly varying coefficients}, J. Diff. Eqs. \textbf{105} (1993), 1-29.

\bibitem{BP} F. Battelli and K. J. Palmer, \emph{Chaos in the Duffing equation},  J. Diff. Eqs. \textbf{101} (1993), 276-301.

\bibitem{BP1} F. Battelli and K. J. Palmer, \emph{Singular perturbations, transversality and Sil'nikov saddle-focus homoclinic orbits},  J. Dyn. Diff. Eqs. \textbf{15} (2004), 357-425.

\bibitem{BP2} F. Battelli and K. J. Palmer, \emph{Heteroclinic connections in singularly perturbed systems},  DCDS Ser. B, Special Issue on Nonautonomous Dynamics, \textbf{3-4} (2008), 431-461.

\bibitem{Cop}  W. A. Coppel, \emph{Dichotomies and Stability Theory}, Lecture Notes in Mathematics, \textbf{629} Springer-Verlag, Berlin, Heidelberg, New York, 1978.

\bibitem{Fen} N. Fenichel, \emph{Geometric singular perturbation theory for ordinary differential equations}, J. Differential Equations \textbf{31} (1979), 53-98.

\bibitem{GR} I. S. Gradshteyn and I. M. Ryzhik, \emph{Table of Integrals, Series, and Products 7th ed.}, Elsevier-Academic Press, New York  (2007), 1611-1632.

\bibitem{Ha} G. Haller, \emph{Chaos Near Resonance}, Applied Mathematical Sciences \textbf{138}, Springer Science+Business Media, New York, 1999.

\bibitem{Hop1} F. C. Hoppensteadt, \emph{Singular perturbations on the infinite interval}, Trans. Am. Math. Soc. {\bf 123} (1966), 521-535

\bibitem{Hop2} F. C. Hoppensteadt, \emph{Properties of solutions of ordinary differential equations with small parameters}, Comm. Pure Appl. Math. {\bf 24} (1971), 807-840

\bibitem{KL} H. Kurland and M. Levi, \emph{Transversal heteroclinic intersections in slowly varying systems}, in \emph{Dynamical Systems Approaches to Nonlinear Problems in Systems and Circuits} F.A. Salam and M. Levi eds, SIAM (1988), 29-38.

\bibitem{Kov} G. Kova\v ci\v c, \emph{Singular perturbation theory for homoclinic orbits in a class of near-integrable Hamiltonian systemss},
SIAM J. Math. Anal. \textbf{26} (1995), 1611-1643.

\bibitem{P1} K. J. Palmer, \emph{Shadowing in Dynamical systems, Theory and  Applications}, Springer-Science + Business Media, B.V., Dordrecht  (2000), originally publ. by Kluwer Academic Pub. in 2000.

\bibitem{P2} K. J. Palmer, \emph{Transverse heteroclinic orbits and Cherry's example of a noninegrable Hamiltonian system}, J. Diff. Equations, \textbf{65} (1986), 321-360.

\bibitem{SP} A. M. Samoilenko and N. A. Perestyuk, \emph{Impulsive Differential Equations}, {\em translated from the Russian by Y. Chapovsky} World Scientific Publ. Co., Singapore, 1995.

\bibitem{Sa} K. Sakamoto, \emph{Invariant manifolds in singular perturbation problems for ordinary differential equations}, Proc. Roy, Soc. Edinburgh Sect A, \textbf{116} (1990), 45-78.

\bibitem{Sz} P. Szmolyan, \emph{Transversal heteroclinic and homoclinic orbits in singular perturbation problems}, J. Differential Equations \textbf{92} (1991), 252-281.

\bibitem{T} A.N. Tikhonov, \emph{Systems of differential equations containing small parameters multiplying some of the derivatives}, Mat. Sb.
\textbf{31} (1952), 575-586.

\bibitem{VB} A. B. Vasileva, V. F. Butuzov, \emph{Asymptotic Expansion of Solutions of Singularly Perturbed Equations}, Nauka 1973,
in Russian.

\bibitem{VBK} A. B. Vasileva, V. F. Butuzov, L. V. Kalachev, \emph{The Boundary Function Method for Singular Perturbation Problems}, Studies in Applied and Numerical Mathematics, SIAM, Philadelphia, 1995.

\bibitem{WH}  S. Wiggins and Ph. Holmes, \emph{Periodic orbits in slowly varying oscillators}, SIAM J. Math. Anal., \textbf{18} (1987) 592-611.

\bibitem{WH2} S. Wiggins and Ph. Holmes, \emph{Homoclinic orbits in slowly varying oscillators}, SIAM J. Math. Anal. \textbf{18} (1987), 612-629.
Erratum: SIAM J. Math. Anal. \textbf{19} (1988), 1254-1255.

\end{thebibliography}
\end{document}